\documentclass{amsart}[10pt]

\theoremstyle{plain}
\newtheorem{theorem}{Theorem}[section]

\newtheorem{corollary}[theorem]{Corollary}

\newtheorem{lemma}[theorem]{Lemma}

\newtheorem{proposition}[theorem]{Proposition}

\theoremstyle{definition}
\newtheorem{example}[theorem]{Example}

\newtheorem{conjecture}[theorem]{Conjecture}

\newtheorem{remark}[theorem]{Remark}

\newtheorem{notation}[theorem]{Notation}

\newtheorem{definition}[theorem]{Definition}

\usepackage[T1]{fontenc}
\usepackage[utf8]{inputenc}
\usepackage{geometry}
\usepackage{graphicx}
\usepackage{scalerel}
\usepackage{amssymb}
\usepackage{amsfonts}
\usepackage{color}
\usepackage{mathtools}
\usepackage{tikz} 
\usepackage{ytableau}
\usepackage{pgfplots}
\usepackage{hyperref}
\usepackage{xcolor}
\usepackage{tikz-cd}
\usepackage{soul}
\usepackage{epsfig}
\usepackage{hyperref}
\usepackage[makeroom]{cancel}
\usepackage{mathtools}
\usepackage{ytableau}
\usepackage{tikz-qtree}
\usepackage{comment} 
\usepackage{comments}

\newcommand{\Hilb}{\mathrm{Hilb}}

\newcommand{\Sym}{\mathrm{Sym}}

\newcommand{\HHH}{\mathrm{HHH}}

\newcommand*\circled[1]{\tikz[baseline=(char.base)]{
\node[shape=circle,draw,inner sep=0.8pt] (char) {#1};}}

\tikzset{
    cell/.style={
        anchor=south west,
        draw,
        minimum size=1cm,
    },
}

\title{Hilbert scheme of points on non-reduced nodal curves}
\author{Yuze Luan}
\begin{document}

\maketitle

\tableofcontents

\section{Introduction}
\subsection{The ORS conjecture}
Given singular curve $C \subset \mathbb{C}^2$ with a planar singularity point at the origin $(0,0)$, one can associate an algebraic link $L$ to it by considering the intersection of $C$ with a small $3$-sphere $S^3$ in $\mathbb{C}^2 \cong \mathbb{R}^4$ centered at $(0,0)$. 

For the singular algebraic curve $C$, we can define a moduli space of ideals in the coordinate ring of $C$ that provides information about the singularity,  \textbf{the Hilbert scheme of points on $\mathbf{C}$} supported at $(0,0)$, denoted by $\Hilb^n(C,0)$.

On the other hand, with tools from topology, for the algebraic link $L$, one can define \textbf{the (colored) Khovanov-Rozansky homology of the link $\mathbf{L}$}, denoted by $\HHH(L)$, which is a link invariant.

A beautiful conjecture by Oblomkov, Rasmussen and Shende \cite{ORS} relates the two seemingly unrelated objects from algebraic geometry and topology, stating that the homology of Hilbert schemes of points on such singular curve $C$ is the Khovanov-Rozansky knot homology of the corresponding links $L$.

\begin{conjecture}
\label{conjORS}
$$\HHH^{a=0}(L)=\bigoplus_{k,n=0}^{\infty}\mathrm{H}^k(\Hilb^n(C,0)).$$
\end{conjecture}

\subsection{Verified cases of the ORS conjecture.\\}\label{sec1.2}

%By a remarkable result of Maulik \cite{Mau}, the generating function for the Euler characteristics of the Hilbert schemes matches the HOMFLY-PT polynomial, which is the de-categorified version of the ORS conjecture.

The ORS conjecture has been verified for some reduced singular curves, and some non-reduced smooth curves. For a non-reduced curve, we are referring to the smooth or singular properties of its underlying reduced curve. 

When the curve $C$ is non-reduced, with components $C_i$ each of multiplicity 
$k_i$, the intersection of $S^3$ with the non-reduced curve $C$ is a link $L$, where each link component of $L$ is colored by a row of Young diagram of $k_i$ boxes, see \cite{HM} for the precise definition. When the curve is reduced, the coloring is trivial for the link homology, as defined in \cite{GKS}.

\begin{enumerate}
\item When $p,q$ are coprime for the curve $x^p-y^q=0$, the link $L$ is a trivially colored torus knot. The Khovanov-Rozansky homology is computed recursively by Hogancamp and Mellit \cite{HM}, and $\Hilb^k (C, 0)$ has an explicit paving by affine cells as computed in \cite{ORS}, which yields an explicit formula for the Poincar\'e polynomial of the homology. Both sides of the Conjecture \ref{conjORS} were computed and compared.

\item In the case when the curve $C$ is $\{y^k=0\}$, non-reduced and smooth,  the homology of the Hilbert scheme is given in \cite{Bej}. There is an explicit Bialynicki-Birula cell decomposition of $\Hilb^n(\{y^k=0\},0)$, so the homology can be computed, and compared to the $k$-row-colored knot homology of the unknot, computed in \cite{HM}. See more details in Example \ref{egy^k}.

\end{enumerate}

A remarkable result by Maulik \cite{Mau} is the decategorified form of the ORS conjecture, that the generating function for the Euler characteristics of the Hilbert schemes matches the HOMFLY-PT polynomial. Maulik's setup is for reduced planar curves, with multiplicity data encoded combinatorially via partitions assigned to the components, while we study the geometry of punctual Hilbert schemes on non-reduced curves directly. 

In general, the conjecture is wide open. We provide in this paper the first verification of the ORS conjecture for a non-reduced and singular curve.

\subsection{Affine paving of the Hilbert scheme}
The key to computing the left hand side of the conjecture, the homology of the Hilbert scheme, usually boils down to finding an affine paving of the Hilbert scheme.  We list a few known affine pavings, and in general the problem of finding affine pavings of the Hilbert scheme is wide open.

\begin{enumerate}
\item An affine paving of the Hilbert scheme of points on the affine plane $\mathbb{C}^2$ was found by Ellingsrud and Strømmer, see \cite{ES}.

\item As mentioned in \ref{sec1.2}, the Hilbert scheme of points on the irreducible and singular curve, $x^v-y^v=0$ for $p,q$ coprime, has an affine paving, and its homology was computed in many sources other than the aforementioned \cite{ORS}, for example, Lusztig and Smelt \cite{LS}.

\item The Hilbert scheme of points at a generic curve singularity with Puiseaux exponents $(nd,md,md+1)$ has an affine paving found by Gorsky, Mazin, Oblomkov, see \cite{GMO}.
\end{enumerate}

\subsection{Methods of this paper}

Classical tools for finding the affine pavings of the Hilbert scheme include the Bialynicki-Birula cell decomposition, which applies nicely only to smooth spaces, and stratification by semigroups obtained from orders of functions in the coordinate ring of the curve, but this theory is only applicable for reduced curves. 
The problems of non-reduced curves remains open to be explored. A key difficulty lies in the lack of effective tools to control the structure of ideals in the Hilbert scheme on a non-reduced curve.

In this paper, we introduce a new method to study punctual Hilbert schemes of non-reduced plane curve singularities. The idea is to construct filtrations of ideals, such that they slice the ideals into layers, compatible with the non-reduced structure of the curve.
And together we define a non-reduced valuation that allows us to systematically track the leading terms of the generators of each layer and relationships between them. 

Our approach leads to affine pavings and torus-type stratifications of the punctual Hilbert schemes on a non-reduced nodal curve with nodal singularity at $(0,0)$, $x^uy^v=0$ for positive integers $u,v$. Finally, the degree of the freedom of the generators are computed inductively and presented as certain statistics of a new type of partition we define, called the weak diagonal partition. 

Although used in the specific setting of $x^u y^v = 0$, the filtration and valuation techniques are expected to apply more broadly to other non-reduced singular curves.

\subsection{Results of this paper}

\subsubsection{Affine paving of $\Hilb^n(\{xy^v=0\},0)$ and the ORS conjecture.\\}

When the curve is $xy^v=0$, the corresponding link is the Hopf link where one strand is colored by a row of $v$ boxes, and the other strand is trivially colored. We have the following theorem that classifies all affine strata and computes the Poincar\'e polynomials of the Hilbert scheme $\Hilb^n(\{xy^v=0\},0)$.

Let $P$ denote the Poincar\'e polynomial of the knot homology, and $\chi$ denote the Poincar\'e polynomial of the Hilbert scheme, $Q$ record the number of points, and $T$ record the homological degree.

\begin{theorem}
The Hilbert scheme $\Hilb^{n}(\{xy^v=0\},0)$ has an affine stratification, and each stratum is indexed by a partition $(i_0, ..., i_{N})$ such that $i_k \leq 1$ for all $k \geq v$. The dimension of the stratum is the number of boxes from the $1$st row to the $v$th row, i.e. $\sum_{k=1}^{v}i_k$.

The Poincar\'e polynomial of the Hilbert scheme $\Hilb^{n}(\{xy^v=0\},0)$ is $$\prod_{i=1}^{v}\frac{1}{1-Q^{i}T^{2(i-1)}}(1+\frac{Q^{v+1}T^{2v}}{1-Q}).$$

\end{theorem}

The Poincar\'e polynomial of the Hopf link, where one strand is colored by a row of $v$ boxes and the other strand is trivially colored, is computed by Mellit, Hogancamp \cite{HM} and Conners \cite{Con}. So the colored ORS conjecture is verified by direct comparison of the Poincar\'e polynomials.

\begin{corollary}[Colored ORS conjecture for the Hopf link]
The Poincar\'e polynomial of the Hilbert scheme of $n$ points on $\{xy^v=0\}$ is the Poincar\'e polynomial of the colored homology of the Hopf link, where one strand is colored by a row of $v$ boxes, and the other strand is trivially colored, up to a factor $\prod_{i=1}^{v-1}T^{i}$ and the change of variables $T \rightarrow (QT^2)^{-1}$.

\begin{align*}
&\sum_{n=0}^{\infty}\chi(\Hilb^n(\{xy^v=0\},0))Q^n=P(T(2,2),(\Sym^{(v)}, \Sym^{(1)})) .
\end{align*}
 
\end{corollary}

\subsubsection{Affine paving of $\Hilb^{n}(\{x^2y^v=0\},0)$.\\}
When the curve is $x^2y^v=0$, the corresponding link is the Hopf link where one strand is colored by a row of $v$ boxes, and the other strand is colored by a row of $2$ boxes. We have the following theorem. 

\begin{theorem}
The Hilbert scheme $\Hilb^{n}(\{x^2y^v=0\},0)$ has an affine stratification, and each stratum is indexed by a partition $(i_0, ..., i_{N})$ such that $i_k \leq 2$ for all $k \geq v$. Denote the index of $i_k$ as $M$ where $i_k=1$ for $k=M,...,N$ and $i_k > 1$ otherwise, if $M$ exists. The dimension of the stratum is the number of boxes from the $1$st row to the $v$th row, plus the number of boxes between the $\lfloor \frac{M+N}{2}\rfloor+1$th row to the $N$th row, $$\sum_{k=1}^{v}i_k + \sum_{k=\lfloor \frac{M+N}{2}\rfloor+1}^{N}i_k.$$ If $M$ does not exist, then the dimension is only the first summand $\sum_{k=1}^{v}i_k$.

The Poincar\'e polynomial of the Hilbert scheme, $\Hilb^{n}(\{x^2y^v=0\},0)$ is 
\begin{align*}
&\sum_{n=0}^{\infty}\chi(\Hilb^n(\{x^2y^v=0\},0))Q^n = \\ 
&\frac{(1-Q)(1-Q^2T^2)+(1-Q)(1+QT^2)Q^{v+1}T^{2v} + Q^{2(v+1)}T^{4v}}{(1-Q)(1-Q^2T^2)\prod_{i=1}^{v}(1-Q^{i}T^{2(i-1)})}.
\end{align*}
\end{theorem}

Because the Poincar\'e polynomial for the Hopf link is not known in this case, we cannot yet verify the ORS conjecture. 

\subsubsection{Torus-like Stratification for $u=v, v-1$ or $v-2$.\\}

When $u=v, v-1$ or $v-2$ for the curve $x^uy^v=0$, we give a stratification of $\Hilb^n(\{x^uy^v=0\},0)$ for any arbitrary positive integers $u,v$. Each stratum is indexed by combinatorial data, see Definition \ref{wddp} as a weak diagonal partition, that records the types of the generators, and degrees of freedom of the generators. We prove the following theorem.

\begin{theorem}
When $u=v,v-1$ or $v-2$, each stratum of $\Hilb^n(\{x^uy^v=0\},0)$ is isomorphic to $(\mathbb{A}-pt)^{m_1} \times \mathbb{A}^{m_2}$, an algebraic torus times an affine space, where $m_1,m_2 \geq 0$ are some non-negative integers recording the degrees of freedom of the generators of the ideals in each stratum. 
\end{theorem}

The degrees of freedom $m_1, m_2$ can be read off directly from the weak diagonal partition indexing the strata. See Theorem \ref{thmdiag} and Proposition \ref{familyofideals} for detailed combinatorial descriptions.

\subsubsection{Virtual Poincar\'e polynomial of the Hilbert scheme $\Hilb^n(\{x^uy^v=0\},0)$ for $u=v,v-1, v-2$.\\}

Following the stratification, we count the number of strata of each isomorphism type, and give formulas of the virtual Poincar\'e polynomials of $\Hilb^n(\{x^uy^v=0\},0)$ for $u=v,v-1,v-2$.

Let $\chi$ denote the virtual Poincar\'e polynomial of an algebraic variety (see \ref{vpp} for the definition). Let $Q$ be the variable recording the number of points of the Hilbert scheme, and $T$ recording the homological degree, we have the following theorem. If we write $\mathbb{L}$ instead of $T^2$, we obtain a motivic version: a formula for the Hilbert-zeta function .

\begin{theorem}

The virtual Poincar\'e polynomial of the Hilbert scheme, $\Hilb^n(\{x^vy^v=0\},0)$ is 

\begin{align*}
\sum_{n=0}^{\infty}\chi(\Hilb^n(\{x^vy^v=0\},0))Q^n = \sum_{k=1}^v \frac{F(k)}{\prod_{i=1}^{k}(1 - Q^iT^{2(i-1)})^2},
\end{align*}
where $F(k)$ is the generating function of the number of all the  strata whose indexing partitions are minimal weak diagonal partitions with exactly $k$ layers, given as follows
$$
F(k)=\begin{cases}
Q^2T^2-Q+1, & \text{if } k=1, \\

\mathbf{1}^T \cdot V_k, & \text{if } k =2,...,v,
\end{cases}
$$  

where $\mathbf{1} =\begin{bmatrix}1 & 1 \end{bmatrix}$ is a row vector of ones, and $\mathbf{1}^T \cdot V_k$ is the sum of the two entries of the vector $V_k$.

The vectors $V_k$ are given as $$V_k= \prod_{s=2}^{k} M_s \cdot V_1.$$

The initial vector $V_1$ is $$V_1=
\begin{bmatrix}
Q^{2k}T^{4(k-1)}(T^2-1) \\ Q^kT^{2(k-1)}
\end{bmatrix}.$$

And the matrix $M_s$ for $s=2, ..., k$ is 
$$M_s =
\begin{bmatrix}
Q^{2(k-s+1)}T^{4(k-s)}(T^2-1)+(1-Q^{k-s+1}T^{2(k-s)})^2 & Q^{k-s+1}T^{2(k-s)}(T^2-1)  \\ 
Q^{k-s+1}T^{2(k-s)} & Q^{2(k-s+1)}T^{4(k-s)}T^2
\end{bmatrix}.$$ 

\end{theorem}
\begin{theorem}
The virtual Poincar\'e polynomial of $\Hilb^n(\{x^{v-1}y^v=0\},0)$ is 
\begin{align*}
\sum_{n=0}^{\infty}\chi(\Hilb^n(\{x^{v-1}y^v=0\},0))Q^n =\sum_{k=1}^v \frac{F_{v-1}(k)}{\prod_{i=1}^{k}(1 - Q^iT^{2(i-1)})^2}.
\end{align*}

The virtual Poincar\'e polynomial of $\Hilb^n(\{x^{v-2}y^v=0\},0)$ is \begin{align*}
\sum_{n=0}^{\infty}\chi(\Hilb^n(\{x^{v-2}y^v=0\},0))Q^n =\sum_{k=1}^v \frac{F_{v-2}(k)}{\prod_{i=1}^{k}(1 - Q^iT^{2(i-1)})^2}.
\end{align*}

The polynomials $F_{v-1}(k)$, $F_{v-2}(k)$ are given as
$$
F_{v-1}(k)=\begin{cases}
F(k)
 & \text{if } k \leq v-1, \\

\mathbf{1}^T \cdot (\prod_{s=2}^{k} M_s \cdot W_1),  & \text{if } k = v,
\end{cases}
$$

$$
F_{v-2}(k)=\begin{cases}
F(k)
 & \text{if } k \leq v-2, \\

\mathbf{1}^T \cdot (\prod_{s=3}^{k} M_s \cdot D_1),  & \text{if } k = v-1,\\

\mathbf{1}^T \cdot (\prod_{s=2}^{k} M_s \cdot D_2),  & \text{if } k = v.
\end{cases}
$$

And the vectors $$W_1=
\begin{bmatrix}
0 \\ Q^kT^{k-1}(1-Q^kT^{2(k-1)})
\end{bmatrix}.$$

$$D_1=
\begin{bmatrix}
Q^{2(k-1)}T^{4(k-2)}(T^2-1)(1-Q^{k-1}T^{2(k-2)}) \\
(Q^{k-1}T^{2(k-2)}+Q^{2(k-1)}T^{4(k-2)})(1-Q^{k-1}T^{2(k-2)}) 
\end{bmatrix}, $$

$$D_2=
\begin{bmatrix}
Q^{k+2(k-1)}T^{2(k-1+2(k-2))}(T^2-1)(1-Q^kT^{2(k-1)})(1-Q^{k-1}T^{2(k-2)}) \\ Q^{k+2(k-1)}T^{2(k-1+2(k-2))}(1-Q^kT^{2(k-1)})(1-Q^{k-1}T^{2(k-2)})
\end{bmatrix}.$$

\end{theorem}

The Appendix section contains a Macaulay2 code that computes all of the virtual Poincar\'e polynomials for $u=v,v-1,v-2$ and any arbitrary $v$ for $u, v \geq 1$.

\subsubsection{New stratification of the punctual Hilbert scheme of points on $\mathbb{C}^2$.\\}

When $u=v$ are both large enough such that $2v \geq n$, $\Hilb^n(\{x^vy^v=0\},0)$ is isomorphic to the punctual Hilbert scheme of $n$ points on the affine plane $\mathbb{C}^2$, $\Hilb^n(\mathbb{C}^2,0)$. So we also obtain a new stratification of $\Hilb^n(\mathbb{C}^2,0)$, not coming from the classical Bialynicki-Birula stratification, along with a new formula for its virtual Poincar\'e polynomial. 

\begin{theorem}

The Hilbert scheme $\Hilb^n(\mathbb{C}^2,0)$ has a stratification where each stratum is indexed by a weak diagonal partition of $n$ boxes, denoted as $\Sigma_{([\circled{s}_k, (i_k,j_k)])}$ where $ 0 \leq k \leq \lceil \frac{n}{2} \rceil -1$, and each stratum is an algebraic torus times an affine space,

$$\Sigma_{([\circled{s}_k, (i_k,j_k)])} \cong (\mathbb{A}-pt)^{m_1} \times \mathbb{A}^{m_2},$$ for some non-negative integers $m_1,m_2$. 

The Hilbert scheme $\Hilb^n(\mathbb{C}^2,0)$ has the following virtual Poincar\'e polynomial
\begin{align*}
&\sum_{n=0}^{\infty}\chi(\Hilb^n(\mathbb{C}^2,0))Q^n  = \sum_{k=1}^\infty \frac{F(k)}{\prod_{i=1}^{k}(1 - Q^iT^{2(i-1)})^2}.
\end{align*}

\end{theorem}

As a corollary, we have the following combinatorial identity, which specializes to the classical Durfee square formula for generating function of integer partitions (\ref{durfeeclassical}) by setting $T=1$. 

\begin{corollary}
We have the combinatorial identity coming from the virtual Poincar\'e polynomial of $\Hilb^n(\mathbb{C}^2,0)$. 

\begin{align*}
\sum_{k=1}^\infty \frac{F(k)}{\prod_{i=1}^{k}(1 - Q^iT^{2(i-1)})^2}= \frac{1}{\prod_{i=1}^{\infty}(1 - Q^iT^{2(i-1)})}.
\end{align*}
\end{corollary}

\subsubsection{Predictions on the colored knot homology}
Because the colored knot homology is only known for $u=1$ and $v$ arbitrary, we are only able to verify the colored ORS conjecture in very limited cases. But our results provides predictions of the colored knot homology of the $(\Sym^{(u)}, \Sym^{(v)})$-colored knot homology of the Hopf link, when $u=2,v-2,v-1$ or $v$, for any arbitrary $v$ such that $u, v \geq 1$.

\textbf{Acknowledgement}
The author thanks Joachim Jelisiejew, Eugene Gorsky, Luke Conners, Joshua Turner, Alexei Oblomkov, Laurent Evain, Anton Mellit, Michel Granger, and Mircea Mustață for helpful discussions.

\section{Khovanov-Rozansky knot homology of an algebraic link}

The Khovanov-Rozansky link homology is a link invariant and the reader should see \cite{GKS} for the introduction, the definitions and properties, and see \cite{HM},\cite{Con} for the definition of the row-colored Khovanov-Rozansky homology.

In this paper, we only introduce the necessary setup and the recursive formulas for computing the row-colored Khovanov-Rozansky link homology of a torus link.

We restate the theorems that computes the $\Sym^{(v)}$ and $ \Sym^{(1)}$ colored link homology of the Hopf link $T(2,2)$, given by Hogancamp, Mellit \cite{HM} and Conners \cite{Con}. 

\begin{theorem}\cite{HM}\cite{Con}
\label{krhomology}
Let $t \in \{0, 1\}^{m+l}$ and $w \in \{0, 1\}^{n+l}$ be binary sequences with $|t|= |w|= l.$ Here $|t| = t_1 +...+ t_{m+l}$ is the number of $1$'s. Let $\mathbf{p}(t, w) \in \mathbb{N}[Q, T, T^{-1}, a, (1-Q)^{-1}]$ denote the unique family of polynomials, indexed by such pairs $(t,w)$ of binary sequences, satisfying
\begin{enumerate}

\item $\mathbf{p}(\emptyset,0)= (\frac{1+a}{1-Q})^n$ and $\mathbf{p}(0, \emptyset)=(\frac{1+a}{1-Q})^m$. 
\item $\mathbf{p}(t1, w1) = (T^l + a)\mathbf{p}(t, w),$ where $|t| = |w| = l.$
\item $\mathbf{p}(t0, w0) = T^{-l}\mathbf{p}(1t, 1w) + QT^{-l}\mathbf{p}(0t, 0w),$ where $|t| = |w| = l.$
\item $\mathbf{p}(t0, w1) = \mathbf{p}(t, 1w).$
\item $\mathbf{p}(t1, w0) = \mathbf{p}(1t, w).$
\end{enumerate}

The Poincar\'e polynomial of the $\Sym^{(v)}$ and $\Sym^{(1)}$ colored link homology of the torus link $T(m_1,m_2)$, denoted by $P(T(m_1,m_2),(\Sym^{(v)}, \Sym^{(1)}))$, is the following, where $d $ is the greatest common divisor of $m_1, m_2$.
\begin{align*}
&P(T(m_1,m_2),(\Sym^{v}, \Sym^{1}))\\
&=\prod_{i=1}^{v} \frac{1}{1-QT^{1-i}}(\mathbf{p}(1^{v} 0^{(m_1d^{(-1)} (d - 1) + v (m_1d^{(-1)} - 1))},  1^{q} 0^{(m_2d^{(-1)} (d - 1) + v (m_2d^{(-1)} - 1)})).\\    
\end{align*}
\end{theorem}

Now using the recursion of $\mathbf{p}$ above, we compute the Poincar\'e polynomial of the Hopf link (setting $a=0$ in the recursion).
\begin{corollary}
\label{Poincar\'ehofp}
When $m_1=m_2=d=2$, up to a factor $\prod_{i=1}^{v-1}T^{i}$, the Poincar\'e polynomial of the homology of the Hopf link is 
\begin{align*}
&P(T(2,2),(\Sym^{v}, \Sym^{1}),a=0)=(1+\frac{QT^{-v}}{1-Q})\prod_{i=1}^{v} \frac{1}{1-QT^{1-i}}. 
\end{align*}
\end{corollary}

\begin{proof}
Plugging in $m_1=m_2=d=2$ to Theorem \ref{krhomology},
\begin{align*}
&P(T(2,2),(\Sym^{(v)}, \Sym^{(1)}),a=0)\\
&=\prod_{i=1}^{v} \frac{1}{1-QT^{1-i}} \mathbf{p}(1^v0^{(1 \times (2-1)+v\times (1-1))},1^v10^{(1 \times (2-1)+v\times (1-1))})\\
&=\prod_{i=1}^{v} \frac{1}{1-QT^{1-i}} \mathbf{p}(1^v0,1^v0).
\end{align*}

Using the recursion of $\mathbf{p}$, setting $a=0$,
\begin{align*}
&\mathbf{p}(1^v0,1^v0)\\
&=T^{-v}\mathbf{p}(1^{v+1},1^{v+1})+QT^{-v}\mathbf{p}(01^v, 01^v)\\
&=\prod_{i=1}^{v-1}T^{i}+(QT^{-v}\prod_{i=1}^{v-1}T^{i})\mathbf{p}(0, 0)\\
&=(\prod_{i=1}^{v-1}T^{i})(1+\frac{QT^{-v}}{1-Q}).
\end{align*}

Up to the factor $\prod_{i=1}^{v-1}T^{i}$, the Poincar\'e polynomial is equal to 
\begin{align*}
&P(T(2,2),(\Sym^v,\Sym^1),a=0)\\
&=(1+\frac{QT^{-v}}{1-Q})\prod_{i=1}^{v} \frac{1}{1-QT^{1-i}}. 
\end{align*}
\end{proof}

\section{Hilbert scheme of points and virtual Poincar\'e polynomial}

Let $g(x,y) \in \mathbb{C}[x,y]$ be a polynomial of two variables, and let $C$ denote the curve defined by the polynomial $C:=\{g(x,y)=0\}$.

\begin{definition}
The punctual Hilbert scheme of $n$ points on the curve $C$, denoted by $\Hilb^n(C,0)$, is the moduli space of ideals $I \in \mathbb{C}[\![x,y]\!]$ such that $\frac{\mathbb{C}[\![x,y]\!]}{I}$ has dimension $n$ as a vector space over $\mathbb{C}$, the vanishing locus of the ideal $V(I)$ is a point $(0,0)$, and that $g(x,y)$ is an element in the ideal $I$.
\end{definition}

\subsection{Reduced nodal singularity}
The homology of the Hilbert scheme of $n$ points on a reduced nodal curve $\{xy=0\}$ was computed in \cite{GKS} Example 6.4. The condition that the monomial $xy$ is a generator of all ideals $I \in \Hilb^n(\{xy=0\},0)$ is very strong and ruled out  almost all possibilities of the generators of the ideal $I$ except for a few, and the classification of all types of ideals follows immediately. We restate the result.

\begin{example}\cite{GKS}
The Hilbert scheme $\Hilb^n(\{xy=0\},0)$ has two types of ideals.

\begin{enumerate}
\item $(\alpha x^i+ \beta y^j)$, where $i,j >0$ and $i+j=n$.

\item $(x^i, y^j)$, where $i+j=n+1$.
\end{enumerate}

By taking limits as $\alpha$ or $\beta$ go to $0$, we obtain the second type of ideal from the first type. So, $$\{(\alpha x^i+\beta y^j),(x^{i+1},y^j),(x^i,y^{j+1})|\alpha,\beta \neq 0 , i+j=n\} = \mathbb{P}^1,$$ and we get that for $n>1$, $\Hilb^n(\{xy=0\}, 0)$ is a chain of $n-1$ projective lines $\mathbb{P}^1$, where each $\mathbb{P}^1$ corresponds to a way of writing n as a sum of two positive integers. The remaining cases, $\Hilb^0(\{xy=0\}, 0)$ and $\Hilb^1(\{xy=0\}, 0)$ are just points.

Let $Q$ be the variable recording the number of points of the Hilbert scheme, and $T$ recording the homological degree, we have the following theorem. The Poincar\'e polynomial therefore is 
\begin{align*}
&\sum_{n=0}^{\infty}\chi(\Hilb^n(\{xy=0\},0))Q^n = \\
& 1+Q+Q^2(1+T^2)+Q^3(1+2T^2)+Q^4(1+3T^2)+Q^5(1+4T^2)+...  \\
& = \frac{1-Q+Q^2T^2}{(1-Q)^2}.
\end{align*}

\end{example}

\subsection{Non-reduced smooth curve}

The homology (also, Hilbert-zeta function) of the Hilbert scheme of $n$ points on a \textbf{non-reduced smooth} curve, say for example $\{y^k=0\}$, was computed in \cite{Bej}. The proof strategy develops from the classical Bialynicki-Birula cell decomposition of the punctual Hilbert scheme of points on the plane. We restate the result.

\begin{example}
\label{egy^k}
\cite{Bej}
The Hilbert scheme $\Hilb^n(\{y^k=0\},0)$ has a cell decomposition, where each cell is a Bialynicki-Birula cell of $\Hilb^n(\mathbb{C}^2,0)$, following from the torus action on the ideals in the Hilbert scheme. Each cell is indexed by a partition $\lambda$ of $n$ such that the parts of $\lambda \leq k$. The dimension of each cell is equal to $n-l(\lambda)$, where $l(\lambda)$ is the biggest part of $\lambda$.

The Poincar\'e polynomial is 
$$\sum_{n=0}^{\infty}\chi(\Hilb^n(\{y^k=0\},0))Q^n =  \prod_{i=1}^k \limits \frac{1}{1-Q^iT^{i-1}}.$$
\end{example}

In this paper, we consider the \textbf{non-reduced and nodal} curve $C=\{x^uy^v=0\}$, such that $u, v \in \mathbb{Z}, u,v \geq 1$. Assume without loss of generality that $u \leq v$.

\subsection{Virtual Poincar\'e polynomial}
We define the virtual Poincar\'e polynomial of the Hilbert scheme $\Hilb^n(\{x^uy^v=0\},0)$ as the virtual Poincar\'e polynomial of its underlying reduced variety, which we restate as follows.

\begin{definition}\cite{BP} \label{vpp}
For every algebraic variety $X$ (possibly singular, or reducible), its virtual Poincar\'e polynomial $\chi(X)$ is uniquely determined by the following properties:
\begin{enumerate}
\item (additivity) $\chi(X) = \chi(Y) + \chi(X-Y)$ for every closed subvariety $Y$ .
\item If $X$ is smooth and complete, then $\chi(X) = \sum_m \dim (\mathrm{H}^m(X)) T^m$ is the usual Poincar\'e
polynomial.
\end{enumerate}
Then $\chi(X) = \chi(Y)\chi(F)$ for every fibration $F \rightarrow X \rightarrow Y$ which is locally trivial for the Zariski topology. 
\end{definition}
\begin{remark}
%The virtual Poincar\'e polynomial of an affine space $\mathbb{A}^n$ is 
%$$\chi(\mathbb{A}^n)(T)= \sum_m \dim H^m(\mathbb{A}^n) T^m = T^0=1.$$

The virtual Poincar\'e polynomial of an algebraic torus $(\mathbb{P}^1-2pt)^n$ is 
\begin{align*}
&\chi((\mathbb{P}^1-2pt)^n)= (\chi(\mathbb{P}^1)-\chi(2pt))^n\\
&= (\sum_m \dim (\mathrm{H}^m(\mathbb{P}^1)) T^m-\sum_m \dim (\mathrm{H}^m(2pt)) T^m)^n\\
&= ((T^2+1)-2)^n=(T^2-1)^n.   
\end{align*}
The virtual Poincar\'e polynomial of an affine space $\mathbb{A}^n$ is
\begin{align*}
&\chi(\mathbb{A}^n)=(\chi(\mathbb{A}))^n=(\chi(\mathbb{P}^1-pt))^n =(\dim(\mathrm{H}^m(\mathbb{P}^1-pt))T^m)^n\\
&=(T^2+1-1)^n =T^{2n}.
\end{align*}

\end{remark}

\section{Two filtrations of an ideal in the Hilbert scheme}

Assume without loss of generality that $u \leq v$ throughout the paper.

For an ideal $I \in \Hilb^n(\{x^uy^v=0\},0)$, our strategy is not to find all the generators of the ideal $I$ simultaneously, but to construct a filtration of ideals of $I$ compatible with the non-reduced structure of the curve $\{x^uy^v=0\}$, such that at each layer of the filtration, we add in one or two generators of $I$, and inductively find all generators of $I$.

Depending on the values of $u,v$, we naturally require two different filtrations.

\medskip

\subsection{Vertical Filtration}

To study $\Hilb^n(\{x^uy^v=0\},0)$, where 
$u$ is small, we define a filtration named Vertical.

For any $I \in \Hilb^n(\{x^uy^v=0\},0)$, there exist an integer $N$ sufficiently large such that the monomial $y^N \in I$. We consider the following filtration on the power series ring $\mathbb{C}[\![x,y]\!]$ of length $N+1$: 
$$ (y^{N}) \subset ... \subset (y^v) \subset (y^{v-1}) \subset ... \subset (y^2) \subset (y) \subset (1)=\mathbb{C}[\![x,y]\!].$$ 

\begin{definition}
\label{filtration_V}
For any ideal $I \subset \mathbb{C}[\![x,y]\!]$, we take the intersections with the filtration above, $W^{k} :=I \cap (y^k)$ for all $k=0,...,N$. Therefore $$W^{N} \subset W^{N-1} \subset ... \subset W^{2} \subset W^{1} \subset W^{0} = I$$ is a filtration of ideals on $I$, named Vertical. 
\end{definition}

\subsection{Diagonal Filtration.}

When $u$ is close to $v$, for example $u=v,v-1,v-2$, the non-reduced structure of the curve $\{x^uy^v=0\}$ naturally gives the following filtration of length $v+1$ of the power series ring $\mathbb{C}[\![x,y]\!]$: 
$$(x^vy^v) \subset (x^{v-1}y^{v-1}) \subset ... \subset (x^2y^2) \subset (xy) \subset (1)=\mathbb{C}[\![x,y]\!].$$ 

\begin{definition}
\label{filtration_D}
For any ideal $I \subset \mathbb{C}[\![x,y]\!]$, we take the intersections with the filtration above, $J^{k} :=I \cap (x^ky^k)$ for all $k=0,...,v$. Therefore $$J^{v} \subset J^{v-1} \subset ... \subset J^{2} \subset J^{1} \subset J^{0} = I$$ is a filtration of ideals on $I$, named Diagonal.
\end{definition}

\subsection{Two Filtrations of vector spaces}

The key part of the definition of an ideal $I$ in the Hilbert scheme is that quotient of the ideal $I$ by the coordinate ring is a vector space over $\mathbb{C}$ of dimension $n$, $\dim(\frac{\mathbb{C}[\![x,y]\!]}{I})=n$. Because the ideal $I$ has the Vertical and Diagonal filtrations, the vector space $\frac{\mathbb{C}[\![x,y]\!]}{I}$ naturally also has filtrations.

\begin{definition}
For an ideal $I \in \Hilb^{n}(\{x^uy^v=0\},0)$, $k=0,...,v$, define the vector space $\Gamma^{k}$ over $\mathbb{C}$ (which has a diagonal flavor): \begin{align*}
\Gamma^{k} := \frac{(x^ky^k)}{(x^ky^k) \cap I} = \frac{(x^ky^k)}{J^{k}}.   
\end{align*}

For $k=0,...,N$, define the vector space $\Theta^{k}$ over $\mathbb{C}$ (which has a vertical flavor): \begin{align*}
\Theta^{k} := \frac{(y^k)}{(y^k) \cap I} = \frac{(y^k)}{W^{k}}.   
\end{align*}

Consider the two filtrations of vector spaces, 

$$(0)=\Gamma^{v} \subset \Gamma^{v-1} \subset ... \subset \Gamma^{k} ... \subset  \Gamma^{1} \subset \Gamma^{0}=\frac{\mathbb{C}[\![x,y]\!]}{I}.$$

$$(0)=\Theta^{N} \subset \Theta^{N-1} \subset ... \subset \Theta^{k} ... \subset  \Theta^{1} \subset \Theta^{0}=\frac{\mathbb{C}[\![x,y]\!]}{I}.$$

\end{definition}

\begin{lemma}
\label{lemvectorfilt}
For an ideal $I \in \Hilb^{n}(\{x^uy^v=0\},0)$, we have 

\begin{align*}
\sum_{k=0}^{v-1} \dim (\frac{\Gamma^{k}}{\Gamma^{k+1}})=n.
\end{align*}

And 
\begin{align*}
\sum_{k=0}^{N-1} \dim (\frac{\Theta^{k}}{\Theta^{k+1}})=n.
\end{align*}
\end{lemma}

\begin{proof}
This follows immediately from the vector space decomposition 
\begin{align*}
\Gamma^{0}=\bigoplus_{t=0}^{v-1}\frac{\Gamma^{k}}{\Gamma^{k+1}} \oplus \Gamma^{v},
\end{align*} 
and
\begin{align*}
\Theta^{0}=\bigoplus_{t=0}^{N-1}\frac{\Theta^{k}}{\Theta^{k+1}} \oplus \Theta^{N}.
\end{align*}
And that both $\Gamma^{v}=0$ and $\Theta^{N}=0$ have dimension $0$, and $\Gamma^{0}= \Theta^{0}=\frac{\mathbb{C}[\![x,y]\!]}{I}$ has dimension $n$ as a vector space over $\mathbb{C}$.
\end{proof}

\begin{corollary}
\label{cordim}
All vector spaces $\Gamma^{k}$ and $\Theta^k$ has finite dimension,  $\dim(\Gamma^{k}) \leq n$ for all $0 \leq k \leq v-1$, and $\dim(\Theta^{k}) \leq n$ for all $0 \leq k \leq N-1$.
\end{corollary}

\begin{proof}
Because $\Gamma^{k}$ is a vector subspace of $\Gamma^{0}$, it has dimension less than equal to the dimension $\dim(\Gamma^{0})=n$. The same proof applies to $\Theta^{k}$.
\end{proof}

\section{Generators}

\subsection{Vertical filtration}

Construct the ideal $I$ inductively from the inner layers to the outer layers in the following sequence: $$W^{N-1}, W^{N-2}, ..., W^{1}, W^{0}.$$ 
To do so, we first find the generators of $W^{N-1}$, and then add generators to obtain the ideal $W^{N-2}$, and inductively add generators to obtain the ideal $I=W^{0}$.

\begin{lemma}
\label{gen_v}
When $0 \leq k \leq N-1$, the generator we add to construct $W^{k}$ on top of $W^{k+1}$ is: 
\begin{align*}
F=y^k \cdot x^i + y^{k+1} \cdot t(x,y), 
\end{align*}
for some integer $i \geq 0$, and power series $t(x,y) \in \mathbb{C}[\![x,y]\!]$.
\medskip
\end{lemma}

\begin{proof}
Assume as part of the induction hypothesis that we know the generators of $W^{k+1}$.

For a generator $f(x,y)$ of $W^{k}$, we can write $f(x,y)$ as
\begin{align*}
f(x,y)=y^k\cdot (a_0+ a_1x+a_{2}x^{2}+...)+y^{k+1}\cdot t(x,y),
\end{align*}
for some power series $t(x,y) \in \mathbb{C}[\![x,y]\!]$.

\textbf{(Existence.)}
Because $\Theta^{k}$ has finite dimension (Corollary \ref{cordim}), we must have that $a_i \neq 0$ for some integer $i \geq 0$. Otherwise, the following infinite sequence of monomials $\{y^k ,  xy^k, \ x^2y^k, ...\}$ are linearly independent and we have a contradiction. Consider the smallest $i$ such that $a_i \neq 0$, then we can rewrite the power series $f(x,y)$ as 
\begin{align*}
f(x,y)=y^k\cdot (a_ix^i+a_{i+1}x^{i+1}+...)+y^{k+1}\cdot t(x,y)\\
= y^k\cdot x^i(a_i+a_{i+1}x^{1}+...)+y^{k+1}\cdot t(x,y)
\end{align*} where $a_i+a_{i+1}x^{1}+...$ is a unit, denoted by $U$, in $\mathbb{C}[\![x,y]\!]$.

Compute 
\begin{align*}
U^{-1} \cdot f(x,y)=y^k\cdot x^i+y^{k+1}\cdot U^{-1} \cdot t(x,y).
\end{align*}

Because $U^{-1} \cdot t(x,y)$ is also a power series in $\mathbb{C}[\![x,y]\!]$, we obtain a generator of the form $F$.

\textbf{(Uniqueness.)}
Assume we have more than one generators of the form $F$, denoted by
\begin{align*}
F_1(x,y)= y^k\cdot x^{i_1}+y^{k+1}\cdot t_1(x,y).\\
F_2(x,y)= y^k\cdot x^{i_2}+y^{k+1}\cdot  t_2(x,y).
\end{align*}

Assume without loss of generality that $i_1 \leq i_2$, then compute 
\begin{flalign*}
&F_2(x,y)-x^{i_2-i_1}F_1(x,y)\\
&=y^{k+1}(t_2(x,y)-x^{i_2-i_1}t_1(x,y)),
\end{flalign*}
which is a power series in $W^{k+1}$. This means that $F_2(x,y)$ can be generated from $F_1(x,y)$ and the generators of $W^{k+1}$, so it is redundant. In summary, we only need one generator $F$ to construct the layer $W^{k}$.
\end{proof}

\subsection{Diagonal filtration.}

Construct the ideal $I$ inductively from the inner layers to the outer layers in the following sequence: $$J^{v-1}, J^{v-2}, ..., J^{1}, J^{0}.$$ 
To do so, we first find the generators of $J^{v-1}$, and then add generators to obtain the ideal $J^{v-2}$, and inductively add generators to obtain the ideal $I=J^{0}$.

\begin{lemma}
\label{lemmagen}
Let $I \in \Hilb^n(\{x^uy^v=0\},0)$.

\textbf{Base case:}
The ideal $J^{v}=(x^uy^v)\cap I$ is principal, generated as $(x^uy^v)$.

\textbf{Induction:} Assume knowing the generators of $J^{k+1}$, we classify the generators we add to construct $J^{k}$.

When $0 \leq k \leq v-1$, the generator(s) we add to construct $J^{k}$ from $J^{k+1}$ are one of the three cases.
\begin{enumerate}

\item (A monomial)$$x^ky^k,$$

\item (A single generator)
$$E=x^ky^k\cdot(ax^i+by^j)+x^{k+1}y^{k+1}\cdot h(x,y),$$

\item (A pair of generators)
\begin{align*}
&E_1=x^ky^k\cdot x^i+x^{k+1}y^{k+1}\cdot h_1(x,y),\\
&E_2=x^ky^k\cdot y^j+x^{k+1}y^{k+1}\cdot h_2(x,y),   
\end{align*}

\end{enumerate}
for some $a, b \in \mathbb{C}$, $a, b \neq 0$;
integers $i,j \geq 1$;
and power series 

$h(x,y), h_1(x,y), h_2(x,y) \in \mathbb{C}[\![x,y]\!]$.

\end{lemma}

\begin{proof}
By definition of the Hilbert scheme, an ideal $I \in \Hilb^n(\{x^uy^v=0\},0)$ implies that $x^uy^v \in I$, and recall that we always assume that $v \geq u$, so $x^vy^v \in I$. So $I \cap (x^vy^v)=(x^vy^v)$. This proves the base case. Now we show the induction steps.

\begin{enumerate}

\item
Assume that $0 \leq k \leq v-1$.

\textbf{(Existence.)}
For a generator $f(x,y) \in J^{k}$, we can write it as 
\begin{align*}
f(x,y)=x^ky^k\cdot (c+a_1x+a_2x^2+...+b_1y+b_2y^2+...)+x^{k+1}y^{k+1}\cdot h(x,y), 
\end{align*}
for some $h(x,y) \in \mathbb{C}[\![x,y]\!]$.

When $c \neq 0$, $c+a_1x+a_2x^2+...+b_1y+b_2y^2+...$ is a unit in $\mathbb{C}[\![x,y]\!]$, denoted by $U$. So \begin{align*}
&U^{-1} \cdot f(x,y) = x^ky^k+x^{k+1}y^{k+1} \cdot U^{-1} \cdot h(x,y)\\
&=x^ky^k(1+xy\cdot U^{-1} \cdot h(x,y)).
\end{align*}
Because $1+xy\cdot U^{-1} \cdot h(x,y)$ is a unit, denoted by $U'$, in $\mathbb{C}[\![x,y]\!]$, we multiply its inverse $U'^{-1}$ to $U^{-1} \cdot f(x,y)$  to obtain $x^ky^k$. So $x^ky^k$ is a generator of $J^{k}$. This proves the case of (1).

When $c=0$, there must be some integers $i,j \geq 1$ such that $a_i, b_j \neq 0$. Assume for contradiction that there exists no such $b_j$, i.e. all $b_j=0$. Then $$l:=\{ x^ky^k\cdot y, \ x^ky^k\cdot y^2, \ x^ky^k\cdot y^3, ...\}$$ is an infinite sequence of linearly independent monomials in $\Gamma^{k}=\frac{(x^ky^k)}{J^{k}}$. This contradicts with Corollary \ref{cordim}.

Because there are $a_i, b_j \neq 0$, write 
\begin{align*}
& \ \ \ \ \ \ \ f(x,y)=x^ky^k(a_ix^i+a_{i+1}x^{i+1}+...
+b_jy^j+b_{j+1}y^{j+1}+...)+x^{k+1}y^{k+1}h(x,y)\\
& \ \ \ \ \ \ \ =x^ky^k(a_ix^i(1+\frac{a_{i+1}}{a_i}x^1+...)
+b_jy^j(1+\frac{b_{j+1}}{b_j}y^{1}+...))+x^{k+1}y^{k+1}h(x,y)
\end{align*}

Denote the units $$U_1:=a_i+\frac{a_{i+1}}{a_i}x^1+...,$$  $$U_2:=b_j+\frac{b_{j+1}}{b_j}y^{1}+...,$$ then $U_1^{-1}$ is a power series in the variable $x$ with constant term $1$, and $U_2^{-1}$ is a power series in the variable $y$ with constant term $1$.
 
Compute 
\begin{align*}
&U_1^{-1} \cdot U_2^{-1} \cdot f(x,y)\\
&=x^ky^k(U_2^{-1}a_ix^i
+U_1^{-1}b_jy^j)+U_1^{-1}U_2^{-1}x^{k+1}y^{k+1}h(x,y).
\end{align*}
Because $U_2^{-1}a_ix^i$ is equal to $a_ix^i$ plus some terms divisible by $xy$, and $U_1^{-1}b_jy^j$ is equal to $b_jy^j$ plus some terms divisible by $xy$, we can collect all the terms in $U_2^{-1}a_ix^ix^ky^k,\  U_1^{-1}b_jy^jx^ky^k,\   U_1^{-1} \cdot U_2^{-1} \cdot h(x,y)x^{k+1}y^{k+1}$ that are divisible by $x^{k+1}y^{k+1}$ as $h'(x,y)$ and  write 
\begin{align*}
E := U_1^{-1} \cdot U_2^{-1} \cdot f(x,y)=x^ky^k(a_ix^i
+b_jy^j)+x^{k+1}y^{k+1}h'(x,y)
\end{align*}
for some power series $h'(x,y)$. We obtain the generator of type (2).

Suppose there are more than one generators \begin{align*}
&E_1 :=x^ky^k(a_{i_1}x^{i_1}
+b_{j_1}y^{j_1})+x^{k+1}y^{k+1}h_1(x,y)\\
&E_2 :=x^ky^k(a_{i_2}x^{i_2}
+b_{j_2}y^{j_2})+x^{k+1}y^{k+1}h_2(x,y)
\end{align*}
If $i_1 < i_2$ and $j_1 < j_2$, then $E_2-\frac{a_{i_2}}{a_{i_1}}x^{i_2-i_1}E_1 -\frac{b_{j_2}}{b_{j_1}}y^{j_2-j_1}E_1$ is divisible by $x^{k+1}y^{k+1}$ and therefore in the ideal $J^{k+1}$, so only $E_1$ is needed as the additional generator for $J^{k}$ and $E_2$ is redundant. The case when $i_1 > i_2$ and $j_1 > j_2$ is exactly the same.

So assume now that $i_1 \leq i_2$ and $j_1 \geq j_2$. Compute 
\begin{align*}
&E_1-\frac{b_{j_1}}{b_{j_2}}y^{j_1-j_2}E_2=a_{i_1}x^ky^kx^{i_1}+x^{k+1}y^{k+1}h_1'(x,y),\\
&E_2-\frac{a_{i_2}}{a_{i_1}}x^{i_2-i_1}E_1=b_{j_2}x^ky^ky^{j_2}+x^{k+1}y^{k+1}h_2'(x,y),
\end{align*}
for some power series $h_1'(x,y),h_2'(x,y)$.
So we have two generator of type (3). 

\textbf{(Uniqueness.)}
If the generator is type (1), $x^ky^k$, then every element in $J^{k}$ is generated by it, so any other generator is redundant.

Otherwise, suppose there are two or more candidates of generators,
\begin{align*}
&E_1 :=x^ky^k(a_{i_1}x^{i_1}
+b_{j_1}y^{j_1})+x^{k+1}y^{k+1}h_1(x,y)\\
&E_2 :=x^ky^k(a_{i_2}x^{i_2}
+b_{j_2}y^{j_2})+x^{k+1}y^{k+1}h_2(x,y)\\
&E_3 :=x^ky^k(a_{i_3}x^{i_3}
+b_{j_3}y^{j_3})+x^{k+1}y^{k+1}h_3(x,y)\\
...
\end{align*}
following the computation and algorithm above, it's only necessary to pick either one or two generators $E_{m_1}, E_{m_2}$ out of those $E_1, E_2, E_3,...$ such that $i_{m_1} \leq i_l$ and $j_{m_2} \leq j_l$ for all $l$. When $m_1=m_2$, there is one unique generator as type (2), and when $m_1 \neq m_2$, there is a unique pair of generators as type (3).
\end{enumerate}
\end{proof}

\section{Non-reduced valuation}

Before diving deeper to the conditions on the generators, we define a valuation called the non-reduced valuation, modified from the intersection valuation of coordinate ring of reduced curves on an integral domain (see for example Pol's thesis \cite{Pol}), motivated from the filtration of an ideal $I \in \Hilb^n(C,0)$. %The non-reduced valuation can be defined on the coordinate ring of any curve $C=\{g(x,y)=\prod_{k=1}^{r}(g_k(x,y))^{m_k} \}$ with $r$ components $$g_1(x,y),...,g_k(x,y),..., g_r(x,y),$$ such that each component $g_k(x,y)$ has multiplicity $m_k$. But because we only discuss the curve $x^uy^v=0$ in this paper, we only make the definition and computation specifically for this curve to avoid the heavy notations.

\subsection{Definition and properties}

\begin{definition}
For a pair of non-negative integers $(u,l)$, and any power series $f(x,y) \in \mathbb{C}[\![x,y]\!]$ such that $f(x,y)$ is divisible by $x^uy^l$, define a function $v_{(u,l)}$ that maps a power series to a point in $(\mathbb{N} \cup \{\infty\})^2$

$$v_{(u,l)}(f):= ([\frac{f}{x^uy^l},y], [\frac{f}{x^uy^l},x]),$$ 
where $[ \phi(x,y) , x]$ (or $[\phi(x,y), y]$) denotes %the intersection multiplicity of the power series $\phi(x,y)$ with the curve component in $x$ (or $y$) at the point $(0,0)$. 
%Algebraically, $[\phi(x,y), x]$ is 
the dimension of the ring $\frac{\mathbb{C}[\![x,y]\!]}{(\phi(x,y),x)}$ as a vector space over the field $\mathbb{C}$, (and similarly $\frac{\mathbb{C}[\![x,y]\!]}{(\phi(x,y),y)}$ for $[\phi(x,y), y]$.)
\end{definition}

\begin{example}
%Let the curve be $x^5y^{11}=0$.
\begin{enumerate}
\item For the polynomial $x^i+y^j \in \mathbb{C}[\![x,y]\!]$, its valuation is $v_{(0,0)}(x^i+y^j) = ([x^i+y^j,y], [x^i+y^j,x])= (i, j)$.
There are no other non-reduced valuation. 

\item For the polynomial $x^2y$, some of its valuations are $v_{(0,0)}(x^2y)=([x^2y,y], [x^2y,x])=(\infty, \infty)$,
$v_{(1,1)}(x^2y)=([x,y], [x,x])=(1, \infty)$, and $v_{(2,1)}(x^2y)=([1,y], [1,x])=(0, 0)$.
\item The valuation $v_{(0,0)}(x^4y^4)=(\infty, \infty)$.

\item For the polynomial $4x^5y^7+10x^8y^5+x^6y^7$, one of its valuations is $$v_{(5,5)}(4x^5y^7+10x^8y^5+x^6y^7)=(3,2).$$
\item For the polynomial $4x^5y^{10}+x^6y^9$,
one of its valuations is $$v_{(5,6)}(4x^5y^{10}+x^6y^9)=(\infty,4).$$
Another one is 
$$v_{(5,9)}(4x^5y^{10}+x^6y^9)=(1,1).$$
\end{enumerate}

\end{example}

%\begin{remark}
%\textcolor{red}{check this}
%For a power series $f(x,y) \in J^{k}$, its non-reduced valuation $v_{(k)}(f(x,y))$ has finite integers in all of its entries if and only if $f(x,y) \notin J^{k+1}$. 
%\end{remark}

Intuitively, the non-reduced valuation $v_{(u,l)}$ of a power series $f(x,y)$ takes out a factor $x^uy^l$ from $f(x,y)$ and then take the intersection valuation with the underlying reduced curve components $x$ and $y$ of $x^uy^v=0$. The non-reduced valuation $v_{(0,0)}$ is an extension of the usual intersection valuation from an integral domain $\frac{\mathbb{C}[\![x,y]\!]}{(xy)}$ to the entire coordinate ring of a non-reduced curve $\frac{\mathbb{C}[\![x,y]\!]}{(x^uy^v)}$. But having only $v_{(0,0)}$ is not enough, as can be seen from Example (2)(3) above, that $v_{(0,0)}$ does not distinguish elements like $x^2y$ or $x^4y^4$, or any other elements that are sent to $(\infty, \infty)$. We would need the other non-reduced valuations to distinguish these elements in the ring that is not an integral domain.  

\begin{lemma}
\label{lem00}
For any power series $\phi_1(x,y), \phi_2(x,y) \in \mathbb{C}[\![x,y]\!]$, the non-reduced valuation $v_{(0,0)}$ satisfy that  
\begin{align}
&v_{(0,0)}(\phi_1(x,y))+v_{(0,0)}(\phi_2(x,y))=v_{(0,0)}(\phi_1(x,y)\phi_2(x,y)),\\
&v_{(0,0)}(\phi_1(x,y)+\phi_2(x,y)) \geq \min(v_{(0,0)}(\phi_1(x,y)), v_{(0,0)}(\phi_2(x,y))),
\end{align}
with cancellation law for (1) if the summand are finite, and with equality for (2) if $v_{(0,0)}(\phi_1(x,y)) \neq v_{(0,0)}(\phi_2(x,y))$.
\end{lemma}

\begin{proof}
When the summand are finite, these follow immediately from properties of the intersection multiplicity. We still write carefully the proof and check the subtlety when some summand are $\infty$. Recall that addition is extended to $\mathbb{N} \cup \{\infty\}$ by $k+\infty = \infty$ for any $k \in \mathbb{N} \cup \{\infty\}.$

Denote the image of $\phi_1(x,y)$ under the map of modulo $x$ as $\Tilde{\phi_1}_y$ (intuitively, killing the variable $x$ and leaving the monomials only with variable $y)$, the image of $\phi_1(x,y)$ under the map of modulo $y$ as $\Tilde{\phi_1}_x$; and denote the order of $\Tilde{\phi_1}_x$ as $p_{1,x}$  (intuitively, the lowest degree of the monomials only in the variable $x$), denote the order of $\Tilde{\phi_1}_y$ as $p_{1,y}$. Similarly define $\Tilde{\phi_2}_x$, $\Tilde{\phi_2}_y, p_{2,x}, p_{2,y}$ for the images and orders obtained from $\phi_2$; define $\Tilde{\phi_1\phi_2}_x$, $\Tilde{\phi_1\phi_2}_y, p_{12,x}, p_{12,y}$ for the images and orders obtained from  $\phi_1\phi_2$. All of the orders can be $\infty$.

Compute
\begin{align*}
&v_{(0,0)}(\phi_1(x,y))+v_{(0,0)}(\phi_2(x,y))\\
&=(\dim(\frac{\mathbb{C}[\![x,y]\!]}{(\phi_1,y)}),\dim(\frac{\mathbb{C}[\![x,y]\!]}{(\phi_1,x)}))+(\dim(\frac{\mathbb{C}[\![x,y]\!]}{(\phi_2,y)}),\dim(\frac{\mathbb{C}[\![x,y]\!]}{(\phi_2,x)}))\\
&=(\dim(\frac{\mathbb{C}[\![x]\!]}{(\Tilde{\phi_1}_x)}),\dim(\frac{\mathbb{C}[\![y]\!]}{(\Tilde{\phi_1}_y)}))+(\dim(\frac{\mathbb{C}[\![x]\!]}{(\Tilde{\phi_2}_x)}),\dim(\frac{\mathbb{C}[\![y]\!]}{(\Tilde{\phi_2}_y)}))\\
&=(p_{1,x}, p_{1,y})+ (p_{2,x}, p_{2,y})\\
&=(p_{1,x}+p_{2,x}, p_{1,y}+p_{2,y}).\\
\end{align*}
Compute 
\begin{align*}
&v_{(0,0)}(\phi_1(x,y)\phi_2(x,y))\\
&=(\dim(\frac{\mathbb{C}[\![x,y]\!]}{(\phi_1\phi_2,y)}),\dim(\frac{\mathbb{C}[\![x,y]\!]}{(\phi_1\phi_2,x)}))\\
&=(\dim(\frac{\mathbb{C}[\![x]\!]}{(\Tilde{\phi_1\phi_2}_x)}),\dim(\frac{\mathbb{C}[\![y]\!]}{(\Tilde{\phi_1\phi_2}_y)}))\\
&=(p_{12,x}, p_{12,y}).
\end{align*}

We want to show $p_{12,x}=p_{1,x} +p_{2,x}$. This is an easy check when all the summand are finite, following from multiplying $\phi_{1x}$ and $\phi_{2x}$ in the ring $\mathbb{C}[\![x]\!]$. When some or all of the summand are infinite, say for example, $p_{1,x} = \infty$, which implies $\Tilde{\phi_{1x}}=0$, so $\Tilde{\phi_{1}\phi_{2x}}=0$ and $p_{12,x} = \infty$. So we check that $\infty + p_{2,x}=\infty$ and indeed $p_{12,x}=p_{1,x} +p_{2,x}$. Same holds for $p_{12,y}=p_{1,y} +p_{2,y}$. This proves (1).

For (2), we compute that 
\begin{align*}
&\min(v_{(0,0)}(\phi_1(x,y)), v_{(0,0)}(\phi_2(x,y)))
=(\min(p_{1,x},p_{2,x}),\min(p_{1,y},p_{2,y})).
\end{align*}
Also compute
\begin{align*}
&v_{(0,0)}(\phi_1(x,y)+\phi_2(x,y))=\dim(\frac{\mathbb{C}[\![x,y]\!]}{(\phi_1(x,y)+\phi_2(x,y)}))\\
&=(\dim(\frac{\mathbb{C}[\![x]\!]}{(\Tilde{\phi_{1x}}+\Tilde{\phi_{2x}})},\dim(\frac{\mathbb{C}[\![y]\!]}{(\Tilde{\phi_{1y}}+\Tilde{\phi_{2y})}}))\\
&=(\min(p_{1,x},p_{2,x}),\min(p_{1,y},p_{2,y})),
\end{align*}
and the last equality in the above equations holds when $p_{1,x} \neq p_{2,x}$ and $p_{1,y} \neq p_{2,y}$. (When $p_{1,x} = p_{2,x}$ or $p_{1,y} = p_{2,y}$, cancellation may happen and we may not have equality.)
\end{proof}

\begin{lemma}
\label{properties}
Let $(k,t)$ and $(u,l)$ be pairs of non-negative integers. The non-reduced valuations $v_{(k,t)}$ and $v_{(u,l)}$ satisfy properties similar to the standard valuation. For some power series $f(x,y), g(x,y) \in \mathbb{C}[\![x,y]\!]$ that are divisible by the appropriate monomials $x^ky^k, x^uy^l$, we have
\begin{enumerate}

\item $$v_{(k,t)}(f(x,y))+v_{(u,l)}(g(x,y))=v_{(k+u,k+l)}(f(x,y)g(x,y)),$$ where we have cancellation law when all summands are finite.

\item $$v_{(k,t)}(f(x,y)+g(x,y)) \geq \min(v_{(k,t)}(f(x,y)), v_{(k,t)}(g(x,y)))$$
$$ \text{ with equality if } v_{(k,t)}(f(x,y)) \neq v_{(k,t)}(g(x,y)).$$

%\item $v_{k}(f(x,y))$ valuates to $(\infty, \infty)$, if and only if $f(x,y)=0$.
\end{enumerate}

\end{lemma} 

\begin{proof}
Using Lemma \ref{lem00}, compute
\begin{enumerate}
\item \begin{align*}
&v_{(k,t)}(f(x,y))+v_{(u,l)}(g(x,y))=v_{(0,0)}(\frac{f(x,y)}{x^ky^k})+v_{(0,0)}(\frac{g(x,y)}{x^uy^l})\\
&= v_{(0,0)}(\frac{f(x,y)}{x^ky^k}\cdot \frac{g(x,y)}{x^uy^l}))=v_{(0,0)}(\frac{f(x,y)g(x,y)}{x^{k+u}y^{k+l}}) = v_{(k+u,k+l)}(f(x,y)g(x,y)).  
\end{align*}

\item  $v_{(k,t)}(f(x,y)+g(x,y))= v_{(0,0)}(\frac{f(x,y)+g(x,y)}{x^ky^k})= v_{(0,0)}(\frac{f(x,y)}{x^ky^k}+ \frac{g(x,y)}{x^ky^k}) \geq \min(v_{(0,0)}(\frac{f(x,y)}{x^ky^k}), v_{(0,0)}(\frac{g(x,y)}{x^ky^k})) = \min(v_{(k,t)}(f(x,y), v_{(k,t)}(g(x,y)))).$

%\item The element $0$ obviously valuations to $(\infty, \infty)$ for all possible $v_k$. Conversely, assume for contradiction that $f(x,y)$ is some non-zero power series, then there must exist a maximum power of $xy$, say $x^sy^s$ such that $x^sy^s$ divides $f(x,y)$. The remainder after the division, denoted by $f'(x,y)$ where $f(x,y)= x^sy^sf'(x,y)$, is also non-zero, and not divisible by $xy$, so $f'(x,y)$ must have either a constant term, or a term $x^i$, or a term $y^j$ for some $i,j$. So $v_s(f(x,y))=v(f'(x,y))$ is not $(\infty, \infty)$ for this $s$, contradiction. 

\end{enumerate}
\end{proof}

\subsection{Valuation of the Vertical filtration}

\begin{definition}
Take integers $0 \leq k \leq N-1$, and the Vertical filtration $(W^{k})$ of an ideal $I \in \Hilb^n(\{x^uy^v=0\},0)$, then the Vertical valuation $v_{(0,k)}$ is defined for all power series in $W^{k}$. Denote the image as $v_{(0,k)}(W^{k})$, which is a lattice in $(\mathbb{N} \cup \{\infty\})^2$.
\end{definition}

Recall we showed in Lemma \ref{gen_v} that the generator $F$ of $W^k$ for $0 \leq k \leq N-1$ is 
$$F^k=y^k\cdot x^{i_k}+ y^{k+1}\cdot t(x,y),$$ for some nonnegative integer $i_k$, and power series $t(x,y)$.

\begin{lemma}
\label{lemik}
Assume $0 \leq k \leq N-1$, consider the Vertical valuation $v_{(0,k)}(W^k)$. All the points $(a,b)$ in $v_{(0,k)}(W^k)$ satisfy that $a \geq i_k$.
\end{lemma}

\begin{comment}

and the row where the $y$ coordinate is equal to $\infty$ is the following lattice, where $i_k \geq 0$ is the degree of $x$ as in the generator $F^k$.

\begin{center}
\begin{tikzpicture}
\clip (-0.7,-0.2) rectangle (10cm,1cm); 

\foreach \y in {0}
    {
\node[draw,circle,inner sep=1pt,fill, label={\tiny($i_k$,$\infty$)}] at (1.8*0,\y) {};
        }
        
\foreach \x in {1,...,2}
    {
        \foreach \y in {0}
        {
        \node[draw,circle,inner sep=1pt,fill, label={\tiny($i_k$+\x ,$\infty$ )}] at (1.8*\x,\y) {};
        }
    }
     
%\foreach \x in {0,...,2}
    %{\node[draw,circle,inner sep=1pt,fill,label={...}] at (1.8*\x,1*3) {};}

\foreach \y in {0}
    {\node[draw,circle,inner sep=1pt,fill,label={...}] at (1.8*3,\y) {};}
\end{tikzpicture}
    
\end{center}
\end{comment}

\begin{proof}

%Take the generator $F^k$ and its product with all monomials of $x,y$, namely $F^k \cdot x^{p}$ for all $p \geq 0$, and they valuate (by Lemma \ref{properties}) to 

Because $F^k$ is the unique generator of $W_k$ on top of $W_{k+1}$, any power series in $W_k$ but not in $W_{k+1}$ is a product of $F^k$ with some other series $g \in \mathbb{C}[\![x,y]\!]$. So any element in $W^k$ is $F^k \cdot g + w$ for some elements $w$ in $W^{k+1}$. %And $w$ valuates to $v_{(0,k)}(w)=v_{(0,0)} \frac{w}{y^k}$. 

Because $w \in W^{k+1}$ is divisible by $y^{k+1}$, $\frac{w}{y^k}= y \cdot l$ for some $l \in \mathbb{C}[\![x,y]\!]$. Now compute using Lemma \ref{properties}
%$v_{(0,0)} \frac{w}{y^k}=(l,\infty)$ for some $l \in (\mathbb{N} \cup \{\infty\})$.
\begin{align*}
&v_{(0,k)}(F^k \cdot g +w)\\
&=v_{(0,k)}((x^{i_k}y^{k}+y^{k+1} \cdot t(x,y)) \cdot g +w)\\
&=v_{(0,0)}(x^{i_k}\cdot g+y \cdot t(x,y) \cdot g+y \cdot l).
\end{align*}
Because $v_{(0,0)}(y \cdot t(x,y) \cdot g +y \cdot l)$ valuates to a pair whose first entry is $\infty$, by Lemma \ref{properties}
The first entry of $v_{(0,0)}(g \cdot x^{i_k}+y \cdot t(x,y)\cdot g +y \cdot l)$ must be the first entry of $v_{(0,0)}(g \cdot x^{i_k})= v_{(0,0)}(g)+ v_{(0,0)}(x^{i_k})$, which must be greater than or equal to $i_k$.
\end{proof}

\begin{notation}
\label{notik}
When the generator $F^k$ of $W_k$ has degree $i_k$, denote that $W_k$ has valuation type $[i_k]$.
\end{notation}

\subsection{Valuation of the Diagonal filtration}

\begin{definition}
Take a pair of integers $(k ,k)$, $0 \leq k \leq v-1$, and the Diagonal filtration $(J^{k})$ of an ideal $I \in \Hilb^n(\{x^uy^v=0\},0)$, then the Diagonal valuation $v_{(k,k)}$ is defined for all power series in $J^{k}$. Denote the image as $v_{(k,k)}(J^{k})$, which is a lattice in $(\mathbb{N} \cup \{\infty\})^2$.
\end{definition}

\begin{lemma}
\label{leminfty}
For any element $f(x,y)$ in $J^{k}$, the valuation $v_{s_{k'}}(f)=(\infty, \infty)$ for any $k' < k$.
\end{lemma}

\begin{proof}
$v_{s_{k'}}(f)=v_{(0,0)}(\frac{f}{x^{k'}y^{k'}})$. Because $f \in J^{k}$, $f$ is divisible by $x^ky^k$, and $\frac{f}{x^{k'}y^{k'}}$ is divisible by $x^{k-k'}y^{k-k'}$, so the dimension of $\frac{\mathbb{C}[\![x,y]\!]}{(\frac{f}{x^{k'}y^{k'}}, x)}$ and $\frac{\mathbb{C}[\![x,y]\!]}{(\frac{f}{x^{k'}y^{k'}}, y)}$ are both $\infty$.
\end{proof}

Now we introduce some notations and compute the non-reduced valuation explicitly for each $J^{k}$.

\begin{notation}
Assume $0 \leq k \leq v-1$. If there is one single additional generator of $J^{k}$,  
\begin{align*}
&E=x^ky^k\cdot(ax^{i_k}+by^{j_k})+x^{k+1}y^{k+1}\cdot h(x,y),\\
\end{align*}
then denote the lattice $v_{k}(J^{k})$ as $v_{k}(J^{k}):= [\circled{1}, (i_k,j_k)]$.

If there is a pair of additional generators of $J^{k}$,
\begin{align*}
&E_1=x^ky^k\cdot x^{i_k}+x^{k+1}y^{k+1}\cdot h_1(x,y),\\
&E_2=x^ky^k\cdot y^{j_k}+x^{k+1}y^{k+1}\cdot h_2(x,y),
\end{align*}
then denote the lattice $v_{k}(J^{k})$ as $v_{k}(J^{k}):= [\circled{2},(i_k,j_k)]$
\end{notation}

\begin{proposition}
\label{valuation}
Assume $0 \leq k \leq v-1$. 

The non-reduced valuation $[\circled{1}, (i_k,j_k)]$ is equal to the following lattice.
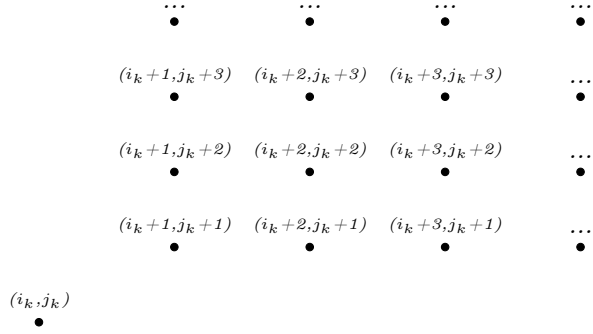
\begin{center}

\begin{tikzpicture}
\clip (-0.7,-0.2) rectangle (10cm,4.5cm); 
\foreach \x in {1,...,3}
    {
        \foreach \y in {1,...,3}
        {
        \node[draw,circle,inner sep=1pt,fill, label={\tiny($i_k$+\x ,$j_k$+\y )}] at (1.8*\x,\y) {};
        }
    }
\node[draw,circle,inner sep=1pt,fill,label={\tiny ($i_k$,$j_k$)}] at (0,0) {};
     
\foreach \x in {1,...,3}
    {\node[draw,circle,inner sep=1pt,fill,label={...}] at (1.8*\x,1*4) {};}

\foreach \y in {1,...,4}
    {\node[draw,circle,inner sep=1pt,fill,label={...}] at (1.8*4,\y) {};}
\end{tikzpicture}
    
\end{center}

The non-reduced valuation $[\circled{2},(i_k,j_k)]$ is equal to the following lattice.
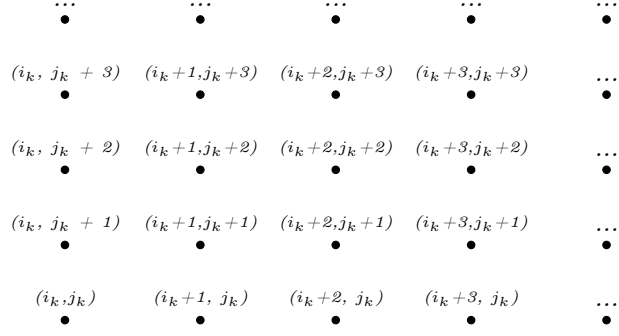
\begin{center}
 
\begin{tikzpicture}
\clip (-0.7,-0.2) rectangle (10cm,5cm); 
\foreach \x in {1,...,3}
    {
        \foreach \y in {1,...,3}
        {
        \node[draw,circle,inner sep=1pt,fill, label={\tiny($i_k$+\x ,$j_k$+\y )}] at (1.8*\x,\y) {};
        }
    }
\node[draw,circle,inner sep=1pt,fill,label={\tiny ($i_k$,$j_k$)}] at (0,0) {};
\foreach \x in {1,...,3}
    {\node[draw,circle,inner sep=1pt,fill,label={\tiny ($i_k$+\x , $j_k$)}] at (1.8*\x,0) {};}
\foreach \y in {1,...,3}
    {\node[draw,circle,inner sep=1pt,fill,label={\tiny ($i_k$, $j_k$ + \y )}] at (0,\y) {};}
     
\foreach \x in {0,...,3}
    {\node[draw,circle,inner sep=1pt,fill,label={...}] at (1.8*\x,1*4) {};}

\foreach \y in {0,...,4}
    {\node[draw,circle,inner sep=1pt,fill,label={...}] at (1.8*4,\y) {};}
\end{tikzpicture} 
\end{center}

\end{proposition}

\begin{proof}
\begin{enumerate}
\item When the additional generator of $J^{k}$ is $E$, then $(x^l+y^t)E$, where integers $l,t \geq 0$, valuates to all points in the lattice: by Lemma \ref{properties},
\begin{align*}
%&v_{k}(E)=(i_k, j_k).\\
&v_{k}((x^l+y^t)E)=v_{(0,0)}(x^l+y^t)+v_{k}(E)\\
&=(l,t)+(i_k, j_k)=(i_k+l, j_k+t).
\end{align*}

Next we show that there are no other points in the valuation $v_{k}(J^{k})$.

Recall that by the inductive construction of $J^{k}$, all generators of $J^{k}$ are either elements of $J^{k+1}$ or $E$. From Lemma \ref{leminfty}, every element $u \in J^{k+1}$ valuates to $(\infty, \infty)$ by all non-reduced valuation $v_{k'}$ for all $k' \leq t$.

Suppose that there exist some elements $e,f \in \mathbb{C}[\![x,y]\!]$, $u \in J^{k+1}$, and the generator of $J^{k}$ is $e \cdot E+u$.

By Lemma \ref{properties}, 
\begin{align*}
&v_{k}(e \cdot E+u)=\min(v_{(0,0)}(e)+v_{k}(E), v_{k}(u))\\
&=\min(v_{(0,0)}(e)+v_{k}(E), (\infty, \infty))\\
&=v_{(0,0)}(e)+v_{k}(E).
\end{align*} 

Because $v_{(0,0)}(e) \geq 0$, we immediately have that $v_{k}(e \cdot E+u) \geq v_{k}(E)$, so there are no points $(m_1, m_2) \in [\circled{1},(i_k, j_k)]$ such that $m_1 < i_k$ or $m_2 < j_k$.

Suppose for contradiction that $v_{k}(e \cdot E + u)= (i_k, j_k+1)$, then 
because $v_{k}(E)=(i_k, j_k)$, by the cancellation law of Lemma \ref{properties}, $v_{(0,0)}(e)=(0,1)$.

So $e$ must satisfy that $[e(x,y),x]= 0$ and that $[e(x,y),y]= 1$. The intersection multiplicity $[e(x,y),x]= 0$ implies that the vector space $\frac{\mathbb{C}[x,y]}{(e(x,y),x)}$ has dimension $0$, which implies that $1$ is in the ideal $(e(x,y),x)$, so we must have some $h(x,y)$ and $g(x,y)$, some positive integer $b$ such that $h(x,y)e(x,y)+g(x,y)x^b$ is a unit. All units in $C[\![x,y]\!]$ are power series with a non-zero constant term, so $e$ must have a non-zero constant term. So $e$ is a unit.

However, the intersection multiplicity of $e(x,y)$ must satisfy $[e(x,y), y] = 1$, which contradicts that $e$ is a unit. So there are no such $e(x,y)$ that valuates to $(0,1)$. 

With similar proof, if the additional generator of $J^{k}$ is $E$, there are no elements of $J^{k}$ that valuates to $(i_k, j_k+l)$ or $(i_k+l, j_k)$ for any integer $l \geq 1$. This completes the proof of (1).

\item When the additional generators of $J^{k}$ are $E_1,E_2$, then $x^lE_1+y^tE_2$, where integers $l,t \geq 0$, valuates to all points in the lattice.
\begin{align*}
&v_{k}(x^lE_1+y^tE_2)= \min (v_{(0,0)}(x^l)+v_{k}(E_1), v_{(0,0)}(y^t)+v_{k}(E_2))\\
&=\min ((l,\infty)+(i_k, \infty), (\infty, t) +(\infty, j_k))\\
&=\min ((l+i_k, \infty), (\infty, t+j_k))\\
&=(l+i_k,t+j_k).
\end{align*}
This proves the existence of all points in $[\circled{2},(i_k,j_k)]$.

Now we prove that there are no other points in $[\circled{2},(i_k,j_k)]$. For any element in $J^{k}$, it is generated as $e_1 \cdot E_1 + e_2 \cdot E_2 + u$, for some $e_1, e_2 \in \mathbb{C}[\![x,y]\!]$, and $u \in J^{k+1}$.
\begin{align*}
&v_{k}(e_1 \cdot E_1 + e_2 \cdot E_2 + u)=\\
&=\min(v_{(0,0)}(e_1)+v_{k}(E_1), v_{(0,0)}(e_2)+v_{k}(E_2), (\infty, \infty))\\
&\geq (i_k, j_k).
\end{align*}
This completes the proof of (2).
\end{enumerate}

\end{proof}

\section{Conditions on the generators}

Let $I$ be an ideal in $\Hilb^n(\{x^uy^v=0\},0)$. When finding the  generators added to construct each ideal $W^k$ (or $J^{k}$) from the previous ideal $W^{k+1}$ (or $J^{k+1}$) in the filtration of $I$, we emphasize on the principle that the generators of the previous layers, once chosen, are considered to be completely fixed in the inductive process, and new generator added should not bring in any extra constrains or rules on the generators in the previous layers. We introduced the generators of each layer in Lemma \ref{lemmagen}, now we discuss the conditions on the new generators added in the inductive process.

%Construct the filtration of $I \in \Hilb^n(\{x^uy^v=0\},0)$ and use the notations as before. 

\subsection{Vertical filtration}

\begin{lemma}
\label{condigenbig}
When $0 \leq k \leq N-1$, the generator $F$ of $W^{k}$ and the generator $F'$ of $W^{k+1}$ satisfy that $i \geq i'$, where \begin{align*}
&F=y^k \cdot x^i + y^{k+1} \cdot t(x,y),\\
&F'=y^{k+1} \cdot x^{i'} + y^{k+2} \cdot k'(x,y).
\end{align*}
 
\end{lemma}

\begin{proof}
Suppose for the purpose of contradiction that $i < i'$. Then \begin{align*}
y \cdot F =  y^{k+1} \cdot x^{i} + y^{k+2} \cdot t(x,y) 
\end{align*}
is an element in $W^{k+1}$ that cannot be generated by $F'$ together with the generators of $J^{s_{k+2}}$. Contradiction.

\end{proof}

\subsection{Diagonal filtration}

Now we study the conditions for $0 \leq k \leq v-1$, but first fix some notations.

\begin{notation}
Fix the following notation for the generators.

For $J^{k}$:
\begin{align*}
&E=x^ky^k\cdot(ax^{i_k}+by^{j_k})+x^{k+1}y^{k+1}\cdot h(x,y)\\
&E_1=x^ky^k\cdot x^{i_k}+x^{k+1}y^{k+1}\cdot h_1(x,y)\\
&E_2=x^ky^k\cdot y^{j_k}+x^{k+1}y^{k+1}\cdot h_2(x,y)
\end{align*}
And for $J^{k+1}$:
\begin{align*}
&E'=x^{k+1}y^{k+1}\cdot(a'x^{i_{k+1}}+b'y^{j_{k+1}})+x^{k+2}y^{k+2}\cdot h'(x,y).\\
&E_1'=x^{k+1}y^{k+1}\cdot x^{i_{k+1}}+x^{k+2}y^{k+2}\cdot h_1'(x,y)\\
&E_2'=x^{k+1}y^{k+1}\cdot y^{j_{k+1}}+x^{k+2}y^{k+2}\cdot h_2'(x,y) 
\end{align*}
\end{notation}

\begin{proposition}
\label{propij}
When $0 \leq k \leq v-1$, the following holds for two consecutive ideals $J^{k}$ and $J^{k+1}$ in the filtration.
\begin{enumerate}
\item When we have $[\circled{1}, (i_k,j_k)], [\circled{1},(i_{k+1},j_{k+1})]$, then 
\begin{align*}
&\text{either }(i_k,j_k)=(i_{k+1},j_{k+1}),\\
& \text{or } (i_k,j_k) \geq (i_{k+1}+1,j_{k+1}+1).
\end{align*}

\item When we have $[\circled{2},(i_k,j_k)], [\circled{2},(i_{k+1},j_{k+1})]$, then $$(i_k,j_k) \geq (i_{k+1}+1,j_{k+1}+1).$$ 
\item When we have $[\circled{1}, (i_k,j_k)], [\circled{2},(i_{k+1},j_{k+1})]$, then $$(i_k,j_k) \geq (i_{k+1},j_{k+1}).$$ 
\item When we have $[\circled{2},(i_k,j_k)], [\circled{1},(i_{k+1},j_{k+1})]$, then $$(i_k,j_k) \geq (i_{k+1}+1,j_{k+1}+1).$$ 
\end{enumerate}

\end{proposition}

\begin{proof}
For the generators $E, E_1, E_2$ in $J^{k}$, the following elements are in $J^{k+1}$
\begin{align*}
&xy \cdot E=x^{k+1}y^{k+1}\cdot(ax^{i_k}+by^{j_k})+x^{k+2}y^{k+2}\cdot h(x,y),\\
&y \cdot E_1=x^{k+1}y^{k+1}\cdot x^{i_k-1}+x^{k+1}y^{k+1}\cdot y \cdot h_1(x,y),\\
&x \cdot E_2=x^{k+1}y^{k+1}\cdot y^{j_k-1}+x^{k+1}y^{k+1}\cdot x \cdot h_2(x,y),
\end{align*}
and their valuations are 
\begin{align*}
&v_{k+1}(xy \cdot E)=(i_k, j_k). \\
&v_{k+1}(y \cdot E_1+x \cdot E_2) \leq (i_k-1, j_k-1).
\end{align*}
To not contradict with the valuations $[\circled{1},(i_{k+1},j_{k+1})]$
and $[\circled{2},(i_{k+1},j_{k+1})]$ as computed in Lemma \ref{valuation}, or in other words to not add new points in the valuation of $J^{k+1}$, we immediately prove (2),(3) and (4). And for (1), we have $(i_k,j_k) \geq (i_{k+1},j_{k+1})$.

%To finish the proof for (1), assume the additional generators for $J^{k}$ and $J^{k+1}$ are denoted as follows
%\begin{align*}
%&E=x^ky^k\cdot(ax^{i_k}+by^{j_k})+x^{k+1}y^{k+1}\cdot h(x,y)\\
%&E'=x^{k+1}y^{k+1}\cdot(a'x^{i_{k+1}}+b'y^{j_{k+1}})+x^{k+2}y^{k+2}\cdot h'(x,y).
%\end{align*}
Suppose that $i_k=i_{k+1}$, then compute
\begin{align*}
&aE'-a'xy \cdot E\\
&=(x^{k+1}y^{k+1}\cdot(aa'x^{i_k}+ab'y^{j_{k+1}})+x^{k+2}y^{k+2}\cdot ah'(x,y))\\
&-(x^{k+1}y^{k+1}\cdot(aa'x^{i_k}+a'by^{j_k})+x^{k+2}y^{k+2}\cdot a'h(x,y))\\
&=x^{k+1}y^{k+1}(ab'y^{j_{k+1}}-a'by^{j_k})+x^{k+2}y^{k+2}\cdot (ah'(x,y)-a'h(x,y)).
\end{align*}
Unless both $j_k=j_{k+1}$ and $ab'=a'b$, we have the following contradiction. 

Assume either $j_k \neq j_{k+1}$ or $ab' \neq a'b$, then together with the generator $E'$, $x^l$ for some positive integer $l$, $J^{k+1}$ has the following points in the valuation 
\begin{align*}
&v_{k+1}(x^lE'+(aE'-a'xy \cdot E))=\\
&=\min(v_{k+1}(x^lE'),v_{k+1}(aE'-a'xy \cdot E)\\
&=\min( (i_k+l,\infty), (\infty, j_k))\\
&=(i_k+l, j_k),
\end{align*}
which is a contradiction to $[\circled{1},(i_{k+1},j_{k+1})]$.

So when $i_k=i_{k+1}$, we must also have $j_k=j_{k+1}$, and that $\frac{a}{b}=\frac{a'}{b'}$.
\end{proof}

In the next proposition, we carefully discuss the boundary cases in Proposition \ref{propij}, to show that equality can be achieved.

\begin{proposition}

\label{prop=}
\begin{enumerate}
\item Assume we have $[\circled{1}, (i_k,j_k)], [\circled{1},(i_{k+1},j_{k+1})]$.

When $(i_k,j_k)=(i_{k+1},j_{k+1}),$
then $$\frac{a}{b}=\frac{a'}{b'}.$$

When 
$(i_k,j_k) = (i_{k+1}+1,j_{k+1}+1)$, there are no extra conditions on the generators.

\item When $[\circled{2},(i_k,j_k)], [\circled{2},(i_{k+1},j_{k+1})]$, and $(i_k,j_k) =(i_{k+1}+1,j_{k+1}+1),$ there are no extra conditions on the generators.

\item When $[\circled{1}, (i_k,j_k)], [\circled{2},(i_{k+1},j_{k+1})]$, and $(i_k,j_k) = (i_{k+1},j_{k+1}),$ there are no extra conditions on the generators.
\item When $[\circled{2},(i_k,j_k)], [\circled{1},(i_{k+1},j_{k+1})]$, and $i_k = i_{k+1}+1$, then 
$$b'y^{j_{k+1}}-a'yh_1(x,y)=0 \text{ modulo } xy;$$
or $j_k=j_{k+1}+1$, then $$a'x^{i_{k+1}}-b'xh_2(x,y)=0 \text{ modulo } xy.$$

\end{enumerate}
\end{proposition}

\begin{proof}
\begin{enumerate}
\item We already proved case (1), that $(i_k,j_k)=(i_{k+1},j_{k+1}),$
implies $\frac{a}{b}=\frac{a'}{b'}.$ Now when $(i_k,j_k)=(i_{k+1}+1,j_{k+1}+1)$, $xy \cdot E$ is in the ideal generated by $E'$ and $J^{s_{k+2}}$ for all $a,b,a',b',h(x,y),h'(x,y)$, so there are no extra conditions.

\item Because $y \cdot E_1$ and $x \cdot E_2$ are in the ideal generated by $E_1', E_2'$ and $J^{s_{k+2}}$, so  there are no extra conditions.

\item Because $xy \cdot E$ is in the ideal generated by $E_1', E_2'$ and $J^{s_{k+2}}$, so  there are no extra conditions.

\item
Assume that $i_k = i_{k+1}+1$ and $j_k \geq j_{k+1}+1$.
The elements 
\begin{align*}
&y \cdot E_1=x^{k+1}y^{k+1}\cdot x^{i_k-1}+x^{k+1}y^{k+1}\cdot y \cdot h_1(x,y)\\
&x \cdot E_2=x^{k+1}y^{k+1}\cdot y^{j_k-1}+x^{k+1}y^{k+1}\cdot x \cdot h_2(x,y),
\end{align*}
are in $J^{k+1}$, and are generated by $E'$ along with elements of $J^{s_{k+2}}$.

Compute 
\begin{align*}
&E'-a'y \cdot E_1\\
&=x^{k+1}y^{k+1}\cdot (b'y^{j_{k+1}}-a'yh_1(x,y))+x^{k+2}y^{k+2}\cdot h'(x,y)\\
\end{align*}

Denote the valuation $$v_{(0,0)}(yh_1(x,y)):=(\infty,m)$$ where $m$ is either some positive integer or $\infty$.
Compute
\begin{align*}
v_{k}(E'-a'y \cdot E_1)=v_{(0,0)}(b'y^{j_{k+1}}-a'yh_1(x,y)) = (\infty, \min(j_{k+1},m)), 
\end{align*}
unless $$b'y^{j_{k+1}}-a'yh_1(x,y)=0 \text{ modulo } xy.$$

So there are elements $x^lE'+(E'-a'y \cdot E_1)$ for all positive integers $l$ in $J^{k+1}$ that valuates to  
\begin{align*}
&v_{k+1}(x^lE'+(E'-a'y \cdot E_1))\\
&=\min ( (i_{k+1}+l, \infty), (\infty, \min(j_{k+1},m)))\\
&= (i_k+l, \min(j_{k+1},m)).   
\end{align*}
This contradicts that $v_{k}(J^{k})=[\circled{1},(i_{k+1},j_{k+1})]$ and should not contain points $(i_k+l, \min(j_{k+1},m))$ for all positive integers $l$. So we must have that 
$$b'y^{j_{k+1}}-a'yh_1(x,y)=0 \text{ modulo } xy.$$
If we write the power series $$h_1(x,y)=\sum_{w_1,w_2=0}^{\infty} \limits c_{w_1,w_2}x^{w_1}y^{w_2},$$
then the coefficients $c_{w_1,w_2}$ are completely determined when $w_1=0$, and free otherwise.

By the same reasoning, if 
$j_k = j_{k+1}$, we also must have that $$a'x^{i_{k+1}}-b'xh_2(x,y)=0 \text{ modulo } xy.$$ 

If we write the power series $$h_2(x,y)=\sum_{w_1,w_2=0}^{\infty} \limits d_{w_1,w_2}x^{w_1}y^{w_2},$$
then the coefficients $d_{w_1,w_2}$ are completely determined when $w_2=0$, and free otherwise. 
\end{enumerate}
\end{proof}

\section{Stratification of the Hilbert scheme}

\subsection{Vertical stratification}

\begin{lemma}
Take the Vertical valuations $v_{(0,k)}$, $k=0,...,N-1$ of the Vertical filtration $W^k$ of the Hilbert scheme $\Hilb^{n}(\{x^uy^v=0\},0)$, and we obtain a sequence of numbers $(i_k)$ where $v_{(0,k)}(W^k)=[i_k]$, such that $i_k \geq i_{k+1}$. We have that $$\sum_{k=0}^{N-1}i_k=n.$$
\end{lemma}

\begin{proof}
When $0 \leq k \leq N-1$, and $v_{k}(J^{k}) = [i_k]$, we want to show that the vector space $\frac{\Theta^{k}}{\Theta^{k+1}}$ is spanned over $\mathbb{C}$ by the basis $$B^k=\{y^{k}, xy^k, x^2y^k, ..., x^{i_k-1}y^k \}.$$  

Suppose for contradiction that $B^k$ is not linearly independent, then there exists some element $g$ in $I \cap (y^k)$ that has at least one of the basis elements as a summand, then $v_{(0,k)}(g)$ is some point whose $x$ coordinate is smaller than $i_k$, contradicting with Lemma \ref{lemik}. If $v_{k}(J^{k})=[i_k]$, then use the generator $F$ to write all monomials $m \in \Theta^{k}$ that are not in $B^k$ as a linear combination of elements in $\Theta^{k+1}$, so every element in the vector space $\frac{\Theta^k}{\Theta^{k+1}}$ can be written as a linear combination of $B^k$. So $B^k$ is a basis of $\frac{\Theta^{k}}{\Theta^{k+1}}$.

Because the vector space $\frac{\Theta^{k}}{\Theta^{k+1}}$ has dimension $i_k$, Lemma \ref{lemvectorfilt} implies that 

$\sum_{k=0}^{N-1}i_k=n.$
\end{proof}

\begin{corollary}
The Vertical valuation $v_{(0,k)}$ maps any ideal $I \in \Hilb^{n}(\{x^uy^v=0\},0)$ and its vertical filtration $W^k$ to a length $N$ partition $(i_k)$ of $n$.
\end{corollary}

\begin{definition}
To define the Vertical stratification of the Hilbert scheme $\Hilb^{n}(\{x^uy^v=0\},0)$, we send any ideal $I \in \Hilb^{n}(\{x^uy^v=0\},0)$ to its Vertical valuation. Collect all ideals $I \in \Hilb^{n}(\{x^uy^v=0\},0)$ whose image is a partition $(i_k)$ of $n$ into a stratum indexed by $(i_k)$.
\end{definition}

The Vertical stratification of the Hilbert scheme $\Hilb^{n}(\{x^uy^v=0\},0)$ coincides with the Bialynicki-Birula stratification. In particular, it can be obtained from the torus action that puts much greater weight on the variable $y$ than the variable $x$. 
\subsection{Diagonal stratification. \\}

The Diagonal stratification is a new stratification of the Hilbert scheme of points, not coming directly from intersecting with any know stratification of the Hilbert scheme.

Before the definition of the Diagonal stratification, we define a new type of diagonal partition that indexes the Diagonal stratification.

Unlike the usual partition with vertical columns and horizontal rows, the \textbf{columns} of the diagonal partitions expand in the diagonal direction, marked by the \textcolor{blue}{blue} arrows in the picture, and the \textbf{rows} are expanding both in the $x$ and $y$ direction, marked by the \textcolor{red}{red} arrows. The \textbf{central column} is marked by the \textbf{\textcolor{blue}{thick blue}} arrow in the middle. Denote the horizontal part of a row as the \textbf{horizontal arm}, and denote the vertical part of a row as the \textbf{vertical arm}. The horizontal and vertical arms of a row has a common intersection, the box in the central column, denoted as \textbf{the central box}.

The rows are also denoted interchangeably as \textbf{layers}. Denote the layer with the central box $1$ as the \textbf{outmost layer}, and the layer with the central box $x^{v-1}y^{v-1}$ as the \textbf{innermost layer}. The outmost layer is also denoted as the \textbf{0th layer}, and moving up, we have the \textbf{1st layer, 2nd layer, ...,} up until the innermost layer, which is the \textbf{(v-1)th layer}.

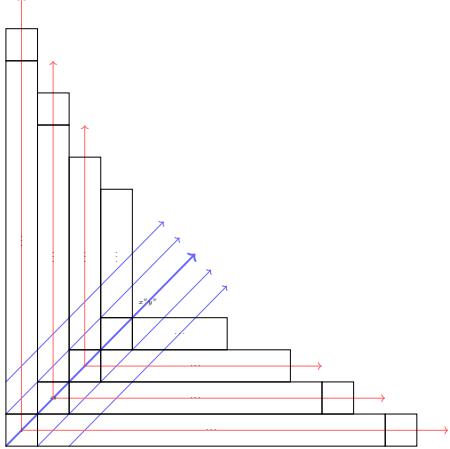
\resizebox{6cm}{6cm}{
\begin{tikzpicture}
%begining layers

%centural column
\draw[->, blue!60, line width=0.6mm] (-1,-1) -- (5,5); 

%right side
\draw[->, blue!60, thick] (-1,0) -- (4.5,5.5);
\draw[->, blue!60, thick] (-1,1) -- (4,6);

%left side
\draw[->, blue!60, thick] (0,-1) -- (5.5,4.5);

\draw[->, blue!60, thick] (1,-1) -- (6,4);

%rows
\draw[->, red!60, thick] (-0.5,-0.5) -- (13,-0.5);
\draw[->, red!60, thick] (0.5,0.5) -- (11,0.5);
\draw[->, red!60, thick] (1.5,1.5) -- (9,1.5);

%vertical rows
\draw[->, red!60, thick] (-0.5,-0.5) -- (-0.5,13);
\draw[->, red!60, thick] (0.5,0.5) -- (0.5,11);
\draw[->, red!60, thick] (1.5,1.5) -- (1.5,9);

\node[cell] at (-1,-1) {\scalebox{.5}{$1$}};
\node[cell, minimum height=11cm] at (-1,0) {\tiny $\vdots$};
\node[cell, minimum width=11cm] at (0,-1) {\tiny $\ldots$};
\node[cell] at (11,-1) {};
\node[cell] at (-1,11) {};

%first/ k-1 th layer
\node[cell] at (0,0) {\scalebox{.5}{$xy$}};
\node[cell, minimum height=8cm] at (0,1) {\tiny $\vdots$};
\node[cell, minimum width=8cm] at (1,0) {\tiny $\ldots$};
\node[cell] at (9,0) {};
\node[cell] at (0,9) {};

%second layer

\node[cell] at (1,1) {\scalebox{.5}{$...$}};
\node[cell, minimum height=6cm] at (1,2) {\tiny $\vdots$};
\node[cell, minimum width=6cm] at (2,1) {\tiny $\ldots$};
%\node[cell, fill=lime!50] at (7,1) {};
%\node[cell, fill=lime!50] at (1,7) {};

%third layer
\node[cell] at (2,2) {};
\node[cell, minimum height=4cm] at (2,3) {\tiny $\vdots$};
\node[cell, minimum width=3cm] at (3,2) {\tiny $\ldots$};
%\node[cell, fill=lime!50] at (6,2) {};
%\node[cell, fill=lime!50] at (2,7) {};

\node[cell, draw=none] at (3,3) {\tiny $x^v y^v$};

\end{tikzpicture}
}

Just like the usual partition where the number of boxes are only allowed to weakly increase, the number of boxes in the following partition we define must also respect some weakly increasing rules.

\begin{definition}
\label{wddp}
The \textbf{weak diagonal partition} has two types of rows, type $\circled{1}$ and type $\circled{2}$, and two type of boxes, solid boxes, and shaded boxes.

A type $\circled{1}$ row has exactly two shaded boxes at each end of the horizontal and vertical arms. And all the rest of the boxes in the middle are solid boxes. %At the $k$th row of type $\circled{1}$, when the shaded box on the horizontal arm of  is $i_k$ boxes away from the central box, and the shaded box on the vertical arm is $j_k$ boxes away from the central box, denote this row as $[\circled{1}, (i_k,j_k)]$.
%or $[\circled{2},(i_k,j_k)]$, where $(i_k,j_k)$ counts total number of boxes of each row, including the shaded boxes for type $\circled{1}$, and the empty boxes for type $\circled{2}$.
A type $\circled{2}$ row has solid boxes only. %To incorporate the ideal generators into our notation, we (artificially) attach two empty boxes at the end of each horizontal and vertical arms of type $\circled{2}$, denoted by boxes with the $\emptyset$ notation. 

At the $k$th row, when the horizontal arm has $i_k$ boxes, and the vertical arm has $j_k$ boxes, (counting both the solid and shaded boxes for type $\circled{1}$ rows),  denote this row as $[\circled{1}, (i_k,j_k)]$ or $[\circled{2}, (i_k,j_k)]$ correspondingly.

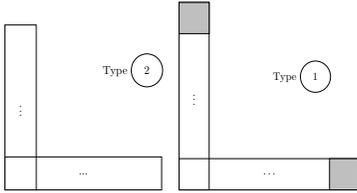
\resizebox{2.2cm}{2.2cm}{
\begin{tikzpicture}
%rows
\node[cell, minimum height=5cm] at (0,0) {$\vdots$};
\node[cell, minimum width=5cm] at (0,0) {$...$};

%ending boxes
%\node[cell] at (5,0) {};
%\node[cell] at (0,5) {};
\node[cell, draw=none] at (3,3) {Type $\circled{2}$};

\end{tikzpicture}}
\resizebox{2.5cm}{2.5cm}{
\begin{tikzpicture}

%rows
\node[cell, minimum height=6cm] at (0,0) {$\vdots$};
\node[cell, minimum width=6cm] at (0,0) {$\ldots$};

%ending boxes
\node[cell, fill=gray!50] at (5,0) {};
\node[cell, fill=gray!50] at (0,5) {};

\node[cell, draw=none] at (3,3) {Type $\circled{1}$};
\end{tikzpicture}
}
\begin{comment}
\resizebox{3.5cm}{2.5cm}{
\begin{tikzpicture}
%rows
\node[cell, minimum height=6cm] at (0,0) {$\vdots$};
\node[cell, minimum width=6cm] at (0,0) {$...$};

%ending boxes
\node[cell] at (5,0) {$\emptyset$};
\node[cell] at (0,5) {$\emptyset$};
\node[cell, draw=none] at (3,3) {Type $\circled{2}$ with generators marked};

\end{tikzpicture}}
\end{comment}

Define the weak diagonal partition to be the diagonal partition whose rows are of type either $\circled{1}$ or $\circled{2}$, and that the number of boxes $(i_k,j_k)$ between two consecutive rows satisfy the rules in Proposition \ref{prop=}. 

To make it easier for the reader, we restate the rules by drawing the following picture indicating the boundary cases (the minimum number of boxes that the outter layer can have) of the rules in Proposition \ref{prop=}. 

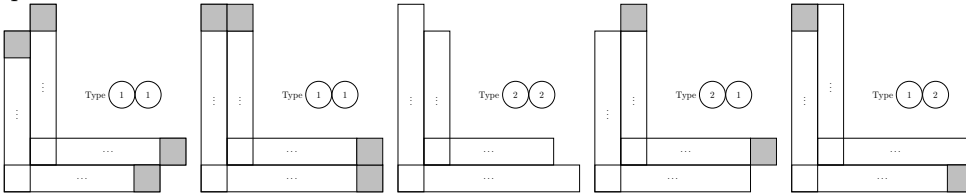
\resizebox{2.5cm}{2.5cm}{
\begin{tikzpicture}

%outter rows
\node[cell, minimum height=6cm] at (0,0) {$\vdots$};
\node[cell, minimum width=6cm] at (0,0) {$\ldots$};

%inner rows
\node[cell, minimum height=6cm] at (1,1) {$\vdots$};
\node[cell, minimum width=6cm] at (1,1) {$\ldots$};

%ending boxes - outter
\node[cell, fill=gray!50] at (5,0) {};
\node[cell, fill=gray!50] at (0,5) {};

%ending boxes - inner
\node[cell, fill=gray!50] at (6,1) {};
\node[cell, fill=gray!50] at (1,6) {};

\node[cell, draw=none] at (3,3) {Type $\circled{1}\circled{1}$};
\end{tikzpicture}
}
\resizebox{2.5cm}{2.5cm}{
\begin{tikzpicture}

%outter rows
\node[cell, minimum height=7cm] at (0,0) {$\vdots$};
\node[cell, minimum width=7cm] at (0,0) {$\ldots$};

%inner rows
\node[cell, minimum height=5cm] at (1,1) {$\vdots$};
\node[cell, minimum width=5cm] at (1,1) {$\ldots$};

%ending boxes - outter
\node[cell, fill=gray!50] at (6,0) {};
\node[cell, fill=gray!50] at (0,6) {};

%ending boxes - inner
\node[cell, fill=gray!50] at (6,1) {};
\node[cell, fill=gray!50] at (1,6) {};

\node[cell, draw=none] at (3,3) {Type $\circled{1}\circled{1}$};
\end{tikzpicture}
}
\resizebox{2.5cm}{2.5cm}{
\begin{tikzpicture}

%outter rows
\node[cell, minimum height=7cm] at (0,0) {$\vdots$};
\node[cell, minimum width=7cm] at (0,0) {$\ldots$};

%inner rows
\node[cell, minimum height=5cm] at (1,1) {$\vdots$};
\node[cell, minimum width=5cm] at (1,1) {$\ldots$};

%ending boxes - outter
%\node[cell] at (6,0) {$\emptyset$};
%\node[cell] at (0,6) {$\emptyset$};

%ending boxes - inner
%\node[cell] at (6,1) {$\emptyset$};
%\node[cell] at (1,6) {$\emptyset$};

\node[cell, draw=none] at (3,3) {Type $\circled{2}\circled{2}$};
\end{tikzpicture}
}
\resizebox{2.5cm}{2.5cm}{
\begin{tikzpicture}

%outter rows
\node[cell, minimum height=6cm] at (0,0) {$\vdots$};
\node[cell, minimum width=6cm] at (0,0) {$\ldots$};

%inner rows
\node[cell, minimum height=5cm] at (1,1) {$\vdots$};
\node[cell, minimum width=5cm] at (1,1) {$\ldots$};

%ending boxes - outter
%\node[cell] at (6,0) {$\emptyset$};
%\node[cell] at (0,6) {$\emptyset$};

%ending boxes - inner
\node[cell, fill=gray!50] at (6,1) {};
\node[cell, fill=gray!50] at (1,6) {};

\node[cell, draw=none] at (3,3) {Type $\circled{2}\circled{1}$};
\end{tikzpicture}
}
\resizebox{2.5cm}{2.5cm}{
\begin{tikzpicture}

%outter rows
\node[cell, minimum height=7cm] at (0,0) {$\vdots$};
\node[cell, minimum width=7cm] at (0,0) {$\ldots$};

%inner rows
\node[cell, minimum height=6cm] at (1,1) {$\vdots$};
\node[cell, minimum width=6cm] at (1,1) {$\ldots$};

%ending boxes - outter
\node[cell, fill=gray!50] at (6,0) {};
\node[cell, fill=gray!50] at (0,6) {};

%ending boxes - inner
%\node[cell] at (7,1) {$\emptyset$};
%\node[cell] at (1,7) {$\emptyset$};

\node[cell, draw=none] at (3,3) {Type $\circled{1}\circled{2}$};
\end{tikzpicture}
}

Finally, a weak diagonal partition of $v$ rows is uniquely denoted by the following sequence, where $k=0,..., v-1$, and each $\circled{s}_k$, indicating the types of the rows, is either $\circled{1}$ or $\circled{2}$,  $$([\circled{s}_k, (i_k, j_k)])= [\circled{s}_0, (i_0, j_0)], [\circled{s}_1, (i_1, j_1)], ..., [\circled{s}_{v-1}, (i_{v-1}, j_{v-1})].$$
\end{definition}

\begin{definition}
Define the total number of boxes $n$ of a weak diagonal partition as the sum of all solid boxes, plus the number of shaded boxes divided by $2$. %The empty boxes do not contribute to the total number of boxes in the diagonal partition.

In other words, for a weak diagonal partition, the total number of boxes is 
$$n=\sum_{k=0}^{v-1}{\Delta_k},$$ where ${\Delta_k}=i_k+j_k$ if $\circled{s}_k=\circled{1}$, and ${\Delta_k}=i_k+j_k-1$ if $\circled{s}_k=\circled{2}$.
\end{definition}

\begin{notation}
\label{1U1}
If two consecutive rows of type $\circled{1}$, denoted by $[\circled{1}, (i_k, j_k)]$ and $[\circled{1}, (i_{k-1}, j_{k-1})]$ satisfy that $i_k= i_{k-1}, j_k=j_{k-1}$, then denote these type of transition as $\circled{1}_E\circled{1}$, $E$ as in the number of boxes remain equal. 

If two consecutive rows of type $\circled{1}$ satisfy that $i_{k-1} \geq i_k+1 , j_{k-1} \geq j_k +1$, then denote these type of transition as $\circled{1}_U\circled{1}$, $U$ as in the number of boxes going up.
\end{notation}

\begin{lemma}
\label{codimlemma}
Every ideal $I \in \Hilb^n(\{x^uy^v=0\},0)$, with its Diagonal filtration, maps by the Diagonal valuations $v_{(0,0)}, v_{(1,1)}, ..., v_{(v-1,v-1)}$ to a weak diagonal partition of a total of $n$ boxes. 
\end{lemma}

\begin{proof}
By construction, every ideal $I \in \Hilb^n(\{x^uy^v=0\},0)$ maps to a weak diagonal partition, given by $v_{(k,k)}(J^k)$ for all $k=0,...,v-1$.

The only nontrivial thing left show that the dimension of the vector space $\frac{\mathbb{C}[\![x,y]\!]}{I}$ is equal to the number of the boxes of the corresponding weak diagonal partition.

Let $(\Gamma^{k})$ be the Diagonal filtration of the vector space $\frac{\mathbb{C}[\![x,y]\!]}{I}$. Proving this lemma is equivalent to showing the following statement.

Let $0 \leq k \leq v-1$,
\begin{enumerate}
 
\item 
when $v_{(k,k)}(J^{k})=[\circled{1},(i_k,j_k)]$,
$$\dim(\frac{\Gamma^{k}}{\Gamma^{k+1}})=i_k + j_k.$$ 

\item when $v_{(k,k)}(J^{k})=[\circled{2},(i_k, j_k)]$, $$\dim(\frac{\Gamma^{k}}{\Gamma^{k+1}})=i_k + j_k-1.$$
\end{enumerate}

And we prove the statement above by finding the explicit basis of each vector space. 

When $0 \leq k\leq v-1$, and $v_{(k,k)}(J^{k})=[\circled{1},(i_k,j_k)]$, we claim that the vector space $\frac{\Gamma^{k}}{\Gamma^{k+1}}$ is spanned over $\mathbb{C}$ by the following basis $B_1$. Because we have the generator $E=x^ky^k(ax^{i_k}+by^{j_k})+... \in J^k$, we write either one of $x^{i_k}$ or $y^{j_k}$ in terms of the other, and pick any one of them as the basis. We pick $x^{i_k}$ in this convention.
\begin{align*}
B_1:=&\{x^ky^k,\\ &x^ky^k \cdot x, ..., x^ky^k \cdot x^{i_k},\\
&x^ky^k \cdot y, ..., x^ky^k \cdot y^{j_k-1}\}.\\
\end{align*}

When $0 \leq k \leq v-1$, and $v_{k}(J^{k}) = [\circled{2}, (i_{k},j_{k})]$, we claim that the vector space $\frac{\Gamma^{k}}{\Gamma^{k+1}}$ is spanned over $\mathbb{C}$ by the basis 
\begin{align*}
B_2:=&\{x^ky^k,\\ &x^ky^k \cdot x, ..., x^ky^k \cdot x^{i_k-1},\\
&x^ky^k \cdot y, ..., x^ky^k \cdot y^{j_k-1}\}.\\
\end{align*}

From the Diagonal valuation, we now conclude by the method of contradiction that the monomials in $B_1,B_2$ are linearly independent. Suppose not, then there must be some element $g$ in $I \cap (x^ky^k)$ that has one of the basis elements as a summand, and this element $g$ would valuates to something smaller than $(i_k, j_k)$ when $0 \leq k \leq v-1$, contradiction with the valuation of each $J^{k}$. 

Now we show that every element in the vector space $\frac{\Gamma^{k}}{\Gamma^{k+1}}$ can be written as a linear combination of the basis. 

Assume that $v_{k}(J^{k}) = [\circled{1}, (i_{k},j_{k})]$, because the ideal $J^{k}$ has the generator $$E=x^ky^k(ax^{i_k}+by^{j_k})+x^{k+1}y^{k+1}h_1(x,y),$$ then $J^{k}$ also contains $$f_1(x,y)=x^ky^k \cdot x^{i_k+1} +x^{k+1}y^{k+1}h_1'(x,y),$$
$$f_2(x,y)=x^ky^k \cdot y^{j_k+1} +x^{k+1}y^{k+1}h_1''(x,y),$$ obtained from multiplying $E$ by $y$ and $x$. 

Take an element $\gamma$ from $\frac{\Gamma^{k}}{\Gamma^{k+1}}$, suppose for the purpose of contradiction that $\gamma$ cannot be written as a linear combination of the basis $B_1$, then $\gamma$ must have the summand of some monomials $m$ of the form $x^ky^{k+l}$ where $l \geq j_k+1$, or of the form $x^{k+r}y^k$ where $r \geq i_k+1$. (Note that these are all the forms that $m$ can take: if a monomial $m$ is divisible by $x^{k+1}y^{k+1}$, then it is either in $J^{k+1}$ or $\Gamma^{k+1}$ by construction, so $\gamma$ cannot have any summand of monomials that are divisible by $x^{k+1}y^{k+1}$.)

However, we know from the elements $f_2(x,y)$ and $f_3(x,y)$ of $J^{k}$ that any monomial $x^ky^l$ where $l \geq i_k+1$, or any monomial $x^ry^k$ where $r \geq j_k
+1$ can be written as a linear combination elements in $\Gamma^{k+1}$. For the monomial $x^ky^k\cdot y^{j_k}$, we can always rewrite it as some copies of $x^ky^k\cdot x^{i_k}$ by the generator $E$. 

So indeed every element in the vector space $\frac{\Gamma^{k}}{\Gamma^{k+1}}$ can be written as a linear combination of $B_1$.

Similarly, if $v_{k}(J^{k})=[\circled{2},(i_k,j_k)]$, then use the generators $E_1, E_2$ to write all monomials $m \in \Gamma^{k}$ that are not in $B_2$ as a linear combination of elements in $\Gamma^{k+1}$, so every element in the vector space $\frac{\Gamma^k}{\Gamma^{k+1}}$ can be written as a linear combination of $B_2$.

\end{proof}

\begin{definition}
The Diagonal stratification of the Hilbert scheme $\Hilb^n(\{x^uy^v=0\},0)$ is given as follows.

For any ideal $I \in \Hilb^n(\{x^uy^v=0\},0)$, under the non-reduced valuation $v_{(k,k)}$ for $k=0,...,v-1$, we obtain a sequence of $v$ lattice of type either $[\circled{1},(i_k,j_k)]$ or $[\circled{2},(i_k,j_k)]$. Collect all ideals $I$ which valuates to the same weak diagonal valuation $([\circled{s}_k, (i_k,j_k)])$ of $n$ boxes into a stratum, denoted by $$\Sigma_{([\circled{s}_k, (i_k,j_k)])}.$$
\end{definition}

\section{Free variables of generators}

Recall that the generators of the Diagonal filtration $(J^{k})$ has $3$ types, and the generators of the Vertical filtration has $1$ type. We define some notations for the generators $F$ of $(W^k)$, and $E, E_1, E_2$ of $(J^{k})$.

\begin{definition}
For each additional generator of one of the types, $F$ of $W^k$, or $E, E_1,E_2$ of $J^{k}$, we perform the following elimination process.

For the monomial in each generator $F, E, E_1, E_2$, denote the monomials that valuates to the minimal pair of integers (by the corresponding Vertical or Diagonal valuations) as the Head part of the generators.

For the generator of layer $W^k$, use the generators of its inner layers $W^\tau$ for all $\tau > k$ to inductively eliminate, in the  order from $\tau=N-1$ to $\tau=k+1$, the infinite series $t(x,y)$ to a finite sum, such that the summand are monomials between the Head parts of the generators of all $W^{\tau}$, or equivalently, the monomials above the $W^k$ layer inside the corresponding weak diagonal partition. We abuse the notation and use $t(x,y)$ to also denote the finite sum of monomials, as the Tail part of the generator $F$. Similarly, perform the same inductive elimination process for the infinite series $h(x,y)$ (or $ h_1(x,y), h_2(x,y)$) of the generators $E$ (or $E_1, E_2$) of $J^{k}$, in the order from $\tau=v-1$ to $\tau=k+1$, and reduce   $h(x,y)$ (or $ h_1(x,y), h_2(x,y)$) to a finite sum, such that the summand are monomials between the Head parts of the generators of $J^{\tau}$ for $\tau=k+1, ..., v-1$.

Mark the Head part by red, and the Tail part by blue.

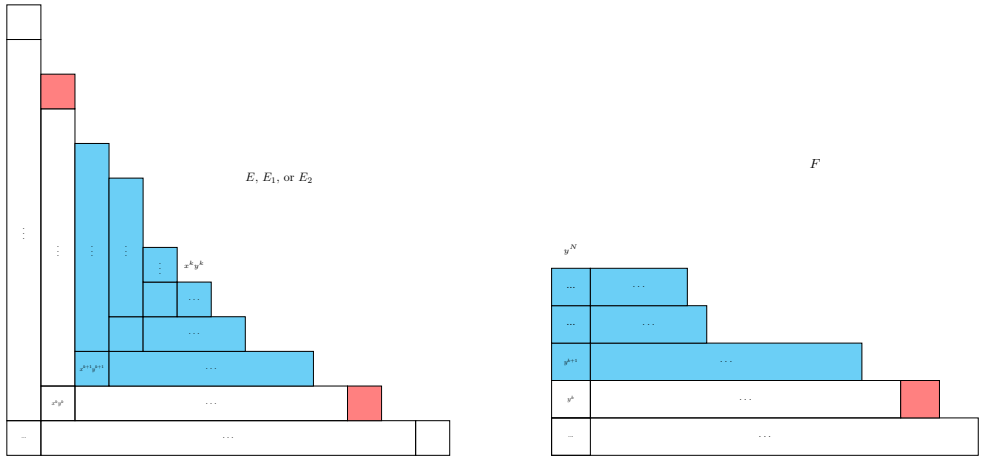
\begin{center}
\begin{tabular}{|c|c|c|}
 \hline
 Generator & Head & Tail \\ 
 \hline
 $F$ & $\color{red!50} y^k \cdot x^i$ & $\color{cyan!60!black} y^{k+1} \cdot t(x,y)$ \\
 \hline
 $E$ & \color{red!50} $x^k y^k\cdot(ax^{i_k}+by^{j_k})$ & $\color{cyan!60!black} x^{k+1}y^{k+1}\cdot h(x,y)$ \\ 
 \hline
 $E_1$ & $\color{red!50} x^ky^k\cdot x^{i_k}$ & $\color{cyan!60!black} x^{k+1}y^{k+1}\cdot h_1(x,y)$ \\ 
 \hline
 $E_2$ & $\color{red!50} x^ky^k\cdot y^{j_k}$ & $\color{cyan!60!black} x^{k+1}y^{k+1}\cdot h_2(x,y)$ \\
 \hline 
\end{tabular}
\end{center}
   
\end{definition}

\resizebox{6cm}{6cm}{
\begin{tikzpicture}
%begining layers
\node[cell] at (-1,-1) {\scalebox{.5}{$...$}};
\node[cell, minimum height=11cm] at (-1,0) {\tiny $\vdots$};
\node[cell, minimum width=11cm] at (0,-1) {\tiny $\ldots$};
\node[cell] at (11,-1) {};
\node[cell] at (-1,11) {};

%first/ k-1 th layer
\node[cell] at (0,0) {\scalebox{.5}{$x^{k}y^{k}$}};
\node[cell, minimum height=8cm] at (0,1) {\tiny $\vdots$};
\node[cell, minimum width=8cm] at (1,0) {\tiny $\ldots$};
\node[cell, fill=red!50] at (9,0) {};
\node[cell, fill=red!50] at (0,9) {};

%second layer

\node[cell, fill=cyan!50] at (1,1) {\scalebox{.5}{$x^{k+1}y^{k+1}$}};
\node[cell, minimum height=6cm, fill=cyan!50] at (1,2) {\tiny $\vdots$};
\node[cell, minimum width=6cm, fill=cyan!50] at (2,1) {\tiny $\ldots$};
%\node[cell, fill=lime!50] at (7,1) {};
%\node[cell, fill=lime!50] at (1,7) {};

%third layer
\node[cell, fill=cyan!50] at (2,2) {};
\node[cell, minimum height=4cm, fill=cyan!50] at (2,3) {\tiny $\vdots$};
\node[cell, minimum width=3cm, fill=cyan!50] at (3,2) {\tiny $\ldots$};
%\node[cell, fill=lime!50] at (6,2) {};
%\node[cell, fill=lime!50] at (2,7) {};

%fourth layer
\node[cell, fill=cyan!50] at (3,3) {};
\node[cell, minimum height=1cm, fill=cyan!50] at (3,4) {\tiny $\vdots$};
\node[cell, minimum width=1cm, fill=cyan!50] at (4,3) {\tiny $\ldots$};
%\node[cell, fill=lime!50] at (5,3) {};
%\node[cell, fill=lime!50] at (3,5) {};
%last layer
\node[cell, draw=none] at (4,4) {\tiny $x^k y^k$};
\node at (7,7) {$E$, $E_1$, or $E_2$};

\end{tikzpicture}
}
\hspace{1cm}
\resizebox{5.8cm}{4cm}{
\begin{tikzpicture}
\node at (6.8,6.8) {$F$};
%begining layers
\node[cell] at (0,-1) {\scalebox{.5}{$...$}};
%\node[cell, minimum height=7cm] at (-1,0) {\tiny $\vdots$};
\node[cell, minimum width=11cm] at (0,-1) {\tiny $\ldots$};

%\node[cell] at (0,-2) {\scalebox{.5}{$1$}};
%\node[cell, minimum width=11cm] at (0,-2) {\tiny $\ldots$};

%first/ k-1 th layer
\node[cell] at (0,0) {\scalebox{.5}{$y^{k}$}};

\node[cell, minimum width=8cm] at (1,0) {\tiny $\ldots$};
\node[cell, fill=red!50] at (9,0) {};

%second layer

\node[cell, fill=cyan!50] at (0,1) {\scalebox{.5}{$y^{k+1}$}};
\node[cell, minimum width=7cm, fill=cyan!50] at (1,1) {\tiny $\ldots$};

%third layer
\node[cell, fill=cyan!50] at (0,2) {\scalebox{.8}{$...$}};
\node[cell, minimum width=3cm, fill=cyan!50] at (1,2) {\tiny $\ldots$};

%fourth layer
\node[cell, fill=cyan!50] at (0,3) {\scalebox{.8}{$...$}};
\node[cell, minimum width=2.5cm, fill=cyan!50] at (1,3) {\tiny $\ldots$};

%last layer
\node[cell, draw=none] at (0,4) {\tiny $ y^N$};

\end{tikzpicture}
}

\begin{remark}
\begin{enumerate}
\item 
The monomials in the blue Tail regions are exactly the monomials in the vector space $\bigoplus_{k=0}^{v-1} \Gamma^k$, or $\bigoplus_{k=0}^{N-1} \Theta^k$, correspondingly.

\item
The generators $E$, or $E_1,E_2$ of the innermost layer $J^{v-1}$ has no Tail part. The generator $F$ of the innermost layer $W^{N-1}$ has no Tail part.
\end{enumerate}
\end{remark}

\begin{definition}
For the Diagonal filtration, set ${\Delta_k} = \dim (\frac{\Gamma^{k}}{\Gamma^{k+1}})$. For the Vertical filtration, set ${\Delta_k} = \dim (\frac{\Theta^{k}}{\Theta^{k+1}})$. 

(For simplicity, we use the same notation ${\Delta_k}$ in both the Vertical and Diagonal filtration settings.)

\end{definition}

\begin{remark}
The number ${\Delta_k}$ is the number of boxes in the $k-1$th layer of the blue Tail part of each generator of the $k$th layer.
\end{remark}

\begin{proposition}

\label{familyofideals}

For the additional generator of the $k$-th layer of the filtration: 

\begin{enumerate}
\item Assume $0 \leq k \leq N-1$. The family of the generator $F$ of $W^k$ is isomorphic to $$\mathbb{A}^{{\Delta_k}}.$$

\item Assume $0 \leq k \leq v-1$, if $v_{k}(J^{k})=[\circled{1}, (i_k,j_k)]$, or equivalently, we have the generator $E$:

\begin{enumerate}
\item If $v_{k+1}(J^{k+1})=[\circled{1},(i_{k+1},j_{k+1})]$, and also $(i_k,j_k)= (i_{k+1},j_{k+1})$, the family of $E$ is isomorphic to  $$\mathbb{A}^{\Delta_k}.$$

\item Otherwise, when $v_{k+1}(J^{k+1})=[\circled{1},(i_{k+1},j_{k+1})]$ and $(i_k,j_k) \geq (i_{k+1}+1,j_{k+1}+1)$, or when $v_{k+1}(J^{k+1})=[\circled{2},(i_{k+1},j_{k+1})]$, the family of $E$ is isomorphic to  $$\mathbb{A}^{\Delta_k} \times (\mathbb{P}^1-2pts)$$ .

\end{enumerate}

\item Assume $0 \leq k \leq v-1$, if $v_{k}(J^{k})=[\circled{2},(i_k,j_k)]$, or equivalently, we have the generators $E_1, E_2$:

\begin{enumerate}

\item If $v_{k+1}(J^{k+1})=[\circled{1},(i_{k+1},j_{k+1})]$, 
the family of both $E_1$ and $E_2$ is isomorphic to  $$\mathbb{A}^{{\Delta_k}}.$$ 

\item If $v_{k+1}(J^{k+1})=[\circled{2},(i_{k+1},j_{k+1})]$, the family of $E_1$ and $E_2$ is isomorphic to  $$\mathbb{A}^{{\Delta_k}+1}.$$

\end{enumerate}

\end{enumerate}
\end{proposition}

\begin{proof}
Multiply the generators $F$ in $W^k$ by $y$, $E$ in $J^k$ by $xy$, and $E_1, E_2$ in $J^k$ by $y$ and $x$ , we obtain polynomials of $W^{k+1}$ and $J^{k+1}$ that can be compared to the generators $F'$, $E'$ and $E_1', E_2'$ of $W^{k+1}$ and $J^{k+1}$. This induces relations (equations) among the coefficients of the monomials in the Tails of the generators of two consecutive layers, $W^{k}$ and $W^{k+1}$, $J^{k}$ and $J^{k+1}$. We compute the number of free coefficients.

\textbf{For the generator $F$.}

Multiply the generator $F$ by $y$, which is shifting it up by the vector $(0,+1)$. The original $F$ is the left picture, and the shifted $F$ is the right picture. 

Define the monomials in the shifted Tail region that is not contained in the original Tail region as \textcolor{teal!90!black}{Diff}. Define the monomials in the shifted Tail region that is also contained in the original Tail region as \textcolor{orange!80!black}{Inside}. In other words, the shifted \textcolor{cyan!65!black}{Tail} is equal to the union of two regions, \textcolor{orange!80!black}{Inside} and \textcolor{teal!90!black}{Diff}.

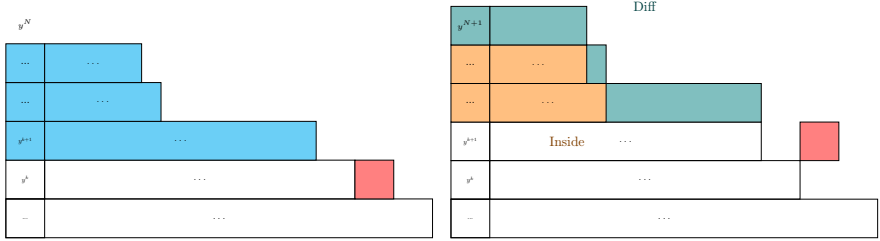
\resizebox{5.8cm}{3.1cm}{
\begin{tikzpicture}
%begining layers
\node[cell] at (0,-1) {\scalebox{.5}{$...$}};
%\node[cell, minimum height=7cm] at (-1,0) {\tiny $\vdots$};
\node[cell, minimum width=11cm] at (0,-1) {\tiny $\ldots$};

%\node[cell] at (0,-2) {\scalebox{.5}{$1$}};
%\node[cell, minimum width=11cm] at (0,-2) {\tiny $\ldots$};

%first/ k-1 th layer
\node[cell] at (0,0) {\scalebox{.5}{$y^{k}$}};

\node[cell, minimum width=8cm] at (1,0) {\tiny $\ldots$};
\node[cell, fill=red!50] at (9,0) {};

%second layer

\node[cell, fill=cyan!50] at (0,1) {\scalebox{.5}{$y^{k+1}$}};
\node[cell, minimum width=7cm, fill=cyan!50] at (1,1) {\tiny $\ldots$};

%third layer
\node[cell, fill=cyan!50] at (0,2) {\scalebox{.8}{$...$}};
\node[cell, minimum width=3cm, fill=cyan!50] at (1,2) {\tiny $\ldots$};

%fourth layer
\node[cell, fill=cyan!50] at (0,3) {\scalebox{.8}{$...$}};
\node[cell, minimum width=2.5cm, fill=cyan!50] at (1,3) {\tiny $\ldots$};

%last layer
\node[cell, draw=none] at (0,4) {\tiny $ y^N$};

\end{tikzpicture}
}
\resizebox{5.8cm}{3.2cm}{
\begin{tikzpicture}
\node at (5,5) {\textcolor{teal!50!black}{Diff}};
\node at (3,1.5) {\textcolor{orange!50!black}{Inside}};
%begining layers
\node[cell] at (0,-1) {\scalebox{.5}{$...$}};
%\node[cell, minimum height=7cm] at (-1,0) {\tiny $\vdots$};
\node[cell, minimum width=11cm] at (0,-1) {\tiny $\ldots$};

%first/ k-1 th layer
\node[cell] at (0,0) {\scalebox{.5}{$y^{k}$}};

\node[cell, minimum width=8cm] at (1,0) {\tiny $\ldots$};
\node[cell, fill=red!50] at (9,1) {};

%second layer

\node[cell] at (0,1) {\scalebox{.5}{$y^{k+1}$}};
\node[cell, minimum width=7cm] at (1,1) {\tiny $\ldots$};
\node[cell, minimum width=7cm, fill=teal!50] at (1,2) {};

%third layer
\node[cell, fill=orange!50] at (0,2) {\scalebox{.8}{$...$}};
\node[cell, minimum width=3cm, fill=orange!50] at (1,2) {\tiny $\ldots$};
\node[cell, minimum width=3cm, fill=teal!50] at (1,3) {};

%fourth layer
\node[cell, fill=orange!50] at (0,3) {\scalebox{.8}{$...$}};
\node[cell, minimum width=2.5cm, fill=orange!50] at (1,3) {\tiny $\ldots$};
\node[cell, minimum width=2.5cm, fill=teal!50] at (1,4) {};
\node[cell, fill=teal!50] at (0,4) {\scalebox{.8}{}};
%last layer
\node[cell, draw=none] at (0,4) {\tiny $y^{N+1}$};

\end{tikzpicture}
}

Because of Lemma \ref{condigenbig}, the length of the boxes at each layer of the blue Tail are weakly decreasing, so we are always able to use the generators of $W^\tau$, inductively for $\tau=k+1, ..., N-1$, the generators of from the lower layers to the higher layers to eliminate the monomials in the shifted \textcolor{red!80!black}{Head} part (in the right picture) and \textcolor{teal!90!black}{Diff} part, by rewriting these monomials in \textcolor{red!80!black}{Head} and \textcolor{teal!90!black}{Diff} as the Tail of generators of the higher layers. After rewriting, the monomials in \textcolor{red!80!black}{Head} and \textcolor{teal!90!black}{Diff} region are rewritten as a linear combination of the monomials in the \textcolor{orange!80!black}{Inside} region. But we know from, Lemma \ref{lemik}, the valuation $v_{\tau}(W^{\tau})$ for all $\tau = k+1,...,N-1$ of higher layers that there should never be a sum of monomials only from the Tail region  (without adding monomials from the corresponding \textcolor{red!80!black}{Head} part), which valuates to points with $x$-coordinates smaller than $i_k$ in the valuation of type $[i_k]$, contradiction. So the linear combination of monomials in the Inside region must be set to be $0$, which give us equations that express the coefficients of the \textcolor{orange!80!black}{Inside} part as polynomials of the coefficients of monomials of the \textcolor{red!80!black}{Head} and \textcolor{teal!90!black}{Diff}. 

The number of free variables is the number of boxes in the \textcolor{teal!90!black}{Diff} part, which is equal to the number of boxes in the lowest layer of the blue Tail region of $F$.

So the family of $F$ of $J^k$ is isomorphic to $\mathbb{A}^{{\Delta_k}}$. This proves (1).

\textbf{For the generator $E$.}

As we have showed in Proposition \ref{prop=}, for the degrees of freedom in the \textcolor{red!80!black}{Head} part, the coefficients $a, b$ where $a,b \neq 0$ of $E$ are free if and only if $i_{k} \geq i_{k+1}+1, j_{k} \geq j_{k+1}+1$. Together, they form a family $\mathbb{P}^1-2pts$ where the two points deleted are $a=0$ and $b=0$. The coefficients $a,b$ are completely determined when $i_{k} = i_{k+1}, j_{k} = j_{k+1}$, so adding no degree of freedom. 

Now we show that the number of free coefficients in the \textcolor{cyan!65!black}{Tail} part are exactly the number of boxes in the \textcolor{teal!90!black}{Diff} region, obtained from shifting the \textcolor{cyan!65!black}{Tail} part by the vector $(+1, +1)$ to the north-east direction. 

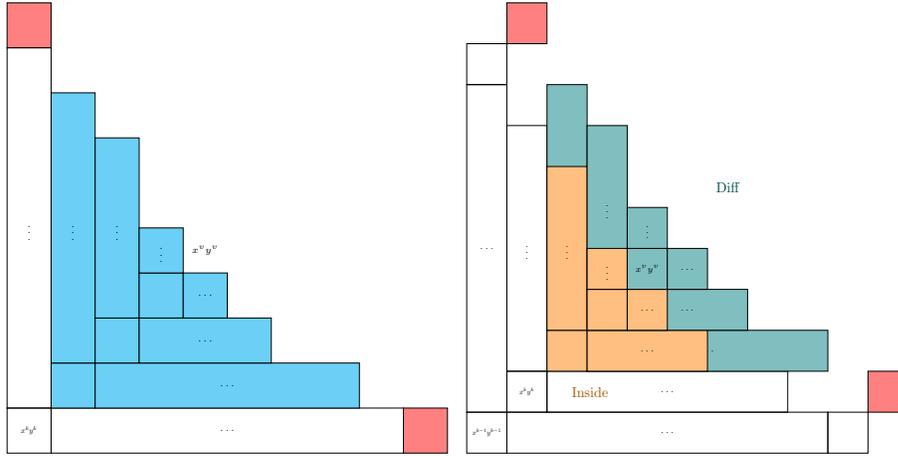
\resizebox{6cm}{6cm}{
\begin{tikzpicture}

%first/ k-1 th layer
\node[cell] at (0,0) {\scalebox{.5}{$x^{k}y^{k}$}};
\node[cell, minimum height=8cm] at (0,1) {\tiny $\vdots$};
\node[cell, minimum width=8cm] at (1,0) {\tiny $\ldots$};
\node[cell, fill=red!50] at (9,0) {};
\node[cell, fill=red!50] at (0,9) {};

%second layer

\node[cell, fill=cyan!50] at (1,1) {\scalebox{.5}{}};
\node[cell, minimum height=6cm, fill=cyan!50] at (1,2) {\tiny $\vdots$};
\node[cell, minimum width=6cm, fill=cyan!50] at (2,1) {\tiny $\ldots$};
%\node[cell, fill=lime!50] at (7,1) {};
%\node[cell, fill=lime!50] at (1,7) {};

%third layer
\node[cell, fill=cyan!50] at (2,2) {};
\node[cell, minimum height=4cm, fill=cyan!50] at (2,3) {\tiny $\vdots$};
\node[cell, minimum width=3cm, fill=cyan!50] at (3,2) {\tiny $\ldots$};
%\node[cell, fill=lime!50] at (6,2) {};
%\node[cell, fill=lime!50] at (2,7) {};

%fourth layer
\node[cell, fill=cyan!50] at (3,3) {};
\node[cell, minimum height=1cm, fill=cyan!50] at (3,4) {\tiny $\vdots$};
\node[cell, minimum width=1cm, fill=cyan!50] at (4,3) {\tiny $\ldots$};
%\node[cell, fill=lime!50] at (5,3) {};
%\node[cell, fill=lime!50] at (3,5) {};
%last layer
\node[cell, draw=none] at (4,4) {\tiny $x^v y^v$};

\end{tikzpicture}
}
\resizebox{6cm}{6cm}{
\begin{tikzpicture}
%first/ k-1 th layer
\node[cell] at (0,0) {\scalebox{.5}{$x^{k-1}y^{k-1}$}};
\node[cell, minimum height=8cm] at (0,1) {\tiny $\ldots$};
\node[cell, minimum width=8cm] at (1,0) {\tiny $\ldots$};
\node[cell] at (9,0) {};
\node[cell] at (0,9) {};

%shift

\node[cell, fill=red!50] at (10,1) {};
\node[cell, fill=red!50] at (1,10) {};

%second layer

\node[cell] at (1,1) {\scalebox{.5}{$x^{k}y^{k}$}};
\node[cell, minimum height=6cm] at (1,2) {\tiny $\vdots$};
\node[cell, minimum width=6cm] at (2,1) {\tiny $\ldots$};
%\node[cell, fill=lime!50] at (7,1) {};
%\node[cell, fill=lime!50] at (1,7) {};

%shift

\node[cell, minimum height=6cm, fill=teal!50] at (2,3) {\tiny $\vdots$};
\node[cell, minimum width=6cm, fill=teal!50] at (3,2) {\tiny $\ldots$};

%\node[cell, fill=lime!50] at (7,1) {};
%\node[cell, fill=lime!50] at (1,7) {};

%third layer
\node[cell, fill=orange!50] at (2,2) {};
\node[cell, minimum height=4cm, fill=orange!50] at (2,3) {\tiny $\vdots$};
\node[cell, minimum width=3cm, , fill=orange!50] at (3,2) {\tiny $\ldots$};
%\node[cell, fill=lime!50] at (6,2) {};
%\node[cell, fill=lime!50] at (2,7) {};

%shift
\node[cell, fill=orange!50] at (3,3) {};
\node[cell, minimum height=4cm, fill=teal!50] at (3,4) {\tiny $\vdots$};
\node[cell, minimum width=3cm, , fill=teal!50] at (4,3) {\tiny $\ldots$};

%fourth layer
\node[cell, fill=orange!50] at (3,3) {};
\node[cell, minimum height=1cm, fill=orange!50] at (3,4) {\tiny $\vdots$};
\node[cell, minimum width=1cm, fill=orange!50] at (4,3) {\tiny $\ldots$};
\node[cell, draw=none] at (2.5,1) {\textcolor{orange!70!black}{Inside}};
%\node[cell, fill=lime!50] at (3,5) {};

%last layer

%shift
\node[cell, fill=teal!50] at (4,4){};
\node[cell, minimum height=1cm, fill=teal!50] at (4,5) {\tiny $\vdots$};
\node[cell, minimum width=1cm, , fill=teal!50] at (5,4) {\tiny $\ldots$};

\node[cell, draw=none] at (4,4) {\tiny $x^vy^v$};

\node[cell, draw=none] at (6,6) {\textcolor{teal!60!black}{Diff}};
\end{tikzpicture}
}

After shifting the polynomial $E$ by $(+1,+1)$, which corresponds to multiplying $E$ by $xy$, we obtain $xy \cdot E$ as an element of $J^{k+1}$. As before, define monomials in the difference between the shifted \textcolor{cyan!65!black}{Tail} from the original Tail by \textcolor{teal!90!black}{Diff}. Define the monomials in the shifted Tail that stay inside the original Tail region after shifting as \textcolor{orange!80!black}{Inside}. In other words, the shifted \textcolor{cyan!65!black}{Tail} is equal to the union of two regions, \textcolor{orange!80!black}{Inside} and \textcolor{teal!90!black}{Diff}. If a row in the Tail region is of type $\circled{1}$, then naturally, the Inside region has two shaded cells at the two ends, each counting as half of a dimension. In this case, each shaded cell also counts as half of a dimension in the \textcolor{teal!90!black}{Diff} region. 

Note that because of Lemma \ref{propij}, the length and height of the boxes at each layer of the Tail are weakly decreasing, so we are always able to use the generators of $J^{\tau}$, $\tau \geq k+1$, to eliminate the monomials in shifted \textcolor{red!80!black}{Head} part (in the right picture) and \textcolor{teal!90!black}{Diff} part, by rewriting these monomials as the Tail of the generators of the previous layers. So these monomials are rewritten as a linear combination of the monomials in the Inside region. But we know from the valuation $v_{\tau}(J^{\tau})$ for all $\tau \geq k+1$ that there should be not be a sum of monomials only from the Inside region  (without adding monomials the corresponding \textcolor{red!80!black}{Head} part), which will valuates to points with coordinates smaller than the valuation of the \textcolor{red!80!black}{Head} part of the generator $(i_{\tau},j_{\tau})$, contradiction. So the coefficients of these monomials in the Tail region must be set to be $0$. 

By setting the coefficients of these monomials to $0$, coefficients of the \textcolor{orange!80!black}{Inside} parts are written as polynomials of the coefficients of the \textcolor{teal!90!black}{Diff} and \textcolor{red!80!black}{Head} part. So the coefficients in the \textcolor{teal!90!black}{Diff} part is free. The total number of boxes in \textcolor{teal!90!black}{Diff} is equal to the number of boxes in the outmost layer of the blue Tail part, $\sum_{\tau=k}^{v-1}(\dim (\frac{\Gamma^\tau}{\Gamma^{\tau+1}})-\dim (\frac{\Gamma^{\tau+1}}{\Gamma^{\tau+2}}))= \dim (\frac{\Gamma^k}{\Gamma^{k+1}}) = {\Delta_k}$

%Inductively, we consider the other valuations $v_{k+2}(J^{k+2}), v_{k+3}(J^{k+3}), ..., v_{{\Delta_k} -1}(J^{{\Delta_k}-1})$ and then obtain that the coefficients for the monomials in the Tail part should all be set to $0$. These gives us the dependence relation of the coefficients of the monomials in the Inside part on the Diff part. 

We specially note that, the elimination/rewriting process of the \textcolor{teal!90!black}{Diff} part by generators of the previous layers is clear when the previous layers have generators of type $E_1, E_2$, because only solid boxes are involved. If some of the previous layers has generator $E$, then we can eliminate one of the two monomials from the \textcolor{red!80!black}{Head} part of $E$, and the other coefficient of the other monomial has to be set to $0$ after the elimination. In other words, out of $2$ coefficients, one of them is free, and the other one is determined; out of $2$ degrees of freedom, whichever one is free or determined is completely by our choice. But still, the number of free coefficients is exactly the number of boxes of each layer ${\Delta_k}$ because we took care of such situation by the construction of the shaded boxes in the type $\circled{1}$ rows, for the reason as we shall explain now. If the $k+1$th layer has generator $E$, then by construction, the two shaded boxes together counts as one degree of freedom in a weak diagonal partition. If the more inner $\tau$th layers, $\tau = k+2, ..., v-1$, has generator $E$, then the two beginnings (on both the horizontal and vertical arms) of the \textcolor{teal!90!black}{Diff} region of $\tau$th layer, and the endings of the \textcolor{teal!90!black}{Diff} region of the $\tau+1$th layer both have shaded boxes, which exactly sums up to two solid boxes. In conclusion, the \textcolor{teal!90!black}{Diff} region has exactly the number of boxes it should have, being the number of free variables.

Therefore, if $i_{k} \geq i_{k+1}+1$, and $j_{k} \geq j_{k+1}+1$, the family of choices of generator $E$ we add at the $J^{k}$ layer is isomorphic to $(\mathbb{P}^1-2pts) \times \mathbb{A}^{{\Delta_k}}$, which proves (2)(b). And if $i_{k}= i_{k+1}$, and $j_{k}= j_{k+1}$, the family of choices of generator $E$ we add at the $J^{k}$ layer is isomorphic to  $\mathbb{A}^{{\Delta_k}},$ which proves (2)(a).
%Note that $\dim (\frac{\Gamma^k}{\Gamma^{k+1}})$ is the number of boxes at the $k$-th layer (outmost layer) of the Tail part.

\textbf{For the generators $E_1, E_2$.}

Multiply $E_1$ by $y$, which correspond to shifting $E_1$ up by the vector $(0,+1)$. Multiply $E_2$ by $x$, which is shifting $E_2$ to the right by $(+1,0)$.

\begin{enumerate}
\item Similar to $E$, we define the monomials in the region of the shifted Tail of $E_1$ that are not contained in the original Tail of $E_1$ to be "\textcolor{teal!80!black}{Diff1}". Define the monomials of the shifted Tail of $E_1$ that are contained in the original Tail of $E_1$ to be "\textcolor{orange!80!black}{Inside1}". The shaded boxes follow the exact same rules as before, and is each counted as half of a box and contained in both the \textcolor{orange!80!black}{Inside1} and \textcolor{teal!90!black}{Diff} regions.
\end{enumerate}

%Now $yE1 \in J^k$ and $xE2 \in J^k$, and we can use generators from the previous layers to eliminate the monomials in Diff1 and Diff2.

We focus on the case of $E_1$ shifting to $y\cdot E_1$, and the case of $E_2$ is exactly the same.

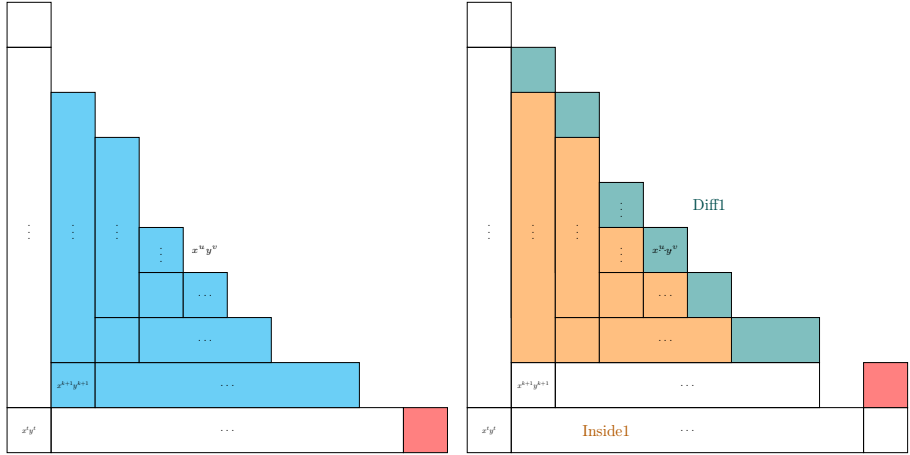
\resizebox{6cm}{6cm}{
\begin{tikzpicture}

%first/ k-1 th layer
\node[cell] at (0,0) {\scalebox{.5}{$x^{t}y^{t}$}};
\node[cell, minimum height=8cm] at (0,1) {\tiny $\vdots$};
\node[cell, minimum width=8cm] at (1,0) {\tiny $\ldots$};
\node[cell] at (0,9)  {};
\node[cell, fill=red!50] at (9,0) {};

%second layer

\node[cell, fill=cyan!50] at (1,1) {\scalebox{.5}{$x^{k+1}y^{k+1}$}};
\node[cell, minimum height=6cm, fill=cyan!50] at (1,2) {\tiny $\vdots$};
\node[cell, minimum width=6cm, fill=cyan!50] at (2,1) {\tiny $\ldots$};
%\node[cell, fill=lime!50] at (7,1) {};
%\node[cell, fill=lime!50] at (1,7) {};

%third layer
\node[cell, fill=cyan!50] at (2,2) {};
\node[cell, minimum height=4cm, fill=cyan!50] at (2,3) {\tiny $\vdots$};
\node[cell, minimum width=3cm, fill=cyan!50] at (3,2) {\tiny $\ldots$};
%\node[cell, fill=lime!50] at (6,2) {};
%\node[cell, fill=lime!50] at (2,7) {};

%fourth layer
\node[cell, fill=cyan!50] at (3,3) {};
\node[cell, minimum height=1cm, fill=cyan!50] at (3,4) {\tiny $\vdots$};
\node[cell, minimum width=1cm, fill=cyan!50] at (4,3) {\tiny $\ldots$};
%\node[cell, fill=lime!50] at (5,3) {};
%\node[cell, fill=lime!50] at (3,5) {};
%last layer
\node[cell, draw=none] at (4,4) {\tiny $x^u y^v$};

\end{tikzpicture}
}
\resizebox{6cm}{6cm}{
\begin{tikzpicture}
%first/ k-1 th layer
\node[cell] at (0,0) {\scalebox{.5}{$x^{t}y^{t}$}};
\node[cell, minimum height=8cm] at (0,1) {\tiny $\vdots$};
\node[cell, minimum width=8cm] at (1,0) {\tiny $\ldots$};
\node[cell] at (9,0) {};
\node[cell] at (0,9) {};

%shift

\node[cell, fill=red!50] at (9,1) {};

%shift

\node[cell, minimum height=6cm, fill=teal!50] at (1,3) {\tiny $\vdots$};
\node[cell, minimum width=6cm, fill=teal!50] at (2,2) {\tiny $\ldots$};
%second layer

\node[cell] at (1,1) {\scalebox{.5}{$x^{k+1}y^{k+1}$}};
\node[cell, minimum height=6cm, fill=orange!50] at (1,2) {\tiny $\vdots$};
\node[cell, minimum width=6cm] at (2,1) {\tiny $\ldots$};

%shift
\node[cell, minimum height=4cm, fill=teal!50] at (2,4) {\tiny $\vdots$};
\node[cell, minimum width=3cm, , fill=teal!50] at (3,3) {\tiny $\ldots$};

%third layer
\node[cell, fill=orange!50] at (2,2) {};
\node[cell, minimum height=4cm, fill=orange!50] at (2,3) {\tiny $\vdots$};
\node[cell, minimum width=3cm, fill=orange!50] at (3,2) {\tiny $\ldots$};
%\node[cell, fill=lime!50] at (6,2) {};
%\node[cell, fill=lime!50] at (2,7) {};

%shift
\node[cell, fill=teal!50] at (4,4){};
\node[cell, minimum height=1cm, fill=teal!50] at (3,5) {\tiny $\vdots$};
\node[cell, minimum width=1cm, , fill=teal!50] at (4,4) {\tiny $\ldots$};

%fourth layer
\node[cell, fill=orange!50] at (3,3) {};
\node[cell, minimum height=1cm, fill=orange!50] at (3,4) {\tiny $\vdots$};
\node[cell, minimum width=1cm, fill=orange!50] at (4,3) {\tiny $\ldots$};
%\node[cell, fill=lime!50] at (5,3) {};
%\node[cell, fill=lime!50] at (3,5) {};

%last layer

\node[cell, draw=none] at (4,4) {\tiny $x^u y^v$};

\node[cell, draw=none] at (5,5) {\textcolor{teal!60!black}{Diff1}};
\node[cell, draw=none] at (2.5,0) {\textcolor{orange!70!black}{Inside1}};
\end{tikzpicture}
} 

For all monomials in the region \textcolor{teal!80!black}{Diff1}, we can use the generators from $J^{\tau}$ for all $\tau \geq k+1$, the generators from previous layers, to eliminate them and rewrite them as the monomials of the Tail of the generators of previous layers. 

Because of the valuations $v_{\tau}(J^{\tau})$ for all $\tau \geq k+1$ of the previous layers, a sum of monomials in the \textcolor{orange!80!black}{Inside1} region of $E_1$ should be set to $0$, so we have equations that write the coefficients of the monomials of the \textcolor{orange!80!black}{Inside1} region as polynomials of coefficients of monomials of the \textcolor{red!80!black}{Head} and \textcolor{teal!80!black}{Diff1} region of $E_1$.

%So the coefficients of the monomials of the Inside region is dependent on the coefficients of the Diff1 and Head region. 

If the $k+1$th layer has generators of type $\circled{2}$, the number of free variables of $E_1$ is exactly the number of boxes in \textcolor{teal!80!black}{Diff1}, which is equal to the length of the horizontal arm of the $k+1$th layer in the Tail region. By the same argument as $E_1$, the number of free variables of $E_2$ is exactly the number of boxes in the vertical arm of the $k+1$th layer in the Tail region. So if the $k$th and the $k+1$th layer both have type $\circled{2}$, then the number of free variables is equal to $\delta+1$, because the central box in the $k+1$th layer is being over-counted once. This proves (3)(b).

When the $k+1$th layer is type $\circled{1}$, we argue carefully the number of free variables as follows. 

When $J^{k+1}=[\circled{1},(i_{k+1},j_{k+1})]$, and either $i_{k} = i_{k+1}$, or $j_{k} = j_{k+1}$ holds, then there are conditions on the coefficients of $E_1,E_2$, as discussed in Lemma \ref{propij}. For example, we restate the condition when $i_{k} = i_{k+1}$, we have that $$b'y^{j_{k+1}}-a'yh_1(x,y)=0 \text{ modulo } xy,$$ which is a condition on the boxes of the Tail part $h_1(x,y)$ that is not divisible by $x$, which is the left-most column of the Tail part. All the coefficients in this column, except for the coefficient of the box $x^{k+1}y^{k+1}y^{j_{k+1}-1}$ being set to $\frac{b'}{a'}$, should be set to $0$. Because the box $x^{k+1}y^{k+1}y^{j_{k+1}}$ is the \textcolor{red!80!black}{Head} part of the generator $E'$ of $J^{k+1}$, the box right below it, $x^{k+1}y^{k+1}y^{j_{k+1}-1}$, should be the second box in the region \textcolor{teal!80!black}{Diff1} (note here that \textcolor{teal!90!black}{Diff1} refers to the region as a subset of the Tail, meaning we are shifting Diff1 in the picture on the right back by $(0,-1)$ as if it came from a box in the Tail), exactly one of the two shaded boxes in the vertical arm of the $k+1$th layer. 

Because from rewriting the monomials in the \textcolor{teal!80!black}{Diff1} part using generator of previous layer, we can only obtain monomials in the more inner layers, we cannot obtain the monomials in boxes in the left-most column of the \textcolor{orange!80!black}{Inside1} part of $E_1$, so contribution from the coefficients of the left-most column of \textcolor{orange!80!black}{Inside1} are already $0$ except for the shaded boxes in the $k+1$th layer, and setting these coefficients to $0$ exactly only puts conditions on the solid boxes in the \textcolor{orange!80!black}{Inside1} region of the $k+1$th layer. 

Together, the second one of the two shaded boxes in the vertical arm of the $k+1$th layer is removed. The total degree of freedom of $E_1$ is equal to the number of boxes in the \textcolor{teal!90!black}{Diff} region minus one shaded box, which is equal to the number of boxes in the horizontal arm of the $k+1$th layer, minus one half, coming from one shaded box on the top of the \textcolor{teal!90!black}{Diff} region on the vertical arm of the $k+1$th layer. 

The arguments are exactly the same for $E_2$ when $J^{k+1}=[\circled{1},(i_{k+1},j_{k+1})]$, and $j_k=j_{k+1}$, and the degree of freedom of $E_2$ is equal to the number of boxes in the vertical arm minus one half.

If $i_{k} \geq i_{k+1}+1$ or $j_{k} \geq j_{k+1}+1$, we have a similar situation. The second shaded box in the vertical arm of the $k+1$th layer must be set to $0$, because it correspond to the \textcolor{red!80!black}{Head} part of the $k+1$th layer, and the other box in the \textcolor{red!80!black}{Head} part on the horizontal arm automatically have coefficient $0$, because the coefficient of it comes from shifting the horizontal arm of the $k$th layer up, which because of $i_k \geq i_{k+1}+1$, is $0$. The degree of freedoms of the generator $E_1$ in this case is the number of boxes in the \textcolor{teal!90!black}{Diff} region minus a shaded box. Same for $E_2$, which has degrees of freedom equal to the number of boxes in the vertical arm of the $k+1$th layer minus half of a box.

So in both cases, going from an inner type $\circled{1}$ layer to an outer $\circled{2}$ layer, the two half boxes in the central column of the $k+1$th layer adds to one solid box. Together with $E_1$ and $E_2$, we add the degree of freedom the same as the number of boxes in the $k+1$th layer, which is equal to $\Delta$. This proves (3)(a).

Finally, for all generators $F,E,E_1,E_2$, we claim that it is enough to compare the generators by shifting one layer up or right only, because the free variables obtained from shifting $1$ layer is a proper subset of the free variables obtained from shifting $2$ or more layers. Shifting $2$ or more layers does not impose more conditions.
\end{proof}

\section{Poincar\'e polynomial of $\Hilb^n(\{xy^v=0\},0)$ and $\Hilb^n(\{x^2y^v=0\},0)$}

\subsection{Poincar\'e polynomial of $\Hilb^n(\{xy^v=0\},0)$}

\begin{theorem}
The Hilbert scheme $\Hilb^{n}(\{xy^v=0\},0)$ has an affine stratification, and each stratum is indexed by a partition $(i_0, ..., i_{N})$ such that $i_k \leq 1$ for all $k \geq v$. The dimension of the stratum is the number of boxes from the $1$st row to the $v$th row, i.e. $\sum_{k=1}^{v}i_k$.

The Poincar\'e polynomial of the Hilbert scheme $\Hilb^{n}(\{xy^v=0\},0)$ is $$\prod_{i=1}^{v}\frac{1}{1-Q^{i}T^{2(i-1)}}(1+\frac{Q^{v+1}T^{2v}}{1-Q}).$$
\end{theorem}

\begin{proof}
We draw the shape following the Vertical filtration and valuation, corresponding to an arbitrary stratum of ideals $I$ in $\Hilb^{n}(\{xy^v=0\},0)$.

\begin{enumerate}
\item If $N \geq v$, then we have a collection of generators, each added at the layer $W^k$ for $v \leq k \leq N-1$, denoted by $F^k$, where each $F^k$ has the form 
$$F^k= xy^k+ \sum_{i=k+1}^{N-1} \limits \sigma^k_{i} y^{i},$$
for some $\sigma^k_{i} \in \mathbb{C}$.

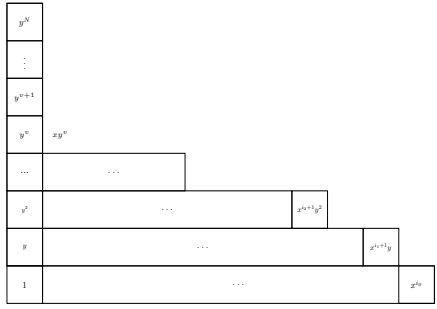
\resizebox{5.8cm}{4cm}{
\begin{tikzpicture}

%begining layers
\node[cell] at (-1,-1) {\scalebox{.7}{$1$}};
\node[cell, minimum height=7cm] at (-1,0) {};
\node[cell, minimum width=11cm] at (0,-1) {\tiny $\ldots$};
\node[cell] at (10,-1) {\scalebox{.7}{$x^{i_0}$}};
\node[cell] at (-1,6) {\scalebox{.7}{$y^{N}$}};

%first/ k-1 th layer
\node[cell] at (-1,0) {\scalebox{.6}{$y$}};

\node[cell, minimum width=9cm] at (0,0) {\tiny $\ldots$};
\node[cell] at (9,0) {\scalebox{.6}{$x^{i_1+1}y$}};

%second layer

\node[cell] at (-1,1) {\scalebox{.5}{$y^{2}$}};
\node[cell, minimum width=7cm] at (0,1) {\tiny $\ldots$};
\node[cell] at (7,1) {\scalebox{.6}{$x^{i_2+1}y^2$}};

%third layer
\node[cell] at (-1,2) {\scalebox{.8}{$...$}};
\node[cell, minimum width=4cm] at (0,2) {\tiny $\ldots$};

%fourth layer
%\node[cell] at (0,3) {\scalebox{.8}{$...$}};
%\node[cell, minimum width=2.5cm] at (1,3) {\tiny $\ldots$};

%last layer
\node[cell, draw=none] at (0,3) {\tiny $x y^v$};

%vertical strip
\node[cell] at (-1,3) {\tiny $y^v$};

\node[cell] at (-1,4) {\tiny $y^{v+1}$};
\node[cell] at (-1,5) {\tiny \vdots};
\end{tikzpicture}
}

The condition that $xy^v \in I$ implies that all the coefficients $\sigma^v_{i}, i=v+1, ...,N-1$ of the generator $F^v$ are equal to $0$. 

By Proposition \ref{familyofideals} and the dependence relationships among all the coefficients $\sigma^k_{i}$ of all $W^k$ for $v \leq k \leq N-1$, $\sigma^v_{i}=0, i=v+1, ...,N-1$ implies that $\sigma^k_{i}=0$ for all $k=v,..., N-1$ and all $i$. Therefore, all the generators $F^k$ of $W^k$ for $v \leq k \leq N-1$ has no degree of freedom.

Once the generator $F^v$ of $W^v$ is fixed, there are no other restrictions on the rest of the generators $F^k$ where $0 \leq k \leq v-1$, so the degree of the freedom of them in total equals to $\sum_{i=0}^{v-1}{\Delta_k}^i$ by Proposition \ref{familyofideals}. We can further conclude that these strata are all affine, and the dimension of each stratum equals to the total number of boxes between the $1$st and $v$th rows (excluding the $0$th row).

\item When $N \leq v-1$, then $y^v \in I$, which implies that $xy^v \in I$. There are no extra conditions on the generators. The strata are all affine, and the dimension of the stratum is exactly the number of all boxes between the $1$st and $v$th rows (excluding the $0$th row).

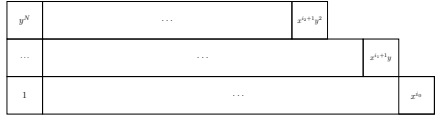
\resizebox{5.8cm}{2.5cm}{
\begin{tikzpicture}

%begining layers
\node[cell] at (-1,-1) {\scalebox{.7}{$1$}};
%\node[cell, minimum height=3cm] at (-1,0) {};
\node[cell, minimum width=11cm] at (0,-1) {\tiny $\ldots$};
\node[cell] at (10,-1) {\scalebox{.7}{$x^{i_0}$}};
\node[cell] at (-1,1) {\scalebox{.7}{$y^{N}$}};

%first/ k-1 th layer
\node[cell] at (-1,0) {\scalebox{.6}{$\ldots$}};

\node[cell, minimum width=9cm] at (0,0) {\tiny $\ldots$};
\node[cell] at (9,0) {\scalebox{.6}{$x^{i_1+1}y$}};

%second layer

%\node[cell] at (-1,1) {\scalebox{.5}{$y^{2}$}};
\node[cell, minimum width=7cm] at (0,1) {\tiny $\ldots$};
\node[cell] at (7,1) {\scalebox{.6}{$x^{i_2+1}y^2$}};

%third layer
%\node[cell] at (-1,2) {\scalebox{.8}{$...$}};
%\node[cell, minimum width=4cm] at (0,2) {\tiny $\ldots$};

%fourth layer
%\node[cell] at (0,3) {\scalebox{.8}{$...$}};
%\node[cell, minimum width=2.5cm] at (1,3) {\tiny $\ldots$};

%last layer
\node[cell, draw=none] at (0,3) {$x y^v$};

\end{tikzpicture}
}

\end{enumerate}

Now we calculate the Poincar\'e polynomial of $\Hilb^n(\{xy^v=0\},0)$ based on combinatorial rules computed above. Use the variable $Q$ to record the number of points $n$ of the Hilbert scheme, and $T$ to record the homological degree. 
Recall that the dimension we computed for each stratum is the complex dimension and should be multiplied by $2$ in the computation of the Poincar\'e polynomial.

\begin{enumerate}
\item We first look only at boxes of the left-most column. The box $1$ at the $0$st layer does not contribute any dimension. When $N \geq v$, we must have $v$ boxes at the left-most column between the $1$st and $v$st layer which contributes to dimension $v$, and then as many boxes on top which does not contribute to any dimension. This gives the factor $$Q^{v+1}T^{2v} \frac{1}{1-Q}.$$

Starting from the second column to the right, each column, if it has $i$ boxes, then this column contributes to dimension $i-1$. The number of boxes $i$ can be any number between $1$ and $v$. This adds the factor $$\prod_{i=1}^{v}\frac{1}{1-Q^{i}T^{2(i-1)}}.$$

In total, the generating function is

$$\frac{Q^{v+1}T^{2v}}{1-Q}\prod_{i=1}^{v}\frac{1}{1-Q^{i}T^{2(i-1)}}.$$

\item When $N \leq v-1$, we have a total of at most $N$ rows. For each column in the partition, if it has $i$ boxes, then it contributes to dimension $i-1$. The number of boxes $i$ can be any number between $1$ and $v$. The generating function is therefore $$\prod_{i=1}^{v}\frac{1}{1-Q^{i}T^{2(i-1)}}.$$
\end{enumerate}

In total, the generating function is $$\prod_{i=1}^{v}\frac{1}{1-Q^{i}T^{2(i-1)}}(1+\frac{Q^{v+1}T^{2v}}{1-Q}).$$
\end{proof}

\begin{corollary}
[Colored ORS conjecture for the Hopf link]
Let $\chi$ denote the Poincar\'e polynomial of the Hilbert scheme of $n$ points on $\{xy^v=0\}$, and $P$ the Poincar\'e polynomial of the $(\Sym^{v}, \Sym^{1})$-colored homology of the Hopf link. Up to the change of variables $T$ to $(QT^2)^{-1}$ of the Poincar\'e polynomial $P$ of the Hopf link, we have

\begin{align*}
\sum_{n=0}^{\infty}\chi(\Hilb^n(\{xy^v=0\},0))Q^n=P(T(2,2),(\Sym^{u}, \Sym^{1})).
\end{align*}

\end{corollary}

\begin{proof}
This immediately follows from the equality of the Poincar\'e polynomials, after the change of variables $T \rightarrow (QT^2)^{-1}$ of the Poincar\'e polynomial

$P(T(2,2),(\Sym^{u}, \Sym^{1}))$, as computed in Corollary \ref{Poincar\'ehofp}.
\end{proof}

\subsection{Poincar\'e polynomial of $\Hilb^n(\{x^2y^v=0\},0)$}

\begin{theorem}
The Hilbert scheme $\Hilb^{n}(\{x^2y^v=0\},0)$ has an affine stratification, and each stratum is indexed by a partition $(i_0, ..., i_{N})$ such that $i_k \leq 2$ for all $k \geq v$. Denote the index of $i_k$ as $M$ where $i_k=1$ for $k=M,...,N$ and $i_k > 1$ otherwise. The dimension of the stratum is the number of boxes from the $1$st row to the $v$th row, plus the number of boxes between the $\lfloor \frac{M+N}{2}\rfloor+1$th row to the $N$th row, $$\sum_{k=1}^{v}i_k + \sum_{k=\lfloor \frac{M+N}{2}\rfloor+1}^{N}i_k.$$ If $M$ does not exist, then the dimension is only the first summand $\sum_{k=1}^{v}i_k$.

The Poincar\'e polynomial of the Hilbert scheme $\Hilb^{n}(\{x^2y^v=0\},0)$ is 
\begin{align*}
\frac{(1-Q)(1-Q^2T^2)+(1-Q)(1+QT^2)Q^{v+1}T^{2v} + Q^{2(v+1)}T^{4v}}{(1-Q)(1-Q^2T^2)\prod_{i=1}^{v}(1-Q^{i}T^{2(i-1)})}.
\end{align*}
\end{theorem}

\begin{proof}
Following Notation \ref{notik}, for an ideal $I \in \Hilb^{n}(\{x^2y^v=0\},0)$, because $x^2y^v \in I$, the Vertical valuation $v_{(0,v)}(W^v)$ can only be of type $[0],[1],[2]$. 
\begin{enumerate}
\item If the valuation $v_{(0,v)}(W^v)$ has type $[0]$, then $y^v \in I$, and the generating function is $$ \frac{1}{\prod_{i=1}^v (1-Q^{i}T^{2(i-1)})}.$$ 

\item If the valuation $v_{(0,v)}(W^v)$ has type $[1]$, then the valuation $v_{(0,k)}(W^k)$ for all $k=v+1,..., N$ also has type $[1]$ for some $N \geq v$, until at some point we have $y^{N+1} \in I$, (if $N=v$, then we don't have any other $v_{(0,k)}(W^k) =[1]$, and $y^{v+1} \in I$). The generator of $W^k$ for $k = v, ..., N$, following recursive relations in the proof of Proposition \ref{familyofideals}, can be written as  
\begin{align*}
xy^k + \sum_{i=k+1}^{N} \sigma_{i-k+v} \cdot  y^i.
\end{align*}

\resizebox{5.8cm}{4cm}{
\begin{tikzpicture}

%begining layers
\node[cell] at (-1,-1) {\scalebox{.7}{$1$}};
\node[cell, minimum height=7cm] at (-1,0) {};
\node[cell, minimum width=11cm] at (0,-1) {\tiny $\ldots$};
\node[cell] at (10,-1) {\scalebox{.7}{$x^{i_0}$}};
\node[cell] at (-1,6) {\scalebox{.7}{$y^{N}$}};

%first/ k-1 th layer
\node[cell] at (-1,0) {\scalebox{.6}{$y$}};

\node[cell, minimum width=9cm] at (0,0) {\tiny $\ldots$};
\node[cell] at (9,0) {\scalebox{.6}{$x^{i_1+1}y$}};

%second layer

\node[cell] at (-1,1) {\scalebox{.5}{$y^{2}$}};
\node[cell, minimum width=7cm] at (0,1) {\tiny $\ldots$};
\node[cell] at (7,1) {\scalebox{.6}{$x^{i_2+1}y^2$}};

%third layer
\node[cell] at (-1,2) {\scalebox{.8}{$...$}};
\node[cell, minimum width=4cm] at (0,2) {\tiny $\ldots$};

%fourth layer
%\node[cell] at (0,3) {\scalebox{.8}{$...$}};
%\node[cell, minimum width=2.5cm] at (1,3) {\tiny $\ldots$};

%last layer
\node[cell, draw=none] at (0,3) {\tiny $x y^v$};

%vertical strip
\node[cell] at (-1,3) {\tiny $y^v$};

\node[cell] at (-1,4) {\tiny $y^{v+1}$};
\node[cell] at (-1,5) {\tiny \vdots};
\end{tikzpicture}
}

In order for $x^2y^v \in I$, we compute 
\begin{align*}
& x \cdot (xy^v +\sum_{i=v+1}^{N} \sigma_{i} \cdot  y^i)= x^2y^v + \sum_{i=v+1}^{N} \sigma_{i}xy^i\\
& = x^2y^v - \sum_{i=v+1}^{N} \sigma_{i} \sum_{j=i+1}^{N} \sigma_{j-i+v} \cdot y^j,
\end{align*}
and the coefficients of $y^j$ for $j=v+2,..., N$ should be set to $0$.

The coefficient of $y^j$ is $$\sum_{k_1,k_2} \sigma_{k_1} \cdot \sigma_{k_2},$$ such that $$k_1+k_2=j+v,$$ for $j=v+2,..., N$, and $k_1, k_2 \geq v+1$.

These equations are symmetric of the form:

\begin{align*}
&\sigma_{v+1} \cdot \sigma_{v+1}=0 \\
&\sigma_{v+1} \cdot  \sigma_{v+2}+\sigma_{v+2} \cdot \sigma_{v+1}=0 \\
&\sigma_{v+1} \cdot  \sigma_{v+3}+\sigma_{v+2} \cdot  \sigma_{v+2}+\sigma_{v+3} \cdot \sigma_{v+1}=0\\
& ...\\
&\sigma_{v+1} \cdot  \sigma_{N-1}+\sigma_{v+2} \cdot  \sigma_{N-2}+...+\sigma_{N-1} \cdot \sigma_{v+1}=0
\end{align*}

Solving these equations, starting from the first one, we notice that we obtain $$\sigma_{v+1} = \sigma_{v+2} = ... =\sigma_{\lfloor \frac{v+N}{2} \rfloor}=0.$$

So all the following coefficients are free:
$$\sigma_{\lfloor \frac{v+N}{2}\rfloor+1}, \sigma_{\lfloor \frac{v+N}{2}\rfloor+2}, ..., \sigma_{N}.$$

The total number of free variables is $$N-\lfloor \frac{v+N}{2}\rfloor = \lfloor \frac{2N}{2}\rfloor - \lfloor \frac{v+N}{2}\rfloor = \lfloor \frac{N-v}{2}\rfloor.$$

In other words, the boxes from the $v$th row to the $\lfloor \frac{v+N}{2}\rfloor$th row do not contribute to the dimension of the stratum, but the boxes from the $\lfloor \frac{v+N}{2}\rfloor+1$th row to the $N$th row each contribute to $1$ dimension to the stratum.

If $N-v$ is even, then $\lfloor \frac{N-v}{2}\rfloor = \frac{N-v}{2}$. If $N-v$ is odd, then $\lfloor \frac{N-v}{2}\rfloor = \frac{N-v-1}{2}$. Note that $N-v$ is the number of boxes above the box $y^v$ (starting from the box $y^{v+1}$) at the left-most column. Note that the box $y^v$ must be contained in the partition, because we assumed that the valuation has type $[1]$ and therefore $I$ does not contain $y^v$. So there must be at least boxes $1, y, ..., y^v$ at the left-most column in the partition.

For the rest of the boxes above the box $y^v$, their existence is optional, but for some even number of boxes we add on top, half of them contributes to the dimension of the stratum, and half of them do not; for some odd number of boxes, denoted by $2j+1$, we add on top of $y^v$, the dimension contributed by these boxes to the stratum is $j+1$, and the other $j$ boxes do not contributes to any dimension of the stratum. So the generating function of these boxes are $$\frac{(1+QT^2)}{(1-Q^2T^2)}.$$

For the rest of the rows, from the $0$th to the $v$th rows, their generating function, because we have fixed all the variables of the generator $xy^v$ at the layer $W^v$, the rest of the generators at the lower layers, $W^{v-1},..., W^{0}$ follow the same rules as in Proposition \ref{familyofideals}, and that each of them adds the dimension to the stratum that is equal to the number of boxes in the row right above it. Because there must be at least boxes $1, y, ..., y^v$ at the left-most column in the partition, there is a factor $$Q^{v+1}T^{2v}$$ in the generating function. The existence of any other boxes to the right of the left-most column (to the right of $1, y, ..., y^v$) is optional, and they contribute the factor $$\frac{1}{\prod_{i=1}^v (1-Q^iT^{2(i-1)})}$$ to the generating function. 

The generating function is therefore 
$$\frac{(1+QT^2)Q^{v+1}T^{2v}}{(1-Q^2T)\prod_{i=1}^v (1-Q^iT^{2(i-1)})}$$

\item If the valuation $v_{(0,v)}(W^v)$ has type $[2]$, then the indexing partition of $I$ looks like the following.

\resizebox{5.8cm}{4cm}{
\begin{tikzpicture}

%begining layers
\node[cell] at (-2,-1) {\scalebox{.7}{$1$}};
\node[cell, minimum height=7cm] at (-2,0) {};
\node[cell, minimum width=11cm] at (-1,-1) {\tiny $\ldots$};
\node[cell] at (10,-1) {\scalebox{.7}{$x^{i_0}$}};
\node[cell] at (-2,6) {\scalebox{.7}{$y^{N}$}};

%first/ k-1 th layer
\node[cell] at (-2,0) {\scalebox{.6}{$y$}};

\node[cell, minimum width=10cm] at (-1,0) {\tiny $\ldots$};
\node[cell] at (9,0) {\scalebox{.6}{$x^{i_1+1}y$}};

%second layer

\node[cell] at (-2,1) {\scalebox{.5}{$...$}};
\node[cell, minimum width=8cm] at (-1,1) {\tiny $\ldots$};

%fourth layer
%\node[cell] at (0,3) {\scalebox{.8}{$...$}};
%\node[cell, minimum width=2.5cm] at (1,3) {\tiny $\ldots$};

%last layer
\node[cell] at (-2,2) {\tiny $y^v$};
\node[cell] at (-1,2) {\tiny $xy^v$};
\node[cell, draw=none] at (0,2) {\tiny $x^2 y^v$};

%second-vertical strip
\node[cell] at (-1,3) {\tiny $...$};

\node[cell, draw=none] at (-1,4) {\tiny $xy^{M}$};
\node[cell] at (-2,5) {\tiny \vdots};

%first-vertical strip
\node[cell] at (-2,3) {\tiny $...$};

\node[cell] at (-2,4) {\tiny $y^{M}$};
\node[cell] at (-2,5) {\tiny \vdots};

\end{tikzpicture}
}

Based on the rules in Proposition \ref{familyofideals} we have computed, either we have a nontrival generator $$xy^M+\sum_{i=M+1}^{N} \sigma_{i} y^i,$$ or $M=N$ and all the coefficients of the generator with $x^2y^v$ as the \textcolor{red!80!black}{Head} part are independent.

For the generator $xy^M+\sum_{i=M+1}^{N} \sigma_{i} y^i$, all the coefficients $\sigma_{i}$ are free. Because $x^2y^v \in I$, we also have that $x^2y^M \in I$.

For $k = M, ..., N$, the generators can be written as  
\begin{align*}
xy^k + \sum_{i=k+1}^{N} \sigma_{i-k+M} \cdot  y^i.
\end{align*}

Now we compute 
\begin{align*}
& x \cdot (xy^M+\sum_{i=M+1}^{N} \sigma_{i} y^i) \\
& = x^2y^M+\sum_{i=M+1}^{N} \sigma_{i} xy^i \\
&= x^2y^M-\sum_{i=M+1}^{N} \sigma_{i} (\sum_{j=i+1}^{N} \sigma_{j-i+M} \cdot  y^j)
\end{align*}

Notice that for the coefficient of the monomial $y^j$, we always have a sum of a product $\sum \sigma_{k_1} \cdot  \sigma_{k_2}$ of $\sigma_{k_1}$ and $\sigma_{k_2}$ such that $k_1+k_2=M+j$, where $j = M+2,...,N$.

So for the coefficients $y^j$ where $j=M+2,...,N$, we obtain the following equations.

\begin{align*}
&\sigma_{M+1} \cdot \sigma_{M+1}=0 \\
&\sigma_{M+1} \cdot  \sigma_{M+2}+\sigma_{M+2} \cdot \sigma_{M+1}=0 \\
& ...\\
\end{align*}

Generically, we have the equations for $K=2M+2,..., M+N$, $k_1, k_2 \geq M+1$: 
$$\sum_{k_1+k_2=K} \sigma_{k_1} \cdot  \sigma_{k_1}.$$

And solving these equations we get that for all $M+1 \leq i \leq \lfloor \frac{M+N}{2} \rfloor $, $\sigma_{i}=0.$ These are the independent variables we kill, which is equal to $\lfloor \frac{M+N}{2} \rfloor- M= \lfloor \frac{N-M}{2} \rfloor$. The number of free variables are the number of boxes between the $\lfloor \frac{M+N}{2} \rfloor +1$th row and the $N$th row, which is the number of half of the boxes between $y^{M+1}$ and $y^N$, depending on even and odd parity.

Once we establish the condition that $x^2y^M \in I$, the generator of $W^{M-1}$, given as $$xy^{M+1}+ \sum_{j=M}^N \lambda_j y^j,$$ has one free variable, $\lambda_N$ which must be set to $0$, and are otherwise already set to $0$: setting $\lambda_j$ to $0$ for $j=M, ..., N-1$ does not induce any additional condition on the other free variables introduced in the previous layers, because setting the coefficients of $x^2y^M$ to be $0$, as we did previously, already took care of these conditions. 

Similar to what we computed in case (2), the generating function for the boxes at $y^M$ and above is given as $$\frac{1+Q}{1-Q^2T^2}.$$

Once the generator of the $M-1$th row, whose \textcolor{red!80!black}{Head} part is $x^2y^{M-1}$, is fixed, the rest of the generators for the $k$th row where $k = v, ..., M-2$ follow the rules of Proposition \ref{familyofideals}. The $2$ degrees of freedom added by each of them must be set to $0$, due to the condition that $x^2y^v \in I$. So these rows with two boxes for the layers $W^{M-1}, ..., W^{v+1}$ does not add any degree of freedom. Their generating series is $$\frac{1}{1-Q^2}.$$

For the rest of the rows at the level of $y^v$ and below, once we establish the condition that $x^2y^v \in I$, their generators are not constrained from extra conditions, and follow only the rules as in Proposition \ref{familyofideals}, and note that we must have the box $xy^v$ in the partition because the valuation is of type $[2]$, so there must be $2(v+1)$ boxes to the left and below $xy^v$, such that the generating function of these $2(v+1)$ boxes is $Q^{2(v+1)}T^{4v}$. Accounting for the rest of the boxes, the generating function is $$\frac{Q^{2(v+1)}T^{4v}}{\prod_{i=1}^{v}(1-Q^{i}T^{2(i-1)})}.$$

The generating function of this case is $$\frac{Q^{2(v+1)}T^{4v}(1+Q)}{(1-Q^2)(1-Q^2T^2)\prod_{i=1}^{v}1-Q^{i}T^{2(i-1)}}$$

\end{enumerate}

In total, the generating function adding all three cases together is \begin{align*}
&\frac{(1-Q)(1-Q^2T^2)+(1-Q)(1+QT^2)Q^{v+1}T^{2v} + Q^{2(v+1)}T^{4v}}{(1-Q)(1-Q^2T^2)\prod_{i=1}^{v}(1-Q^{i}T^{2(i-1)}}.\\
%&=\frac{Q^{3}T-Q^{2}T-Q+1+(-Q^{2}T+Q\,T-Q+1)}{(1-Q)(1-Q^2T)\prod_{i=1}^{v}1-Q^{i}T^{2(i-1)}}\\
\end{align*}
\end{proof}
 
\begin{conjecture}
The colored knot homology is \begin{align*}
\frac{(1-Q)(1-Q^2T^2)+(1-Q)(1+QT^2)Q^{v+1}T^{2v} + Q^{2(v+1)}T^{4v}}{(1-Q)(1-Q^2T^2)\prod_{i=1}^{v}(1-Q^{i}T^{2(i-1)}}.
\end{align*}
\end{conjecture}

This stratification is still reasonably nice when $u=2$, but for $u=3$ and bigger, there soon arises many different cases to discuss, for example the number of identical columns at the left, and the author tried of course to find some pattern in general for any arbitrary $u,v$, but there does not seem to be much unified structure, but rather things separates into a large number of different cases to discuss. The idea is still simple as the cases discussed before, to write out all the coefficients of the generator $x^uy^v+...$, and set them to be $0$, but some of these coefficients are free, and some of them are dependent on the free ones of the previous layers. So setting all the coefficients of the generator $x^uy^v+...$ to $0$ is not just killing the degree of freedom that this single generator itself brings in, but also induces equations on the free variables of the previous layers. This is a rather unfortunate situation, because the construction of the generators grows from top to bottom, so this condition is in some sense backwards and unnatural to the stratification, as we need to use the bottom layer to determine the top layers. But of course, we can still brute force, and if one really wants to, they can write several pages more and compute the cases for $u=3$ and $u=4$, but the author does not find such computations atheistical or meaningful, and would rather see another method coming in for arbitrary $u,v$.

\section{Poincar\'e polynomial of $\Hilb^n(\{x^vy^v=0\},0)$}

\begin{theorem}
\label{thmdiag}
Every Diagonal stratum $\Sigma_{([\circled{s}_k, (i_k,j_k)])}$ of the Hilbert schemes, $\Hilb^n(\{x^uy^v=0\},0)$, where $u$ is equal to $v,v-1$ or $v-2$, is isomorphic to the product of an algebraic torus times an affine space,

$$\Sigma_{([\circled{s}_k, (i_k,j_k)])}= (\mathbb{A}-pt)^{m_1} \times \mathbb{A}^{m_2},$$ for some non-negative integers $m_1,m_2$. 
The non-negative integer $m_1$ is equal to the number of type $\circled{1}$ layers plus $1$,  such that if the $\circled{1}$ layer appears consecutively after another $\circled{1}$ layer, the number of boxes of the rows strictly increase. The non-negative integer $m_2$ equals to the total number of boxes from the $1$st to the $v-1$th row in this weak diagonal partition, (the total number of boxes minus the number of boxes in the $0$th row,) plus the number of consecutive type $\circled{2}$ rows, where each two consecutive type $\circled{2}$ rows adds one to $m_2$.
\end{theorem}

\begin{proof}

\begin{enumerate}
\item 
%It is clear from the construction of the Diagonal filtration that for strata of $\Hilb^n(\{x^uy^v=0\},0)$, all the free variables of all generators of all layers of the filtration remains to be free, and there are no constrains. 

The condition $x^{v}y^v \in I$ is equivalent to the total number of rows of the weak diagonal partition less than or equal to $v$. Fixing the maximum number of rows to be $v$, we compute explicitly the degree of freedom added by each generator. We are exactly translating Proposition \ref{prop=} into the language of weak diagonal partitions. 

Each pair of generators of type $E_1$ and $E_2$ in the $k$th row adds degree of freedom equal to the total number of boxes in the $k+1$th row, unless the previous row also has generator of type $E_1, E_2$, then they add one more degree of freedom. Each \textcolor{red!80!black}{Head} part of the generator $E$ of the $k$th layer adds $\mathbb{P}^1-2pts$ degree of freedom, unless the generator of the $k+1$ layer is also $E$ and the $k$th and $k+1$th rows have the same number of boxes, then the \textcolor{red!80!black}{Head} part of $k$th layer is dependent on the $k+1$th layer and does not add any degree of freedom. The Tail part of the generator $E$ adds the total degree of freedom equal to the number of boxes in the previous layer.

So we can conclude that $m_1$ is the number of independent \textcolor{red!80!black}{Head} parts of $E$ of all layers, and $m_2$ is the number of boxes of all rows except for the outmost layer, plus the number of consecutive $\circled{2}$'s.

\item
When $u=v-1$, then $x^{v-1}y^v \in I$, which implies that the $v-1$th layer is either generated by $x^{v-1}y^{v-1}$, corresponding to a weak diagonal partition of $v-1$ rows or less, or the $v-1$th layer must be of type $\circled{2}$, with $E_2=x^{v-1}y^v$. This means that the $v-1$th row of the corresponding weak diagonal partition must be of type $\circled{2}$ and the vertical arm has exactly $1$ box. So comparing to the Hilbert scheme $\Hilb^n(\{x^vy^v=0\},0)$, the Hilbert scheme $\Hilb^n(\{x^{v-1}y^v=0\},0)$ is missing a few strata where the $v$th row of the indexing partition has other type or length. But all the strata of $\Hilb^n(\{x^{v}y^v=0\},0)$ are either completed contained or completely disjoint from $\Hilb^n(\{x^{v-1}y^v=0\},0)$, and the arguments we just made for $u=v$ still applies to the case when $u=v-1$.

\item
The case when $u=v-2$ is similar but just slightly more complicated combinatorially, as we require both $x^{v-2}y^v \in I$ and $x^{v-1}y^v \in I$. As explained earlier, the $v-1$th layer must have type $\circled{2}$ or be generated as $x^{v-1}y^{v-1}$. %If the $v-1$th layer is generated as $x^{v-1}y^{v-1}$, then this question is reduced to all the strata of the Hilbert scheme $\Hilb^n(\{x^{v-1}y^{v-1}=0\},0)$, and clearly the theorem holds following from the discussion of $u=v$. If the $v-1$th layer has type $\circled{2}$, then it must be $\circled{2}(i_{v-1},1)$ where $E_2=x^{v-1}y^v$ and $E_1=x^{v-1}y^{v-1+i_{v-1}}$. 

By the combinatorial rules of Proposition \ref{prop=}, if the $v-2$th layer has type $\circled{2}$ the $v-2$th layer must be $[\circled{2},(i_{v-2},2)]$ where $i_{v-2} \geq i_{v-1}+1$. If the $v-2$th layer has type $\circled{1}$, then it must be $[\circled{1},(i_{v-2},1)]$ where $i_{v-2} \geq i_{v-1}$. This generator $E$ has the \textcolor{red!80!black}{Head} part equal to $(ax^{i_{v-2}}+by) x^{v-2}y^{v-2}$ However, in order for $x^{v-2}y^v \in I$, we need to shift the \textcolor{red!80!black}{Head} part up by $(0,+1)$, which correspond to multiplication by $y$, and require by the shape of the valuation that $$y(ax^{i_{v-2}}+by) x^{v-2}y^{v-2}= ax^{i_{v-2}+v-2}y^{v-1}+bx^{v-1}y^{v} \in I.$$
This holds true if and only if $i_{v-2}+v-1 \geq i_{v-1}+v$, which implies that $i_{v-2} \geq i_{v-1}+1$. Again, this condition just puts constrains on the location of boxes in the \textcolor{red!80!black}{Head} parts of generators, which means just ruling out some possible shapes of the diagonal partition. So a stratum of $\Hilb^n(\{x^{v}y^{v}=0\},0)$ is either fully contained in $\Hilb^n(\{x^{v-2}y^{v-2}=0\},0)$ or completely disjoint from it. The rules of the number of free variables $m_1$ and $m_2$ for each stratum applies to every stratum with any possible weak diagonal partition. So once the stratum is contained in $\Hilb^n(\{x^{v-2}y^{v-2}=0\},0)$, the rules of variables $m_1$ and $m_2$ remain the same.
\end{enumerate}
\end{proof}

\subsection{Minimal shapes of weak diagonal partitions.\\}

Recall that a weak diagonal partition of $r$ rows can be represented by a free word of $\circled{1}, \circled{2}$ of length $r$, and a length $r$ sequence of pairs of positive integers,  $$([\circled{s}_k, (i_k, j_k)])=[\circled{s}_0, (i_0, j_0)],[\circled{s}_1, (i_1, j_1)], ...,[\circled{s}_{r-1}, (i_{r-1}, j_{r-1})].$$

%where $k=r-1,...,0$, and each $\circled{s}_k$ is either $\circled{1}$ or $\circled{2}$, recording the type of the $k$th row, and $(i_k, j_k)$ is a pair of positive integers recording the location of the boxes at the end of each row.

\begin{definition}
Given a weak diagonal partition of $r$ rows, denoted by $([\circled{s}_k, (i_k, j_k)])$, there exists a unique minimal weak diagonal partition, denoted by $([\circled{s}_k, (m_k, n_k)])$, such that the relationship between the consecutive rows of type $\circled{1}$'s (either going up, denoted as $\circled{1}_U\circled{1}$, or equal, denoted as $\circled{1}_E\circled{1}$, as in Notation \ref{1U1}) are the same as $([\circled{s}_k, (i_k, j_k)])$; and the pair $(m_k,n_k)$ is the smallest among all the possible weak diagonal partitions of $r$ rows where the $k$th row has type $\circled{s}_k$ for all $k=1,..., r-1$.
\end{definition}

\begin{remark}
We can keep only the essential information, the free word of $\circled{1}, \circled{2}$, and drop the pairs of numbers $(i_k,j_k)$ in the minimal weak diagonal partition, while keeping in mind that there are two possibilities for any consecutive $\circled{1}\circled{1}$in the free word.
So a minimal weak diagonal partition of $k$ rows can be uniquely denoted as a word in the free group generated by $\circled{1}$ and $\circled{2}$ of length $k$, where between each two consecutive $\circled{1}$'s, the relationship are labeled as either $\circled{1}_E\circled{1}$ or $\circled{1}_U\circled{1}$.
\end{remark}

\begin{example}
\begin{enumerate}
\item The minimal weak diagonal partitions of $1$ row are 

$[\circled{1},(1,1)]$ of $2$ boxes, and $[\circled{2},(1,1)]$ of $1$ box.

\item The minimal weak diagonal partitions of $2$ rows are:

\begin{enumerate}
\item $[\circled{1},(1,1)], [\circled{1},(1,1)]$ of $4$ boxes, denoted as $\circled{1}_E\circled{1}$. 

\item $[\circled{1},(2,2)], [\circled{1},(1,1)]$ of $6$ boxes, denoted as $\circled{1}_U\circled{1}$.

\item $[\circled{1},(1,1)], [\circled{2},(1,1)]$ of $3$ boxes, denoted as $\circled{1}\circled{2}$.

\item 
$[\circled{2},(2,2)], [\circled{1},(1,1)]$ of $5$ boxes, denoted as $\circled{2}\circled{1}$.

\item $[\circled{2},(2,2)], [\circled{2},(1,1)]$ of $4$ boxes, denoted as $\circled{2} \circled{2}$.
\end{enumerate}

\end{enumerate}
\end{example}

\begin{lemma}
Collecting all the minimal weak diagonal partition of length $k$ into groups represented by a free word of $\circled{1}, \circled{2}$ of length $k$. Equivalently when $k \geq 2$, it can be represented by a length $k$ walk in the following weighted graph, where the adjacency matrix $M_s(s,k)$ is a $2$ by $2$ matrix,

$$M_s =
\begin{bmatrix}
\circled{1} \rightarrow \circled{1} & \circled{2} \rightarrow \circled{1}  \\ 
\circled{1} \rightarrow \circled{2} & \circled{2} \rightarrow \circled{2}
\end{bmatrix}.$$

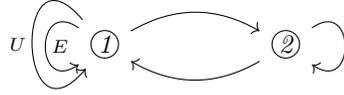
\begin{center}

\begin{tikzcd}
\circled{1} \arrow[rr, bend left] \arrow[loop, distance=4em, in=230, out=130, "U" '] \arrow[loop, distance=2em, in=215, out=145, "E"] &  & \circled{2} \arrow[ll, bend left] \arrow[loop, distance=2em, in=325, out=35]
\end{tikzcd}

\end{center}

\end{lemma}

\begin{proof}
When we have two consecutive $\circled{1}$'s in the free word, we have either $\circled{1}_U\circled{1}$ or $\circled{1}_E\circled{1}$, and they are the two directed edges going from the vertex $\circled{1}$ to itself. Otherwise, going from $\circled{1}$ to $\circled{2}$, or going from $\circled{2}$ to $\circled{1}$, or going from $\circled{2}$ to $\circled{2}$ is unique and has $1$ edge between the corresponding verticies.

\end{proof}

\begin{theorem}
\label{thmv}
The virtual Poincar\'e polynomial of the Hilbert scheme $Hilb^n(\{x^vy^v=0\},0)$ is 

\begin{align*}
\sum_{n=0}^{\infty}\chi(\Hilb^n(\{x^vy^v=0\},0))Q^n=\sum_{k=1}^v \frac{F(k)}{\prod_{i=1}^{k}(1 - Q^iT^{2(i-1)})^2}.
\end{align*}

And $F(k)$ is the generating function of all the strata whose indexing partitions are minimal weak diagonal partitions of exactly $k$ layers, given as follows
$$
F(k)=\begin{cases}
Q^2T^2-Q+1, & \text{if } k=1, \\

\mathbf{1}^T \cdot V_k, & \text{if } k =2,...,v,
\end{cases}
$$  

where $\mathbf{1} =\begin{bmatrix}1 & 1 \end{bmatrix}$ is a row vector of ones, and $\mathbf{1}^T \cdot V_k$ is the sum of the two entries of the vector $V_k$.

The vectors $V_k$ is given as $$V_k= \prod_{s=2}^{k} M_s \cdot V_1.$$

The initial vector $V_1$ is $$V_1=
\begin{bmatrix}
Q^{2k}T^{4(k-1)}(T^2-1) \\ Q^kT^{2(k-1)}
\end{bmatrix}.$$

For $s=2, ..., k$, the adjacency matrix $$M_s =
\begin{bmatrix}
Q^{2(k-s+1)}T^{4(k-s)}(T^2-1)+(1-Q^{k-s+1}T^{2(k-s)})^2 & Q^{k-s+1}T^{2(k-s)}(T^2-1)  \\ 
Q^{k-s+1}T^{2(k-s)} & Q^{2(k-s+1)}T^{4(k-s)}T^2
\end{bmatrix}.$$

\end{theorem}
 
\begin{proof}

For each stratum, recall in Theorem \ref{thmdiag} that we have already described the shapes of its corresponding weak diagonal partition of $n$ boxes, and described the dimension of the stratum based on the statistics of the diagonal partition. Now we only have to count the number of each weak diagonal partition of each dimension. 

The proof has four steps. 

\textbf{Step 1, reduce a weak diagonal partition to a minimal one. } Every weak diagonal partition can uniquely shrink to a minimal one, and conversely, given all minimal weak diagonal partition of $k$ rows, one can generate all the weak diagonal partition by adding two regular partitions of at most $k$ rows on both the horizontal and vertical arms. (When we say regular partitions, we mean the ones where columns are vertical and rows horizontal, the usual, non-diagonal partitions.) 

Recall in Theorem \ref{thmdiag} that the number of points in the Hilbert scheme is the total number of boxes in the weak diagonal partition, and the dimension of the stratum is the total number of boxes of all but the outmost row, plus dimension contributions from the rows $\circled{1}_U \circled{1}$, plus dimension contributions from the consecutive twos rows of type $\circled{2}$.

So for a virtual Poincar\'e polynomial of a weak diagonal partition of $k$ rows in total, namely from $k-1$st row to $0$th row, it has two factors: One of them is given as the virtual Poincar\'e polynomial of the corresponding minimal weak diagonal partition, where all the dimension contribution of $\circled{1}_U \circled{1}$, and twos consecutive rows of type $\circled{2}$ are considered. The dimension contribution from all the boxes in the minimal weak diagonal partition is also counted in the virtual Poincar\'e polynomial of the minimal weak diagonal partition itself. Another factor is the virtual Poincar\'e polynomial of the two regular partition that we each add on the vertical and horizontal arms, recording only the contribution of dimension from all the boxes involved. 

For the two regular partitions that we add, the outmost row (lowest row) contributes only to the number of points in the Hilbert scheme but not the dimension, and all the other boxes in other rows contribute to both the number of points in the Hilbert scheme and the dimensions of the stratum. This contributes the factor $$\frac{1}{\prod_{i=1}^{k}(1 - Q^iT^{2(i-1)})^2}$$ to the virtual Poincar\'e polynomial.

The only thing to be careful about is when the relation $\circled{1}_E \circled{1}$ applies to, for example the $(s+1)$th and the $s$th rows, then the total number of boxes on these two consecutive rows must be the same, so we are not allowed to add just any regular partition on the vertical or horizontal arms. We must add regular partitions whose $s+1$th and the $s$th rows have the same length. This is equivalent to multiplying the factor $$(1 - Q^{k-s+1}T^{2(k-s)})^2$$ to the virtual Poincar\'e polynomial of such minimal weak diagonal partition to cancel out with the freedom in the denominator of $\frac{1}{\prod_{i=1}^{k}(1 - Q^iT^{2(i-1)})^2}.$

\textbf{Step 2, computing the adjacency matrix for transiting from the $s+1$th row to the $s$th row.}

Denote the virtual Poincar\'e polynomial of all strata of minimal weak diagonal partitions of exactly $k$ rows by $F(k)$. Now keeping in mind of the case $\circled{1}_E \circled{1}$ and add the factor $(1 - Q^{k-s+1}T^{2(k-s)})^2$ when needed, we compute the $F(k)$.

We inductively build a minimal weak diagonal partition of $k$ rows by adding rows from the innermost row to the outermost rows. The virtual Poincar\'e polynomial of a minimal weak diagonal partition of $k$ rows has two parts, the contribution of dimension and number of points from the boxes themselves, and the additional factors contributing to the dimension we have to consider.

If we add a row of $\circled{1}$ on top of $\circled{1}$ to construct the $s$th row from the $s+1$th row, then there are two possibilities. 

If the number of boxes go up by $2$, then we add in total $2(k-s+1)$ boxes, for we need to add $2$ more boxes for all rows between the $s$th row and the $0$th row, so there are in total $k-s+1$ rows to be added $2$ boxes. The dimension of the stratum increase by the total number of boxes in all but the outmost row, so the dimension increases by $2 \cdot 2(k-s)$. Then we also need to consider that there is another contribution of $\mathbb{P}^1-2pts$ by the generators of $\circled{1}$ in the $s$th row, which is $T^2-1$. 

If the number of boxes between two consecutive $\circled{1}$ rows are equal, then there are no increase of number of boxes or the dimension of the stratum at the level of the $s$th row. But we have to be careful that the $s$th row and the $s+1$th row must have the same number of boxes. So in order to cancel out the freedom of moving the boxes to have different length, we multiply by $(1-Q^{k-s+1}T^{2(k-s)})^2$.

Adding a row of $\circled{2}$ on top of $\circled{1}$, we must fill the two shaded boxes at the end of row $\circled{1}$ to turn them into solid boxes, and therefore adding two half boxes, and in total $1$ box at the ends of every row between the $s$th row and the $0$th row, adding the factor $Q^{k-s+1}T^{2(k-s)}$.

Adding a row of $\circled{1}$ on top of $\circled{2}$, we must add two shaded boxes at the end of row $\circled{2}$, and therefore adding in total $1$ box at the ends of every row between the $s$th row and the $0$th row, adding the factor $Q^{k-s+1}T^{2(k-s)}$. The row $\circled{1}$ also adds the freedom of $\mathbb{P}^1-2pts$, so there is also a factor of $T^2-1$.

Adding a row of $\circled{2}$ on top of $\circled{2}$, we must add two solid boxes from the $s$th to the $0$th row, so there is a factor of $Q^{2(k-s+1)}T^{4(k-s)}T^2$. There is also an extra factor of $T^2$ due to the rule that the dimension of the stratum increase by $1$ when we have two consecutive $\circled{2}$.

So we have proved the entries of the transition matrix $M_s$.

\textbf{Step 3, the initial vector.}
If we start with a row of type $\circled{1}$, then there are $2$ boxes with dimension contributed by these two boxes, plus an additional factor of $T^2-1$, because it's a row of type $\circled{1}$, so the factor is $Q^{2k}T^{4(k-1)}(T^2-1)$. If we start with a row of type $\circled{2}$, then there is $1$ box with dimension contributed only by this box, so the factor is $Q^kT^{2(k-1)}$.

\textbf{Step 4, the case when $k=1$.}

When we only have $1$ row, then the matrix $M_s$ or the initial vector $V_1$ is not needed. We compute directly that type $\circled{1}$ contributes a summand $Q^2(T^2-1)$ and type $\circled{2}$ contributes a summand $Q$. In total we have 
$$\frac{Q^2T^2-Q^2+Q}{(1-Q)^2}.$$ But we need to also count the case with $0$ boxes, which contributes to the summand $1$.  

Adding them all together, we have $$1+\frac{Q^2T^2-Q^2+Q}{(1-Q)^2}= \frac{Q^2T^2-Q+1}{(1-Q)^2}.$$ 
\end{proof}

We have a Macaulay2 code that computes the virtual Poincar\'e polynomial as both rational functions and power series for any arbitrary positive integer $v$, see the Appendix Section \ref{app}, and we list the first few $v=2,3,4$ as examples.
\begin{example}
\label{ex1.2}
\begin{enumerate}
\item 
\label{1.2}
When $v=2$, the virtual Poincar\'e polynomial is 
\begin{align*}
&\chi(\Hilb^n(\{x^2y^2=0\},0))\\
&=\frac{Q^{6}T^{8}-Q^{5}T^{6}+Q^{4}T^{6}-Q^{4}T^{4}+Q^{3}T^{4}+Q^{3}T^{2}-Q^{2}T^{2}-Q+1}{Q^{6}T^{4}-2\,Q^{5}T^{4}+Q^{4}T^{4}-2\,Q^{4}T^{2}+4\,Q^{3}T^{2}-2\,Q^{2}T^{2}+Q^{2}-2\,Q+1}. 
\end{align*}

\item When $v=3$, the virtual Poincar\'e polynomial is

$\chi(\Hilb^n(\{x^3y^3=0\},0))
=(Q^{12}T^{18}-Q^{11}T^{16}+Q^{10}T^{16}-Q^{10}T^{14}+Q^{9}T^{14}+Q^{8}T^{10}-2\,Q^{7}T^{10}+Q^{6}T^{10}+Q^{7}T^{8}-2\,Q^{6}T^{8}+Q^{5}T^{8}-Q^{6}T^{6}+Q^{4}T^{6}+Q^{4}T^{4}-Q^{3}T^{4}+Q^{3}T^{2}-Q^{2}T^{2}-Q+1)/(Q^{12}T^{12}-2\,Q^{11}T^{12}+Q^{10}T^{12}-2\,Q^{10}T^{10}+4\,Q^{9}T^{10}-2\,Q^{8}T^{10}-2\,Q^{9}T^{8}+5\,Q^{8}T^{8}-4\,Q^{7}T^{8}+Q^{6}T^{8}+4\,Q^{7}T^{6}-8\,Q^{6}T^{6}+4\,Q^{5}T^{6}+Q^{6}T^{4}-4\,Q^{5}T^{4}+5\,Q^{4}T^{4}-2\,Q^{3}T^{4}-2\,Q^{4}T^{2}+4\,Q^{3}T^{2}-2\,Q^{2}T^{2}+Q^{2}-2\,Q+1).$
\medskip

\item When $v=4$, the virtual Poincar\'e polynomial is 

$\chi(\Hilb^n(\{x^4y^4=0\},0)) =(Q^{20}T^{32}-Q^{19}T^{30}+Q^{18}T^{30}-Q^{18}T^{28}+Q^{17}T^{28}-Q^{15}T^{24}+Q^{14}T^{24}+2\,Q^{15}T^{22}-3\,Q^{14}T^{22}+Q^{13}T^{22}-2\,Q^{13}T^{20}+2\,Q^{12}T^{20}-Q^{12}T^{16}+2\,Q^{11}T^{16}-Q^{10}T^{16}-Q^{11}T^{14}+3\,Q^{10}T^{14}-3\,Q^{9}T^{14}+Q^{10}T^{12}+Q^{8}T^{14}+Q^{9}T^{12}-3\,Q^{8}T^{12}+Q^{7}T^{12}-Q^{7}T^{10}+Q^{6}T^{10}-Q^{7}T^{8}+Q^{5}T^{8}-Q^{6}T^{6}+2\,Q^{5}T^{6}-Q^{4}T^{6}+Q^{4}T^{4}-Q^{3}T^{4}+Q^{3}T^{2}-Q^{2}T^{2}-Q+1)/(Q^{20}T^{24}-2\,Q^{19}T^{24}+Q^{18}T^{24}-2\,Q^{18}T^{22}+4\,Q^{17}T^{22}-2\,Q^{16}T^{22}-2\,Q^{17}T^{20}+5\,Q^{16}T^{20}-4\,Q^{15}T^{20}-2\,Q^{16}T^{18}+Q^{14}T^{20}+8\,Q^{15}T^{18}-10\,Q^{14}T^{18}+4\,Q^{13}T^{18}+5\,Q^{14}T^{16}-12\,Q^{13}T^{16}+9\,Q^{12}T^{16}+4\,Q^{13}T^{14}-2\,Q^{11}T^{16}-12\,Q^{12}T^{14}+12\,Q^{11}T^{14}+Q^{12}T^{12}-4\,Q^{10}T^{14}-10\,Q^{11}T^{12}+18\,Q^{10}T^{12}-10\,Q^{9}T^{12}-4\,Q^{10}T^{10}+Q^{8}T^{12}+12\,Q^{9}T^{10}-12\,Q^{8}T^{10}-2\,Q^{9}T^{8}+4\,Q^{7}T^{10}+9\,Q^{8}T^{8}-12\,Q^{7}T^{8}+5\,Q^{6}T^{8}+4\,Q^{7}T^{6}-10\,Q^{6}T^{6}+8\,Q^{5}T^{6}+Q^{6}T^{4}-2\,Q^{4}T^{6}-4\,Q^{5}T^{4}+5\,Q^{4}T^{4}-2\,Q^{3}T^{4}-2\,Q^{4}T^{2}+4\,Q^{3}T^{2}-2\,Q^{2}T^{2}+Q^{2}-2\,Q+1).$
\end{enumerate}
\end{example}

\begin{remark}
The numerators in the examples above are all irreducible and cannot be factored, but the highest degree is always $Q^{v(v+1)}T^{2v^2}$. The denominators are $\prod_{i=1}^{v}(1 - Q^iT^{2(i-1)})^2$.
\end{remark}

\section{Hilbert scheme of points on $\mathbb{C}^2$.}

\begin{lemma}
When $2v \geq n$, we have the isomorphism between the Hilbert scheme of points on a non-reduced nodal singularity, and the Hilbert scheme of points on the plane $\mathbb{C}^2$, 

$$\Hilb^n(\{x^{v}y^v=0\},0) \cong \Hilb^n(\mathbb{C}^2,0).$$
\end{lemma}

\begin{proof}
The inclusion $\Hilb^n(\{x^{v}y^v=0\},0) \subset \Hilb^n(\mathbb{C}^2,0)$ is immediate, now we show the other inclusion $\Hilb^n(\mathbb{C}^2,0) \subset \Hilb^n(\{x^{v}y^v=0\},0)$ as follows.
For an ideal $I \subset \mathbb{C}[\![x,y]\!]$, if its Diagonal valuation is a weak diagonal partition of $n$ boxes, then $\frac{\mathbb{C}[\![x,y]\!]}{I}$ has dimension $n$, so $I \in \Hilb^n(\mathbb{C}^2,0)$. We want to show that when $2v \geq n$, every ideal $I$ which maps to a weak diagonal filtration of $n$ boxes satisfy that $x^vy^v \in I$.  

The least number of boxes we must have for each layer of the weak diagonal partition is either $1$ box of type $\circled{2}$ for the innermost layer, or $2$ boxes of type $\circled{1}$ for the rest of the layers. So when $2v \geq n$, all possible weak diagonal partitions of $n$ boxes have $v$ rows or less, which implies that $x^vy^v \in I$ for all $I \in \mathbb{C}[\![x,y]\!]$. So when $2v \geq n$, any $I \in \Hilb^n(\mathbb{C}^2,0)$ is also in $\Hilb^n(\{x^{v}y^v=0\},0)$.
\end{proof}

\begin{theorem}

The Hilbert scheme $\Hilb^n(\mathbb{C}^2,0)$ has a stratification where each stratum is indexed by a weak diagonal partition of $n$ boxes, denoted as $\Sigma_{([\circled{s}_k, (i_k,j_k)])}$ where $ 0 \leq k \leq \lceil \frac{n}{2} \rceil -1$, and each stratum is an algebraic torus times an affine space,

$$\Sigma_{([\circled{s}_k, (i_k,j_k)])} \cong (\mathbb{A}-pt)^{m_1} \times \mathbb{A}^{m_2},$$ for some non-negative integers $m_1,m_2$. 

The non-negative integer $m_1$ is equal to the number of type $\circled{1}$ layers plus $1$,  such that if the $\circled{1}$ layer appears consecutively after another $\circled{1}$ layer, the number of boxes of the rows strictly increase. The non-negative integer $m_2$ equals to the total number of boxes from the $1$st to the $v-1$th row in this weak diagonal partition, (the total number of boxes minus the number of boxes in the $0$th row,) plus the number of consecutive type $\circled{2}$ rows, where each two consecutive type $\circled{2}$ rows adds one to $m_2$.
    
\end{theorem}

\begin{proof}
Because $\Hilb^n(\{x^{v}y^v=0\},0) \cong \Hilb^n(\mathbb{C}^2,0)$ when $2v \geq n$, we take some arbitrary $v$ large enough that $2v \geq n$, so $\Hilb^n(\{x^{v}y^v=0\},0)$ has the stratification into strata indexed by weak diagonal partitions of $n$ boxes, which exactly becomes the the stratification of $\Hilb^n(\mathbb{C}^2,0)$.
\end{proof}

\begin{theorem}
The punctual Hilbert scheme of points on the plane $\mathbb{C}^2$ $\Hilb^n(\mathbb{C}^2,0)$ has the following virtual Poincar\'e polynomial
\begin{align*}
&\sum_{n=0}^{\infty}\chi(\Hilb^n(\mathbb{C}^2,0))Q^n  = \sum_{k=1}^\infty \frac{F(k)}{\prod_{i=1}^{k}(1 - Q^iT^{2(i-1)})^2}.
\end{align*}

\end{theorem}

\begin{proof}
Take the sum of the virtual Poincar\'e polynomials of of all possible weak diagonal partitions of all $n$. This is exactly equivalent to considering all possible weak diagonal partitions of all numbers of rows $k$, where $k$ goes from $1$ to infinity. 
\end{proof}

\begin{corollary}
\label{durfee}
We have the combinatorial identity coming from the virtual Poincar\'e polynomial of $\Hilb^n(\mathbb{C}^2,0)$. 

\begin{align}
\label{eqpar}
\sum_{k=1}^\infty \frac{F(k)}{\prod_{i=1}^{k}(1 - Q^iT^{2(i-1)})^2}= \frac{1}{\prod_{i=1}^{\infty}(1 - Q^iT^{2(i-1)})}.
\end{align}
\end{corollary}

\begin{remark}
Setting $T=1$ in the  the above Equation (\ref{eqpar}), we recover the following classical identity of the generating function of integer partitions, where the left hand side is computed using the Durfee square method.

\begin{align}
\label{durfeeclassical}
1+\sum_{k=1}^\infty \frac{Q^{k^2}}{\prod_{i=1}^{k}(1 - Q^i))^2}=\frac{1}{\prod_{i=1}^{\infty}(1 - Q^i)}.
\end{align}
We briefly show the computation of setting $T=1$ as follows. 

The initial vector becomes $$V_1=
\begin{bmatrix}
0 \\ Q^k
\end{bmatrix}.$$

For $s=2, ..., k$, the adjacency matrix becomes $$M_s =
\begin{bmatrix}
(1-Q^{k-s+1})^2 & 0  \\ 
Q^{k-s+1} & Q^{2(k-s+1)}
\end{bmatrix}.$$
So the vectors $V_k$ for $s=2, ..., k$ becomes $$V_k= \begin{bmatrix}
0 \\ Q^k \prod_{s=2}^k Q^{2(k-s+1)}
\end{bmatrix}
= \begin{bmatrix}
0 \\ Q^{k^2} 
\end{bmatrix}.$$

And the polynomial $F(k)$ is therefore $$
F(k)=\begin{cases}
Q^2-Q+1, & \text{if } k=1, \\

Q^{k^2} , & \text{if } k =2,...,v,
\end{cases}
$$
so we indeed have $$\frac{Q^2-Q+1}{(1 - Q))^2}= 1+ \frac{Q^{1}}{(1 - Q))^2}$$ as the initial cases for $k=1$, which matches the left hand side, the Durfee square formula.

This should not be surprising, because by setting $T=1$, we get rid of all the type $\circled{1}$ rows, and only construct the weak diagonal partition with type $\circled{2}$ rows, which is exactly constructing the regular partitions, so we get back the formula for the regular partitions. In this sense, the identity in Corollary \ref{durfee} is a deformation of the classical Durfee square decomposition of the generating function of integer partitions, and the factor $F(k)$ is a deformation of the term $Q^{k^2}$. Combinatorially, the minimal weak diagonal partitions are analogs of the Durfee squares, serving as the "cores" of the weak diagonal partitions.
\end{remark}

\section{Poincar\'e polynomial of $\Hilb^n(\{x^{v-1}y^v=0\},0)$ and $\Hilb^n(\{x^{v-2}y^v=0\},0)$}
\begin{theorem}
Define 
$$
F_{v-1}(k)=\begin{cases}
F(k)
 & \text{if } k \leq v-1, \\

\mathbf{1}^T \cdot (\prod_{s=2}^{k} M_s \cdot W_1),  & \text{if } k = v,
\end{cases}
$$

$$
F_{v-2}(k)=\begin{cases}
F(k)
 & \text{if } k \leq v-2, \\

\mathbf{1}^T \cdot (\prod_{s=3}^{k} M_s \cdot D_1),  & \text{if } k = v-1,\\

\mathbf{1}^T \cdot (\prod_{s=2}^{k} M_s \cdot D_2),  & \text{if } k = v.
\end{cases}
$$

And the vectors $$W_1=
\begin{bmatrix}
0 \\ Q^kT^{k-1}(1-Q^kT^{2(k-1)})
\end{bmatrix}.$$

$$D_1=
\begin{bmatrix}
Q^{2(k-1)}T^{4(k-2)}(T^2-1)(1-Q^{k-1}T^{2(k-2)}) \\
(Q^{k-1}T^{2(k-2)}+Q^{2(k-1)}T^{4(k-2)})(1-Q^{k-1}T^{2(k-2)}) 
\end{bmatrix}, $$

$$D_2=
\begin{bmatrix}
Q^{k+2(k-1)}T^{2(k-1+2(k-2))}(T^2-1)(1-Q^kT^{2(k-1)})(1-Q^{k-1}T^{2(k-2)}) \\ Q^{k+2(k-1)}T^{2(k-1+2(k-2))}(1-Q^kT^{2(k-1)})(1-Q^{k-1}T^{2(k-2)})
\end{bmatrix}.$$

Then the virtual Poincar\'e polynomial of $\Hilb^n(\{x^{v-1}y^v=0\},0)$ is given as 
\begin{align*}
\sum_{n=0}^{\infty}\chi(\Hilb^n(\{x^{v-1}y^v=0\},0))Q^n =\sum_{k=1}^v \frac{F_{v-1}(k)}{\prod_{i=1}^{k}(1 - Q^iT^{2(i-1)})^2}.
\end{align*}

And the virtual Poincar\'e polynomial of $\Hilb^n(\{x^{v-2}y^v=0\},0)$ is \begin{align*}
\sum_{n=0}^{\infty}\chi(\Hilb^n(\{x^{v-2}y^v=0\},0))Q^n = \sum_{k=1}^v \frac{F_{v-2}(k)}{\prod_{i=1}^{k}(1 - Q^iT^{2(i-1)})^2}.
\end{align*}

\end{theorem}

\begin{proof}
\begin{enumerate}
\item 
\textbf{Virtual Poincar\'e polynomial of } $\mathbf{\Hilb^n(\{x^{v-1}y^v=0\},0)}.$
For the Hilbert scheme $\Hilb^n(\{x^{v-1}y^v=0\},0)$, we require that the monomial $x^{v-1}y^{v}$ is contained in any ideal $I \in \Hilb^n(\{x^{v-1}y^v=0\},0)$. So for the innermost layer of the Diagonal filtration, $J^{v-1}$, either contains the monomial $x^{v-1}y^{v-1}$ or not. If $J^{v-1}$ contains the monomial $x^{v-1}y^{v-1}$, then the situation reduces to finding virtual Poincar\'e polynomial of all the strata of $\Hilb^n(\{x^{v-1}y^{v-1}=0\},0)$, and the number of rows of the weak diagonal partition is $v-1$ or less, which we already know what to do from Theorem \ref{thmv}, and the virtual Poincar\'e polynomial of the strata of the corresponding minimal diagonal partitions are given as $F(k)$ for $k \leq v-1$.

When we require that $x^{v-1}y^{v-1} \notin J^{v-1}$, then because $J^{v-1}$ contains $x^{v-1}y^{v}$, the Diagonal valuation $v_{(v-1,v-1)}$ of $J^{v-1}$ must contain the point $(\infty,1)$, so the valuation must be of type $[\circled{2},(w,1)]$ for some $w \geq 1$, with one of the generator being $x^{v-1}y^{v}$, and the other generator $x^{v-1+w}y^{v-1}$.

In this case, when we have $k=v$ rows in total, the initial vector, denoted by $W_1$, is given as $$W_1=
\begin{bmatrix}
0 \\ Q^kT^{k-1}(1-Q^kT^{2(k-1)})
\end{bmatrix},$$ 
because we start with the minimal type $\circled{2}$ row with one box, which contributes to the factor $Q^kT^{k-1}$, and note that because $x^{v-1}y^{v}$ is a fixed generator of the $v-1$th row, it cannot move, so we must multiply by a factor of $(1-Q^kT^{2(k-1)})$ to prevent the freedom of moving it. 
\medskip

\item
\textbf{Virtual Poincar\'e polynomial of } $\mathbf{\Hilb^n(\{x^{v-2}y^v=0\},0)}.$

For the Hilbert scheme $\Hilb^n(\{x^{v-2}y^v=0\},0)$, we require that the monomial $x^{v-2}y^{v}$ is contained in any ideal $I \in \Hilb^n(\{x^{v-2}y^v=0\},0)$, so the monomial $x^{v-1}y^{v}$ is also contained in any ideal $I$. 

Again, for the $v-2$th layer of the Diagonal filtration, $J^{v-2}$, either it contains the monomial $x^{v-2}y^{v-2}$ or not. If $J^{v-2}$ contains the monomial $x^{v-2}y^{v-2}$, then the situation reduces to finding virtual Poincar\'e polynomial of all the strata of $\Hilb^n(\{x^{v-2}y^{v-2}=0\},0)$, and the number of rows of the weak diagonal partition is $v-2$ or less, which we already know what to do from Theorem \ref{thmv}, and the virtual Poincar\'e polynomial of the strata of the corresponding minimal diagonal partitions are given as $F(k)$ for $k \leq v-2$.

When we require that $x^{v-2}y^{v-2} \notin J^{v-2}$, then because $J^{v-2}$ contains $x^{v-2}y^{v}$, the Diagonal valuation $v_{(v-2,v-2)}$ of $J^{v-2}$ must contain the point $(\infty,2)$, so this puts constrains on the valuation we can have. If the valuation is type $[\circled{2},(l_1, l_2)]$, then $l_2 \leq 2$. If the valuation is type $[\circled{1},(d_1, d_2)]$, then $d_2 = 1$.

Additionally, same as the case before, $x^{v-1}y^{v} \in I$ implies that either $x^{v-1}y^{v-1} \in I$, so that the weak diagonal partitions only have $v-1$ rows in total, or that the layer $J^{v-1}$ is $[\circled{2},(w, 1)]$ for $w \geq 1$.

Using the rules of Proposition \ref{prop=}, or equivalently, the definition of a weak diagonal partition, we classify all possible minimal weak diagonal partitions:

\begin{enumerate}
\item Suppose $x^{v-1}y^{v-1} \in I$.
\begin{enumerate}
\item When $v_{(v-2,v-2)}(J^{v-2})=[\circled{2},(l_1, 1)]$ for $l_1 \geq 1$, which minimizes to $[\circled{2},(1, 1)]$. Because the $v-1$th layer does not have any box, so the generators of type $\circled{2}$ of the $v-2$th layer do not have the Tail part, so $x^{v-2}y^{v-1} \in I$, and therefore $x^{v-2}y^{v} \in I$ for all ideals $I$ in this stratum.

We cannot shift the generator on the vertical arm freely, so we need to multiply by a factor $1-Q^{k-1}T^{2(k-2)}$.

The initial vector is $Q^{k-1}T^{2(k-2)}(1-Q^{k-1}T^{2(k-2)})$.
    
\item When $v_{(v-2,v-2)}(J^{v-2})=[\circled{2},(l_1, 2)]$ for $l_1 \geq 1$, which minimizes to $[\circled{2},(1, 2)]$. Because the $v-1$th layer does not have any box, so the generators of type $\circled{2}$ of the $v-2$th layer do not have the Tail part, so $x^{v-2}y^{v} \in I$ for all ideals $I$ in this stratum.

We cannot shift the generator on the vertical arm freely, so we need to multiply by a factor $1-Q^{k-1}T^{2(k-2)}$.

The initial vector is $Q^{2(k-1)}T^{4(k-2)}(1-Q^{k-1}T^{2(k-2)})$.

\item When $v_{(v-2,v-2)}(J^{v-2})=[\circled{1},(l_1, 1)]$ for $l_1 \geq 1$, which minimizes to $[\circled{1},(1, 1)]$. Because the $v-1$th layer does not have any box, so the generators of type $\circled{1}$ of the $v-2$th layer do not have the Tail part, so $x^{v-2}y^{v-2}(ax^{l_1}+by) \in I$, and therefore $y \cdot x^{v-2}y^{v-2}(ax^{l_1}+by) = ax^{v-2+l_1}y^{v-1} +bx^{v-2}y^{v-2+2}  \in I,$
and we know that $ax^{v-2+l_1}y^{v-1} \in I$ because $x^{v-1}y^{v-1} \in I$. So $x^{v-2}y^{v} \in I$ for all ideals $I$ in this stratum.

We cannot shift the generator on the vertical arm freely, so we need to multiply by a factor $1-Q^{k-1}T^{2(k-2)}$.

The initial vector is $Q^{2(k-1)}T^{4(k-2)}(T^2-1)(1-Q^{k-1}T^{2(k-2)})$.

\end{enumerate}
\item Suppose $x^{v-1}y^{v-1} \notin I$.
\begin{enumerate}
\item When $v_{(v-1,v-1)}(J^{v-1})=[\circled{2},(w, 1)]$ and $v_{(v-2,v-2)}(J^{v-2})=[\circled{2},(l_1, 2)]$ such that $w \geq 1$ and $l_1 \geq w+1$ which minimize to $ [\circled{2},(1, 1)][\circled{2},(2, 2)]$. Because we require that $x^{v-2}y^v \in I$, the Tail part of the generator $x^{v-2}y^{v}+...$ should be set to $0$, and because the Tail part has $1$ free variable, the dimension is reduced by $1$.

We cannot shift the generators for both the $v-1$th and $v-2$th layers on the vertical arm freely, so we need to multiply by a factor $(1-Q^kT^{2(k-1)})(1-Q^{k-1}T^{2(k-2)})$.

So the initial vector is $Q^{k+2(k-1)}T^{2(k-1+2(k-2))}(1-Q^kT^{2(k-1)})(1-Q^{k-1}T^{2(k-2)})$.
    
\item When $v_{(v-1,v-1)}(J^{v-1})=[\circled{2},(w, 1)]$ and $v_{(v-2,v-2)}(J^{v-2})=[\circled{1},(l_1, 1)]$ such that $w \geq 1$ and $l_1 \geq w$. 

Because we require that $x^{v-2}y^v \in I$, we must have $l_1 \geq w+1$, and free variables of the generator of $J^{v-2}$ remains as usual. 

This minimizes to 
$[\circled{2},(1, 1)], [\circled{1},(1, 2)]$.

We cannot shift the generators for both the $v-1$th and $v-2$th layers on the vertical arm freely, so we need to multiply by a factor $(1-Q^kT^{2(k-1)})(1-Q^{k-1}T^{2(k-2)})$.

The initial vector is $Q^{k+2(k-1)}T^{2(k-1+2(k-2))}(T^2-1)$.
\end{enumerate}

\end{enumerate}

\end{enumerate}
\end{proof}

\section{Appendix}
\label{app}
Here is a Macaulay 2 computer program that computes the virtual Poincar\'e polynomial of $\Hilb^n(\{x^uy^v=0\},0)$ for $u=1,2,v-2,v-1,v$ and any arbitrary $v$ such that $u, v \geq 1$. Here are the things you can compute with it, with examples:

\begin{enumerate}
\item The function fracForCurve(v) computes the virtual Poincar\'e polynomial of $\Hilb^n(\{x^vy^v=0\},0)$ as a rational function, where the input, the integer $v$, is the multiplicity of the curve $x^vy^v=0$.

For example, try the following command for $x^2y^2=0$, you would get the polynomial in Example \ref{ex1.2} (1):
\begin{verbatim}
fracForCurve(2)
\end{verbatim}

Similarly, the functions fracForCurveVMinus1(v) and fracForCurveVMinus1fracForCurveVMinus2(v) compute the virtual Poincar\'e polynomial of $\Hilb^n(\{x^{v-1}y^v=0\},0)$ and $\Hilb^n(\{x^{v-2}y^v=0\},0)$ as rational functions, where the input, the integer $v$, is the multiplicity of the curve $x^{v-1}y^v=0$ or curve $x^{v-2}y^v=0$ correspondingly.

\item The function seriesForCurve(v,m) computes the virtual Poincar\'e polynomial of $\Hilb^n(\{x^vy^v=0\},0)$ as a power series, where the first entry $v$ is the multiplicity of the curve $x^vy^v=0$, and the second entry $m$ is the maximum degree of the power series where we cut off. (This number $m$ intuitively cuts the power series around, but not exactly at, the number $n$, the total number of points $n$ of the Hilbert scheme $\Hilb^n(\{x^vy^v=0\},0)$, because $n$ is the degree of the variable $Q$, and $m$ is the total degree of both the variable $Q$ and $T$.)

For example, try the following command for $x^3y^3=0$,
\begin{verbatim}
seriesForCurve(3,10)
\end{verbatim}
you would get, up to degree $10$,

$\texttt{FormalSeries} \{Q^{10}+Q^{8}T^{2}+2\,Q^{6}T^{4}+Q^{4}T^{6}+Q^{9}+Q^{7}T^{2}+2\,Q^{5}T^{4}+Q^{8}+Q^{6}T^{2}+2\,Q^{4}T^{4}+Q^{7}+Q^{5}T^{2}+Q^{3}T^{4}+Q^{6}+Q^{4}T^{2}+Q^{5}+Q^{3}T^{2}+Q^{4}+Q^{2}T^{2}+Q^{3}+Q^{2}+Q+1,\,10\}$ 

\medskip

\item The function punctualHilb(r,m) is built for a sanity check. It computes the Poincar\'e polynomial of the Hilbert scheme of $n$ points on the complex plane $\mathbb{C}^2$, $\Hilb^n(\mathbb{C}^2,0)$, as a power series.

Because for $\Hilb^n(\{x^vy^v=0\},0)$, when $v$ is large enough comparing to $n$, say $2v \geq n$, then  
$$\Hilb^n(\{x^vy^v=0\},0) \cong \Hilb^n(\mathbb{C}^2,0).$$

The first input $r$ is the maximum number of rows of the partition that indexes the Bialynicki-Birula cells of $\Hilb^n(\mathbb{C}^2,0)$, and the second entry $m$ is the degree of the power series where we cut off.

For example, try the following command for $\Hilb^n(\{x^{10}y^{10}=0\},0)$
and $\Hilb^n(\mathbb{C}^2,0)$ for approximately $n=20$.
\begin{verbatim}
seriesForCurve(10,20)-punctualHilb(20,20)
\end{verbatim}
We get the output: 

$\texttt{FormalSeries}\left\{0,\,20\right\}$. 

This means that the difference between the two power series is $0$, the two Poincar\'e series of the two Hilbert schemes indeed match.

\end{enumerate}

\begin{verbatim}
loadPackage "FormalGroupLaws";

R = QQ[Q,T];

--- define matrix M_s 
Ms= (k,s)-> (
    M= matrix{{Q^(2*(k-s+1))*T^(4*(k-s))*(T^2-1)+(1-Q^(k-s+1)*T^(2*(k-s)))^2,
    Q^(k-s+1)*T^(2*(k-s))*(T^2-1)},
        {Q^(k-s+1)*T^(2*(k-s)), 
        Q^(2*(k-s+1))*T^(4*(k-s))*T^2}};
    return M;
    );
 


-- compute Vk
Vk = (k) -> (
    V =matrix{{Q^(2*k)*T^(4*(k-1))*(T^2-1)}, {Q^k*T^(2*(k-1))}};
    for s from 2 to k do (
        V = Ms(k,s) * V;
        );
    return V;
    );

-- compute the sum of the Poincar\'e series up to r rows
fracForCurve=(r)->(
    seriesSum := (Q^2*T^2 - Q + 1)/(1-Q)^2;
    if r > 1 then(
       for k from 2 to r do (
           V = Vk(k);
           denom = product(for j from 1 to k list (1 - Q^j*T^(2*(j-1)))^2);
           term = (V_(0,0) + V_(1,0))/ denom;
           seriesSum = seriesSum + term;
           );
    );
    return seriesSum;
    );

seriesForCurve=(r,m)->(
    D= denominator(fracForCurve(r));
    N= numerator(fracForCurve(r));
    ANS1=truncate(series(N,m)*inverse(series(D,m)),m);
    return ANS1;
    );

punctualHilb=(r,m)->(
    P=1;
    for i from 1 to r do (
        P = P / (1-Q^i*T^(2*(i-1))); 
        );
    ANS2=truncate(series(numerator(P),m)*inverse(series(denominator(P),m)),m);
    return ANS2;
    );


-- compute W1
W1 = (k) -> (
    V =matrix{{0}, {Q^k*T^(2*(k-1))*(1-Q^k*T^(2*(k-1)))}};
    for s from 2 to k do (
        V = Ms(k,s) * V;
        );
    return V;
    );

--compute the generating function for curve x^{v-1}y^v=0
fracForCurveVMinus1=(r)->(
    seriesSum := (Q^2*T^2 - Q + 1)/(1-Q)^2;
    if r > 2 then(
       for k from 2 to r-1 do (
           V = Vk(k);
           denom = product(for j from 1 to k list (1 - Q^j*T^(2*(j-1)))^2);
           term = (V_(0,0) + V_(1,0))/ denom;
           seriesSum = seriesSum + term;
           );
       );
   V=W1(r);
   denom = product(for j from 1 to r list (1 - Q^j*T^(2*(j-1)))^2);
   term = (V_(0,0) + V_(1,0))/ denom;
   seriesSum = seriesSum + term;
   return seriesSum;
   );

-- compute D1
D1 = (k) -> (
    V =matrix{{Q^(2*(k-1))*T^(4*(k-2))*(T^2-1)*(1-Q^(k-1)*T^(2*(k-2)))}, 
    {(Q^(k-1)*T^(2*(k-2))+Q^(2*(k-1))*T^(4*(k-2)))*(1-Q^(k-1)*T^(2*(k-2)))}};
    for s from 3 to k do (
        V = Ms(k,s) * V;
        );
    return V;
    );

-- compute D2
D2 = (k) -> (
    V =matrix{{Q^(k+2*(k-1))*T^(2*(k-1+2*(k-2)))*(T^2-1)
    *(1-Q^k*T^(2*(k-1)))*(1-Q^(k-1)*T^(2*(k-2)))}, 
    {Q^(k+2*(k-1))*T^(2*(k-1+2*(k-2)))*(1-Q^k*T^(2*(k-1)))*(1-Q^(k-1)*T^(2*(k-2)))}};
    for s from 3 to k do (
        V = Ms(k,s) * V;
        );
    return V;
    );

--compute the generating function for curve x^{v-2}y^v=0
fracForCurveVMinus2=(r)->(
    seriesSum := (Q^2*T^2 - Q + 1)/(1-Q)^2;
    if r > 2 then(
       for k from 2 to r-2 do (
           V = Vk(k);
           denom = product(for j from 1 to k list (1 - Q^j*T^(2*(j-1)))^2);
           term = (V_(0,0) + V_(1,0))/ denom;
           seriesSum = seriesSum + term;
           );
       );
   V=D1(r);
   denom = product(for j from 1 to r-1 list (1 - Q^j*T^(2*(j-1)))^2);
   term = (V_(0,0) + V_(1,0))/ denom;
   seriesSum = seriesSum + term;
   V=D2(r);
   denom = product(for j from 1 to r list (1 - Q^j*T^(2*(j-1)))^2);
   term = (V_(0,0) + V_(1,0))/ denom;
   seriesSum = seriesSum + term;
   return seriesSum;
   );

\end{verbatim}


\begin{thebibliography}{9}
\bibitem{GS} I.M. Gessel and R.P. Stanley, \emph{‘Algebraic Enumeration’, in Handbook of Combinatorics Vol. 2.} 1995, pp.1021–1061.

\bibitem{ORS}A. Oblomkov, J. Rasmussen, V. Shende. \emph{The Hilbert scheme of a plane curve singularity and the HOMFLY homology of its link. With an appendix by Eugene Gorsky.} Geom. Topol. 22 (2018), no. 2, 645–691.
\bibitem{Mau} D. Maulik. \emph{Stable pairs and the HOMFLY polynomial.} Invent. Math. 204 (2016), no. 3, 787–831.

\bibitem{Cau} S. Cautis. \emph{Remarks on coloured triply graded link invariants.}  Algebr. Geom. Topol. 17.6 (2017), pp. 3811–3836.

\bibitem{GKS}O. Kivinen, E. Gorskiy, J. Simental. \emph{Algebra and geometry of link homology.} 
Bulletin of the London Mathematical Society 55 (2023), no. 2, 537-591.

\bibitem{Bej}
D. Bejleri. \emph{The motivic Hilbert zeta function of a planar n-fold thickening of a smooth curve.} \url{https://math.umd.edu/~dbejleri/non_reduced_hilbert_zeta.pdf}

\bibitem{HM} M. Hogancamp, A. Mellit. \emph{Torus Link Homology.}
ArXiv:1909.00418[math.GT]

\bibitem{BP}M. Brion, E. Peyre. \emph{The virtual Poincar\'e polynomials of homogeneous spaces.} Compositio Mathematica 134, 319–335 (2002).

\bibitem{Pol} D. Pol. \emph{Singularités libres, formes et résidus logarithmiques. Mathématiques générales} [math.GM]. Université d’Angers, 2016. Français. NNT : 2016ANGE0021 . tel-01441450
\bibitem{Con} L. Conners. \emph{Row-Column mirror symmetry for colored torus knot homology.} Selecta Mathematica, 30 (2024), no. 97.

\bibitem{ES} G. Ellingsrud, S.A. Stromme. \emph{On the homology of the Hubert scheme of points in the plane.} Inventiones mathematicae 87 (1987): 343-352. 

\bibitem{LS} G. Lusztig, J. Smelt. \emph{Fixed point varieties on the space of lattices.} Bull. London Math. Soc. 23 (1991), no.
3, 213–218

\bibitem{GMO} E. Gorsky, M. Mazin, A. Oblomkov. \emph{Generic curves and non-coprime Catalans.} Journal of Algebraic Geometry.

\end{thebibliography}
\end{document}